\documentclass[10pt]{amsart}
\usepackage{microtype}
\usepackage{amsmath,amsthm,amsfonts,amssymb,mathtools,mathscinet}
\usepackage[bookmarks=false]{hyperref}
\usepackage{stmaryrd}
\usepackage{amsaddr}
\usepackage{ulem}
\usepackage[all]{xy}
\SelectTips{cm}{}
\usepackage{tikz}
\usetikzlibrary{decorations.markings}
\usetikzlibrary{decorations.pathreplacing}
\usetikzlibrary{arrows,shapes,positioning}
\tikzstyle directed=[postaction={decorate,decoration={markings,
    mark=at position #1 with {\arrow{>}}}}]
\tikzstyle rdirected=[postaction={decorate,decoration={markings,
    mark=at position #1 with {\arrow{<}}}}]
    
    \tikzset{anchorbase/.style={baseline={([yshift=-0.5ex]current bounding box.center)}}}
\usepackage{bbm}

\theoremstyle{plain}
\newtheorem{thm}{Theorem}[section]
\newtheorem{thm-nono}{Theorem}
\newtheorem{prop}[thm]{Proposition}
\newtheorem{cor}[thm]{Corollary}
\newtheorem{lem}[thm]{Lemma}
\newtheorem{conv}[thm]{Convention}
\theoremstyle{definition}
\newtheorem{defi}[thm]{Definition}
\theoremstyle{remark}
\newtheorem{rem}[thm]{Remark}
\newtheorem{exa}[thm]{Example}
\def\PW#1{\textcolor[rgb]{0.00,0.00,1.00}{#1}}%

\begin{document}
%\input{diagrams-deformedfoam.tex}

%%%%% more macros
\def\emph#1{{\sl #1\/}}
\newcommand{\Ru}{\cal{R}_\cal{B}}
\newcommand{\Su}{\cal{S}_\cal{B}}
\newcommand{\Vu}{\cal{V}_\cal{B}}
\newcommand{\Wu}{\cal{W}_\cal{B}}
\newcommand{\Tu}{\cal{T}_\cal{B}}
\newcommand{\LH}[2]{\mathrm{KhR}^{#2}(#1)} 
\newcommand{\LHH}[2]{\mathrm{KhR}^{\infty}_{#2}(#1)} 
\newcommand{\LHHr}[2]{\overline{\mathrm{KhR}}^{\infty}_{#2}(#1)} 
\newcommand{\LHr}[2]{\overline{\mathrm{KhR}}^{\Sigma}_{#2}(#1)} 
\newcommand{\LC}[3]{\left \llbracket #1_{#3} \right \rrbracket ^{#2}}
\newcommand{\LCa}[2]{\left \llbracket #1 \right \rrbracket ^{#2}} 
\newcommand{\taut}{\mathrm{taut}} 
\newcommand{\Tor}{\mathrm{Tor}} 
\newcommand{\Mf}{\mathrm{MF}^\Sigma} 
\newcommand{\Mfu}{\mathrm{MF}} 
\newcommand{\Mfd}{\mathrm{MF'}} 
\newcommand{\UMf}{\mathrm{UMF}} 
\newcommand{\grtaut}{\mathrm{grtaut}} 
\newcommand{\Hom}{{\rm Hom}}
\newcommand{\End}{{\rm End}}
\newcommand{\Homcat}{\Hom} 
\newcommand{\idem}{\mathbbm{1}}
\newcommand{\extp}{\bV^a \C[\xi]}
\newcommand{\NWeb}{N\textbf{Web}}
\newcommand{\Foam}[1]{#1\textbf{Foam}}
\newcommand{\Bimod}{\textbf{Bimod}}
\newcommand{\grFoam}[1]{#1\textbf{grFoam}}
\newcommand{\Dec}{\mathrm{Dec}} 
\newcommand{\zerorep}{V}
\def\mf{\mathfrak}
\def\cal#1{\mathcal{#1}}%
\newcommand{\glnn}[1]{\mf{gl}_{#1}}
\newcommand{\slnn}[1]{\mf{sl}_{#1}}
\newcommand{\bV}{\textstyle{\bigwedge}}
\newcommand{\cat}[1]{\ensuremath{\mbox{\bfseries {\upshape {#1}}}}}
\newcommand{\foam}[1]{#1\cat{Foam}}
\newcommand{\Kom}{\cat{Kom}}
\newcommand{\grVect}{\cat{grVect}}
\newcommand{\Vect}{\cat{Vect}}
\newcommand{\Sym}{\mathrm{Sym}}
\def\la{\langle}
\def\ra{\rangle}
\def\dmod{{\mathrm{-mod}}}   %% finitely-generated modules
\newcommand{\comm}[1]{}

 \newcommand{\tikzcrossing}[7][1]{
\begin{tikzpicture} [scale=.5,fill opacity=0.2]
	%bottom web
	\draw[very thick]  (2,0) to (1,0) to [out=180,in=315] (0.2,0.3);
	\draw[very thick, directed=.65] (-0.2,0.7) to [out=135,in=0] (-1,1) to (-2,1);
	\draw[very thick, directed=.85] (2,1) to (1,1) to [out=180,in=0](-1,0) to (-2,0) ;
	%\node[opacity=1] at (2.25,1) {$_{#2}$};
	%\node[opacity=1] at (2.25,0) {$_{#3}$};
	\node[opacity=1] at (1.8,1.5) {$_{#4}$};
	\node[opacity=1] at (1.8,-0.5) {$_{#5}$};
	\node[opacity=1] at (-1.8,1.5) {$_{#6}$};
	\node[opacity=1] at (-1.8,-0.5) {$_{#7}$};
\end{tikzpicture}
 }

\def\C{{\mathbbm C}}
\def\N{{\mathbbm N}}
\def\R{{\mathbbm R}}
\def\Z{{\mathbbm Z}}
\def\Q{{\mathbbm Q}}
\def\H{{\mathbbm H}}
\def\P{{\mathbbm P}}
\def\P{{\mathbbm P}}
\def\k{{\mathbbm k}}
\def\V{{\mathbbm V}}
\def\W{{\mathbbm W}}
\def\X{{\mathbbm X}}
\def\Y{{\mathbbm Y}}
\def\L{{\mathbbm L}}
\def\M{{\mathbbm M}}
\def\A{{\mathbbm A}}
\def\S{{\mathbbm S}}
\def\T{{\mathbbm T}}
\newcommand{\Bb}{\mathbb{B}}

\author{Paul Wedrich}
\thanks{Present address: Department of Mathematics, Imperial College London, SW7 2AZ London, United Kingdom}
\address{Department of Mathematics, Imperial College London, SW7 2AZ London, United Kingdom\newline 
\indent Max Planck Institute for Mathematics, Vivatsgasse 7, 53111 Bonn, Germany}
\email{p.wedrich@gmail.com}

\thanks{Funding: This work was supported by an EPSRC doctoral training grant at the Centre for Mathematical Sciences, University of Cambridge; by the Leverhulme Research Grant RP2013-K-017 to Dorothy Buck, Imperial College London; and by a Fortbildungsstipendium at the Max Planck Institute for Mathematics.}

\title{Exponential growth of colored HOMFLY-PT homology}
\begin{abstract}
We define reduced colored $\glnn{N}$ link homologies and use deformation spectral sequences to characterize their dependence on color and rank.  We then define reduced colored HOMFLY-PT homologies and prove that they arise as large $N$ limits of $\glnn{N}$ homologies. Together, these results allow proofs of many aspects of the physically conjectured structure of the family of type A link homologies. In particular, we verify a conjecture of Gorsky, Gukov and Sto\v{s}i\'c about the growth of colored HOMFLY-PT homology.

\smallskip
\noindent Keywords: \textbf{Knot homology, Khovanov homology, HOMFLY-PT, deformation, spectral sequence.}
\end{abstract}
\maketitle
\tableofcontents 

\section{Introduction}
It is well-known that the famous \emph{Jones polynomial} is an instance of the large class of  \emph{Reshetikhin--Turaev (RT) invariants} of knots and links \cite{RT}. These are rational functions in one variable $q$ and depend on a datum consisting of an oriented, framed link $\cal{L}$ in $S^3$, a complex semisimple Lie algebra $\mathfrak{g}$ and a labelling (coloring) of the connected components of $\cal{L}$ by representations of $\mathfrak{g}$. The Jones polynomial arises in the case where $\mathfrak{g}=\slnn{2}$ and all link components are labelled by the vector representation.

Given such a large family of invariants, it is natural to inquire into their dependence on the datum of link, Lie algebra and labelling. The best-studied subfamily is the one where $\mathfrak{g} \in \{ \glnn{N} \}_{N\in \N}$ and many important aspects of the relationship between these invariants have been uncovered. For example, for a fixed link and labelling, it is well-known that the $\glnn{N}$ invariants stabilize for large $N$ to a \emph{colored HOMFLY-PT} invariant in two variables $a$ and $q$, which contains generalizations of the Alexander polynomial and from which the $\glnn{N}$ invariants can be recovered as one-variable specializations $a=q^N$. The dependence of HOMFLY-PT invariants on the labelling is governed by recurrence relations \cite{GLL}, which are notoriously difficult to compute and are conjectured to have deep relationships to the A-polynomial, character varieties and knot contact homology, see Garoufalidis \cite{Gar2}, Gukov--Saberi \cite{GSa} and references therein. However, as regards questions like the one about dependence of the invariant under variations of the input link, the RT invariants seem to carry too little structure to allow meaningful answers. 

This deficit has been remedied by the development of \emph{link homology theories} categorifying the RT invariants, starting with Khovanov's categorification of the Jones polynomial \cite{Kho1}. They associate to a datum of link, Lie algebra and labelling a bi-graded vector space (usually as the homology of a graded chain complex), whose graded Euler characteristic recovers the corresponding RT invariant. Additionally, for every smooth cobordism between two links in $\R^3$, one can construct a linear map between the corresponding invariants, and it is expected (and proven for fundamental representations in type A, see Ehrig--Tubbenhauer--Wedrich~\cite{ETW} and references therein) that this construction is functorial and only depends on the isotopy type of the cobordism. This provides a much more satisfactory framework for thinking about the dependence of the invariant on the underlying link than would have been possible for the original RT invariants. Furthermore, some of these purely combinatorial link homologies can be related to gauge-theoretic link invariants, with important consequences, for example a proof that Khovanov homology detects the unknot \cite{KM}, a question which is still open for the Jones polynomial. However, the increase in sophistication in categorified invariants is bought at the price of increased algebraic, combinatorial or geometric complexity in the definition. In fact, only for the simplest cases it is possible to compute link homologies directly from the definition. In parallel, the questions about the dependence of the invariant on the other parts of the datum become correspondingly harder and also more interesting. 

The purpose of this paper is to study the rank and labelling (color) dependence of $\glnn{N}$ link homologies with respect to fundamental (exterior power) representations. Such link homologies have been constructed by Wu \cite{Wu1} and Yonezawa \cite{Yon} and many others\footnote{We refer the reader to the introduction of \cite{MW} for a brief survey of the different approaches to colored $\glnn{N}$ link homologies and an argument for why they are equivalent, which is based on work of Cautis \cite{Cau2}.}, building on the pioneering work of Khovanov--Rozansky \cite{KR1}, but up until now, little has been known about the relationships between members of this family.

\subsection{The large $N$ limit, categorified} For knots and in the \emph{uncolored} case of the labelling by the vector representation of $\glnn{N}$, Rasmussen \cite{Ras1} has categorified the relationship between the HOMFLY-PT polynomial and the $\glnn{N}$ RT polynomials by constructing \emph{specialization spectral sequences} from Khovanov--Rozansky's reduced triply-graded HOMFLY-PT homology \cite{KR2} to their reduced $\glnn{N}$ homologies \cite{KR1}. These spectral sequences degenerate for large $N$ and exhibit the triply-graded homology as a stable limit of the $\glnn{N}$ homologies. In this sense, all uncolored reduced $\glnn{N}$ homologies can be packaged into a single link homology theory together with a family of specialization spectral sequences.

In this paper, we define bi-graded \emph{reduced colored $\glnn{N}$ homologies} and triply-graded \emph{reduced colored HOMFLY-PT homologies} and show that they satisfy analogous relationships, generalizing the results of Rasmussen.

\begin{thm-nono}
\label{thm-1} Let $\cal{L}(i)$ be a labelled link with a marked component $i$ of minimal label. There is a spectral sequence 
\[\LHHr{\cal{L}(i)}{} \quad \rightsquigarrow \quad \overline{\mathrm{KhR}}^{\glnn{N}}(\cal{L}(i))\]
from colored HOMFLY-PT homology to colored $\glnn{N}$ homology, both reduced at the component $i$. See Theorem~\ref{thm-redHss}.
\end{thm-nono}

\begin{thm-nono}
\label{thm-2}.
Let $\cal{K}$ be a labelled knot. For sufficiently large $N$, there are isomorphisms
\begin{equation*}
\bigoplus_{{i+2N j = I}\atop{h-j=J}} \LHHr{\cal{K}}{i,j,h} \;\cong\; \overline{\mathrm{KhR}}_{I,J}^{\glnn{N}}(\cal{K})
\end{equation*}
between a grading-collapsed version of reduced colored HOMFLY-PT homology and reduced colored $\glnn{N}$ homology. See Theorem~\ref{thm-gradcollapse}.
\end{thm-nono}

Theorem~\ref{thm-2} relies on the fact that the reduced colored HOMFLY-PT homology is finite-dimensional for knots, see Proposition~\ref{prop-findim}. On the contrary, the original \emph{unreduced} colored HOMFLY-PT homology of Mackaay--Sto\v{s}i\'c--Vaz \cite{MSV} and Webster--Williamson \cite{WW} is always infinite-dimensional. Nevertheless, we find an analogue of Theorem~\ref{thm-1}, which relates it to unreduced colored $\glnn{N}$ homologies and---more generally---their $\Sigma$-deformed versions, which have been defined by Wu in \cite{Wu2} and further studied by Rose and the author in \cite{RW}.

\begin{thm-nono}\label{thm-3} Let $\cal{L}$ be a labelled link. There is a spectral sequence 
\[\LHH{\cal{L}}{} \quad\rightsquigarrow\quad \LH{\cal{L}}{\Sigma}{}\]
from unreduced colored HOMFLY-PT homology to unreduced, $\Sigma$-deformed colored $\glnn{N}$ homology. See Theorem~\ref{thm-Hss}.
\end{thm-nono}

\subsection{The physical structure}
The existence of reduced colored $\glnn{N}$ and HOMFLY-PT homologies and the theorems describing their relationships, although new, will not come as a surprise to the expert. In fact, they are at the very centre of current interest in the study of link homologies from the perspective of theoretical physics. Since Gukov--Schwarz--Vafa's physical interpretation of Khovanov--Rozansky homology as a space of BPS states \cite{GSV}, there has been significant cross-fertilization between the mathematical and physical sides of this field. A particularly interesting outcome is a package of conjectures about the structure of the family of colored $\glnn{N}$ and HOMFLY-PT homologies, which has been developed over a decade in a series of papers by Dunfield--Gukov--Rasmussen \cite{DGR}, Gukov--Walcher \cite{GW}, Gukov--Sto\v{s}i\'c \cite{GS}, Gorsky--Gukov--Sto\v{s}i\'c \cite{GGS} and Gukov--Nawata--Saberi--Sto\v{s}i\'c--Su\l{}kowski \cite{GNSSS}. 

Thanks to recent advances in understanding the higher representation theory underlying link homologies (in particular \emph{categorical skew Howe duality}), which led to the convergence of several different approaches to link homologies \cite{Cau2, QR, RW, MW}, we are now in a position to prove a significant part of the physically conjectured structure. In the following theorems, $\cal{K}^k$ denotes a knot labelled with the $k^{th}$ exterior power of the vector representation.

\begin{thm-nono}\label{thm-4} 
There is a spectral sequence
\begin{align*}
\LHHr{\cal{K}^k}{}  \quad &\rightsquigarrow \quad  \left(\LHHr{\cal{K}^1}{}\right)^{\otimes k}
\end{align*}
%which preserves the homological grading. 
See Corollary~\ref{cor_expgrw}.
\end{thm-nono}

This implies that the colored HOMFLY-PT homologies of a knot grow at least exponentially in color, as predicted in \cite{GGS}. A related result has been obtained previously by the author for categorical invariants of rational tangles, \cite{Wed1}. The next theorem provides rank-reducing spectral sequences, as suggested in \cite{DGR,GW}. In special cases ($M=k=1$), these have been defined and used before by Lewark--Lobb in \cite{LL} for the construction of new slice genus bounds.

\begin{thm-nono}\label{thm-5} 
For $N\geq M\geq k$ there is a spectral sequence
\begin{align*}
\overline{\mathrm{KhR}}^{\glnn{N}}(\cal{K}^k)  \quad &\rightsquigarrow \quad \overline{\mathrm{KhR}}^{\glnn{M}}(\cal{K}^k)
\end{align*}
%which preserves the homological grading. 
See Corollary~\ref{cor_rankdiff}.
\end{thm-nono}  

We are optimistic that the these and other spectral sequences (Theorem~\ref{thm_reducedss}) will help to understand relations among the many slice genus bounds and concordance homomorphisms obtained from $\glnn{N}$ link homologies (see Lobb \cite{Lo1}, Wu \cite{Wu3} and Lewark--Lobb \cite{LL}) as analogs of Rasmussen's $s$-invariant \cite{Ras2}. Finally, there are color-reducing spectral sequences, analogous to the colored differentials in \cite{GS,GGS}.

\begin{thm-nono}\label{thm-6} 
For $k\geq h$ there is a spectral sequence
\begin{align*}
\LHHr{\cal{K}^k}{}  \quad &\rightsquigarrow \quad \LHHr{\cal{K}^h}{}
\end{align*}
%which preserves the homological grading (up to a shift). 
See Corollary~\ref{cor_coldiff}.
\end{thm-nono}

\subsection{Open problems} While this paper solves a significant part of the physically conjectured structure from \cite{DGR,GW,GS,GGS,GNSSS}, several adventurous questions remain open and deserve further mathematical attention. We finish this introduction with a list of some of these open problems (roughly in increasing order of perceived difficulty) and comment on their relation to the results in the present work.
\\
\newline
\textit{The module structure of colored HOMFLY-PT homology.} In our definition of reduced colored HOMFLY-PT homology we assume that the reduction happens on a link component of minimal label. Related to this technical restriction are questions about the relationship between reduced and unreduced HOMFLY-PT homology, see Remark~\ref{rem-redunred}, and about extensions of colored $\glnn{N}$ and HOMFLY-PT homologies to invariants of pointed links along the lines of Baldwin--Levine--Sarkar \cite{BLS}. Further, if this assumption can be avoided or circumvented, then proofs for some of the conjectures about color-shifting properties on links made in \cite{GNSSS} should be within reach. 
\\
\newline
\textit{Are the spectral sequences actually differentials?} Theorem~2 implies that the spectral sequences in Theorem~5 are degenerate for sufficiently large $M$. In general, it is unknown how fast the spectral sequences from Theorems~1 and 3-6 converge\footnote{Note, however, that since the first arXiv version of this paper appeared, Naisse--Vaz have proposed a strategy for proving that the \textit{Rasmussen spectral sequences}---the spectral sequences from Theorems~1 and 3 in the uncolored case---converge on the second page \cite{NV1}.} and whether the differentials in the spectral sequences between (colored) HOMFLY-PT homologies in Theorems 4-6 preserve the $a$-grading. The latter is because these Theorems rely on the stabilization behaviour from Theorem 5, which requires collapsing the $a$-grading, and on deformation spectral sequences between finite-rank homologies, whose grading behaviour sensitively depends on the chosen deformation parameters \cite{LL}. For conjectures about these properties see \cite{DGR, GS, GGS, GNSSS}.
\\
\newline
\textit{Reduced homologies for non-fundamental colors.} Unreduced colored $\glnn{N}$ homologies have also been defined for labellings by non-fundamental representations via cabling operations by infinite-twist torus links, see Rozansky \cite{Roz} and Cautis \cite{Cau2}. Provided compatible definitions of reduced colored $\glnn{N}$ and HOMFLY-PT homologies for these labellings, many of the results in the present paper can be expected to extend to this setting to produce specialization, rank-reducing and color-reducing spectral sequences.
\\
\newline
\textit{The $q$-symmetry and another grading.} A long-standing conjecture about reduced HOMFLY-PT homologies posits that there is an involution on reduced uncolored HOMFLY-PT homology, which inverts the $q$-grading. On reduced colored HOMFLY-PT homologies, it is expected to swap labellings with their \emph{transposes}, e.g. exterior power with symmetric power representations. This $q$-symmetry, also called \emph{mirror symmetry} in \cite{DGR, GS, GGS, GNSSS}, is further conjectured to intertwine with specialization spectral sequences and to be related to an alternative homological grading---$t_r$ in their notation---effectively making colored HOMFLY-PT homology quadruply-graded. 
\\
\newline
\textit{Lie superalgebra link homologies.} Gorsky--Gukov--Sto\v{s}i\'c \cite{GGS} have suggested that the $q$-symmetry and its related structure might be more naturally interpreted in a hypothetical framework of colored $\mathfrak{gl}_{N|M}$ link homologies, whose $M=0$ cases are related to the known $\glnn{N}$ homologies. Recent work of the author in collaboration with Tubbenhauer and Vaz \cite{TVW} supports this conjecture on the decategorified level by exhibiting the $q$-symmetry of colored HOMFLY-PT polynomials as a consequence of relations between the representation categories of quantum $\mathfrak{gl}_{N|M}$ for various parameters $N,M$. Some progress towards a definition of $\mathfrak{gl}_{N|M}$ link homologies via categorical skew Howe duality has been made by Grant \cite{Gra2}. A string theory setup for such link homologies has been proposed by Mikhaylov--Witten \cite{MiW}.
\\
\newline
\textit{A spectral sequence to knot Floer homology.} The classical Alexander polynomial can be recovered as the $a=1$ specialization of the (uncolored) HOMFLY-PT polynomial. Dunfield--Gukov--Rasmussen \cite{DGR} have conjectured the existence of a corresponding specialization spectral sequence from (uncolored) HOMFLY-PT homology to the \emph{knot Floer homology} developed by Ozsv\'{a}th--Szab\'{o} and Rasmussen\cite{OS, RasTh}. Recent developments related to this conjecture and the present paper include work by Dowlin \cite{Dow} and Baldwin--Levine--Sarkar \cite{BLS}, who match up structural properties of both types of link homologies, and Ellis--Petkova--V{\'e}rtesi \cite{EPV}, whose results can be interpreted as saying that knot Floer homology (and its extension to tangles) is already a $\mathfrak{gl}_{1|1}$ homology.
\\
\newline

\noindent \textbf{Outline of the paper:}
Theorems~1, 2 and 3 are proved in Section~\ref{sec-HH} using $\Z$-graded matrix factorizations and arguments analogous to those in Rasmussen's work on the uncolored case \cite{Ras1}. The keystone Lemma~\ref{lem-singlehor} is proved using a technique inspired by the annular web evaluation algorithm of Queffelec--Rose \cite{QR2}. Theorems~4, 5 and 6 follow via large $N$ stabilization from analogous results for reduced $\glnn{N}$ link homologies. These are proved using the deformation theory for reduced $\glnn{N}$ link homologies, which is developed in Section~\ref{sec-slN}, building on earlier work of the author in collaboration with Rose on the unreduced case \cite{RW} and work of Lewark--Lobb on the uncolored case \cite{LL}. For this we use the $\glnn{N}$ foam 2-categories introduced by Queffelec--Rose in \cite{QR}, building on earlier work of Mackaay--Sto\v{s}i\'c--Vaz \cite{MSV2}. These foam $2$-categories were given a satisfying new combinatorial description by Robert--Wagner~\cite{RWa} after the preprint version of this article appear on the mathematics arXiv.  For the reader's convenience we will recall matrix factorizations and foam 2-categories and explain their use in the construction of link homology theories. However, in the interest of compactness we will blackbox some technical aspects, especially regarding the relationships between the link homologies defined using foams and matrix factorization, for which we refer the reader to the cited literature. 
\newline

\noindent \textbf{Acknowledgements:} It is my pleasure to thank Anna Beliakova, Christian Blanchet, Sergei Gukov, Mikhail Khovanov, Lukas Lewark, Andrew Lobb, Satoshi Nawata, Louis-Hadrien Robert, David Rose, Marko Sto\v{s}i\'c, Daniel Tubbenhauer, Emmanuel Wagner, and Edward Witten for valuable discussions related to this work. Special thanks go to Jake Rasmussen for his advice and comments on the manuscript, and to an anonymous referee, whose comments have significantly improved the exposition. Further, I would like to thank the California Institute of Technology, Columbia University, the Simons Center for Geometry and Physics, the Universit\'e catholique de Louvain and the Universit\"{a}t Z\"{u}rich for their kind hospitality and support while this work was in progress.
\newline

\noindent \textbf{Conventions:} All links and tangles are assumed to be oriented and framed, link and tangle diagrams are assumed to be oriented and they represent oriented, framed links and tangles via the blackboard framing.

\section{Deformed colored $\glnn{N}$ homology}
\label{sec-slN}
Sections~\ref{sec-2-1} and \ref{sec-2-2} recall the essential features of the foam 2-categories of Queffelec--Rose \cite{QR} and their use in the construction of (deformed) colored $\glnn{N}$ Khovanov--Rozansky homologies. In Section~\ref{sec-module} we consider these link homologies as modules over the corresponding unknot invariants. Section~\ref{sec-2-4} extends the deformation theory of $\glnn{N}$ link homologies from \cite{RW} to 2-ended tangles. This allows the definition of corresponding reduced homologies in Section~\ref{sec-reducedN} and deformation spectral sequences in Section~\ref{sec-2-6}, which leads to a proof of Theorem 5.

\subsection{Foams}
\label{sec-2-1}

The monoidal 2-category $\foam{N}$ is a categorification of the category of fundamental representations of $U_q(\glnn{N})$ in its diagrammatic incarnation $\NWeb$. The category $\NWeb$ of $\glnn{N}$ webs is a $\Z[q^{\pm 1}]$-linear monoidal category whose objects are sequences $\mathbf{a} = (a_1,\ldots, a_m)\in \{1,\ldots,N\}^m$ for $m \geq 0$, which correspond to tensor products $\bV^{a_1}_q\C_q^N\otimes \cdots\otimes \bV^{a_m}_q\C_q^N$ of fundamental $U_q(\glnn{N})$-representations. The morphisms in $\NWeb$ are linear combinations of $\glnn{N}$ webs ---leftward oriented, trivalent graphs, with edges labelled by elements of $\{1,\ldots,N\}$, which are generated under planar composition by
\[
\xy
(0,0)*{
\begin{tikzpicture}[scale=.4]
	\draw [very thick,directed=.55] (2.25,0) to (.75,0);
	\draw [very thick,directed=.55] (.75,0) to [out=135,in=0] (-1,.75);
	\draw [very thick,directed=.55] (.75,0) to [out=225,in=0] (-1,-.75);
	\node at (3.25,0) {\tiny $a+b$};
	\node at (-1.5,.75) {\tiny $a$};
	\node at (-1.5,-.75) {\tiny $b$};
\end{tikzpicture}
};
\endxy
\quad , \quad
\xy
(0,0)*{
\begin{tikzpicture}[scale=.4]
	\draw [very thick,rdirected=.55] (-2.25,0) to (-.75,0);
	\draw [very thick,rdirected=.55] (-.75,0) to [out=45,in=180] (1,.75);
	\draw [very thick,rdirected=.55] (-.75,0) to [out=315,in=180] (1,-.75);
	\node at (-3.25,0) {\tiny $a+b$};
	\node at (1.5,.75) {\tiny $a$};
	\node at (1.5,-.75) {\tiny $b$};
\end{tikzpicture}
};
\endxy
\quad , \quad
\xy
(0,0)*{
\begin{tikzpicture}[scale=.4]
	\draw [very thick,directed=.55] (2.25,0) to (0,0);
\node at (-.35,0) {\tiny $a$};
\node at (2.5,0) {\tiny $a$};
\end{tikzpicture}
};
\endxy
\]
which we view as mapping from the sequence determined by the labelled points on the right boundary (read top to bottom) to the one determined by the left. Merge and split vertices correspond to the natural intertwiners $\bV^{a}_q\C_q^N\otimes \bV^{b}_q\C_q^N \to \bV^{a+b}_q\C_q^N$ and $\bV^{a+b}_q\C_q^N \to \bV^{a}_q\C_q^N\otimes \bV^{b}_q\C_q^N$ respectively. The morphisms in $\NWeb$ are subject to the $\Z[q^{\pm 1}]$-linear relations that hold among the corresponding $U_q(\glnn{N})$-intertwiners. We illustrate some of these relations in Section~\ref{sec-app}. More details about $\NWeb$ and its relationship with the representation theory of $U_q(\glnn{N})$ appear in \cite{CKM, TVW}.

The monoidal 2-category $\foam{N}$ has the same objects as $\NWeb$ and its 1-morphisms are formal direct sums of $\glnn{N}$ webs (now not considered up to any relations). The 2-morphisms are matrices of $\C$-linear combinations of $\glnn{N}$ foams -- singular cobordisms between such webs assembled from local pieces
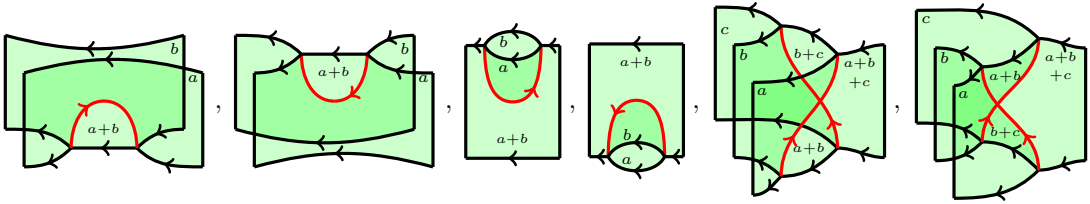
\[
\xy
(0,0)*{
\begin{tikzpicture} [scale=.5,fill opacity=0.2]
	%shading
	\path [fill=green] (4.25,-.5) to (4.25,2) to [out=165,in=15] (-.5,2) to (-.5,-.5) to 
		[out=0,in=225] (.75,0) to [out=90,in=180] (1.625,1.25) to [out=0,in=90] 
			(2.5,0) to [out=315,in=180] (4.25,-.5);
	\path [fill=green] (3.75,.5) to (3.75,3) to [out=195,in=345] (-1,3) to (-1,.5) to 
		[out=0,in=135] (.75,0) to [out=90,in=180] (1.625,1.25) to [out=0,in=90] 
			(2.5,0) to [out=45,in=180] (3.75,.5);
	\path[fill=green] (.75,0) to [out=90,in=180] (1.625,1.25) to [out=0,in=90] (2.5,0);
	%bottom web
	\draw [very thick,directed=.55] (2.5,0) to (.75,0);
	\draw [very thick,directed=.55] (.75,0) to [out=135,in=0] (-1,.5);
	\draw [very thick,directed=.55] (.75,0) to [out=225,in=0] (-.5,-.5);
	\draw [very thick,directed=.55] (3.75,.5) to [out=180,in=45] (2.5,0);
	\draw [very thick,directed=.55] (4.25,-.5) to [out=180,in=315] (2.5,0);
	%seam
	\draw [very thick, red, directed=.75] (.75,0) to [out=90,in=180] (1.625,1.25);
	\draw [very thick, red] (1.625,1.25) to [out=0,in=90] (2.5,0);
	%vertical edges
	\draw [very thick] (3.75,3) to (3.75,.5);
	\draw [very thick] (4.25,2) to (4.25,-.5);
	\draw [very thick] (-1,3) to (-1,.5);
	\draw [very thick] (-.5,2) to (-.5,-.5);
	%top web
	\draw [very thick,directed=.55] (4.25,2) to [out=165,in=15] (-.5,2);
	\draw [very thick, directed=.55] (3.75,3) to [out=195,in=345] (-1,3);
	%labels
	\node [black, opacity=1]  at (1.625,.5) {\tiny{$_{a+b}$}};
	\node[black, opacity=1] at (3.5,2.65) {\tiny{$b$}};
	\node[black, opacity=1] at (4,1.85) {\tiny{$a$}};		
\end{tikzpicture}
};
\endxy
\; , \; 
\xy
(0,0)*{
\begin{tikzpicture} [scale=.5,fill opacity=0.2]
	%shading
	\path [fill=green] (4.25,2) to (4.25,-.5) to [out=165,in=15] (-.5,-.5) to (-.5,2) to
		[out=0,in=225] (.75,2.5) to [out=270,in=180] (1.625,1.25) to [out=0,in=270] 
			(2.5,2.5) to [out=315,in=180] (4.25,2);
	\path [fill=green] (3.75,3) to (3.75,.5) to [out=195,in=345] (-1,.5) to (-1,3) to [out=0,in=135]
		(.75,2.5) to [out=270,in=180] (1.625,1.25) to [out=0,in=270] 
			(2.5,2.5) to [out=45,in=180] (3.75,3);
	\path[fill=green] (2.5,2.5) to [out=270,in=0] (1.625,1.25) to [out=180,in=270] (.75,2.5);
	%bottom web
	\draw [very thick,directed=.55] (4.25,-.5) to [out=165,in=15] (-.5,-.5);
	\draw [very thick, directed=.55] (3.75,.5) to [out=195,in=345] (-1,.5);
	%seam
	\draw [very thick, red, directed=.75] (2.5,2.5) to [out=270,in=0] (1.625,1.25);
	\draw [very thick, red] (1.625,1.25) to [out=180,in=270] (.75,2.5);
	%vertical edges
	\draw [very thick] (3.75,3) to (3.75,.5);
	\draw [very thick] (4.25,2) to (4.25,-.5);
	\draw [very thick] (-1,3) to (-1,.5);
	\draw [very thick] (-.5,2) to (-.5,-.5);
	%top web
	\draw [very thick,directed=.55] (2.5,2.5) to (.75,2.5);
	\draw [very thick,directed=.55] (.75,2.5) to [out=135,in=0] (-1,3);
	\draw [very thick,directed=.55] (.75,2.5) to [out=225,in=0] (-.5,2);
	\draw [very thick,directed=.55] (3.75,3) to [out=180,in=45] (2.5,2.5);
	\draw [very thick,directed=.55] (4.25,2) to [out=180,in=315] (2.5,2.5);
	%labels
	\node [black, opacity=1]  at (1.625,2) {\tiny{$_{a+b}$}};
	\node[black, opacity=1] at (3.5,2.65) {\tiny{$b$}};
	\node[black, opacity=1] at (4,1.85) {\tiny{$a$}};		
\end{tikzpicture}
};
\endxy \;,\;
\xy
(0,0)*{
\begin{tikzpicture} [scale=.5,fill opacity=0.2]
	%back cup 
	\path[fill=green] (-.75,4) to [out=270,in=180] (0,2.5) to [out=0,in=270] (.75,4) .. controls (.5,4.5) and (-.5,4.5) .. (-.75,4);
	%sheet
	\path[fill=green] (-.75,4) to [out=270,in=180] (0,2.5) to [out=0,in=270] (.75,4) -- (1.25,4) -- (1.25,1) -- (-1.25,1) -- (-1.25,4) -- (-.75,4);
	%front cup
	\path[fill=green] (-.75,4) to [out=270,in=180] (0,2.5) to [out=0,in=270] (.75,4) .. controls (.5,3.5) and (-.5,3.5) .. (-.75,4);
	%bottom web
	\draw[very thick, directed=.55] (1.25,1) -- (-1.25,1);
	\path (.75,1) .. controls (.5,.5) and (-.5,.5) .. (-.75,1); %for spacing symmetry
	%seam
	\draw [very thick, red, directed=.7] (-.75,4) to [out=270,in=180] (0,2.5) to [out=0,in=270] (.75,4);
	%vertical edges
	\draw[very thick] (1.25,4) -- (1.25,1);
	\draw[very thick] (-1.25,4) -- (-1.25,1);
	%top web
	\draw[very thick,directed=.65] (1.25,4) -- (.75,4);
	\draw[very thick,directed=.65] (-.75,4) -- (-1.25,4);
	\draw[very thick,directed=.55] (.75,4) .. controls (.5,3.5) and (-.5,3.5) .. (-.75,4);
	\draw[very thick,directed=.55] (.75,4) .. controls (.5,4.5) and (-.5,4.5) .. (-.75,4);
	%labels
	\node [black, opacity=1]  at (0,1.5) {\tiny{$_{a+b}$}};
	\node[black, opacity=1] at (-.25,3.375) {\tiny{$a$}};
	\node[black, opacity=1] at (-.25,4.1) {\tiny{$b$}};	
\end{tikzpicture}
};
\endxy
\; , \;
\xy
(0,0)*{
\begin{tikzpicture} [scale=.5,fill opacity=0.2]
	%back cup 
	\path[fill=green] (-.75,-4) to [out=90,in=180] (0,-2.5) to [out=0,in=90] (.75,-4) .. controls (.5,-4.5) and (-.5,-4.5) .. (-.75,-4);
	%sheet
	\path[fill=green] (-.75,-4) to [out=90,in=180] (0,-2.5) to [out=0,in=90] (.75,-4) -- (1.25,-4) -- (1.25,-1) -- (-1.25,-1) -- (-1.25,-4) -- (-.75,-4);
	%front cup
	\path[fill=green] (-.75,-4) to [out=90,in=180] (0,-2.5) to [out=0,in=90] (.75,-4) .. controls (.5,-3.5) and (-.5,-3.5) .. (-.75,-4);
	%top web
	\draw[very thick, directed=.55] (1.25,-1) -- (-1.25,-1);
	\path (.75,-1) .. controls (.5,-.5) and (-.5,-.5) .. (-.75,-1); %for spacing symmetry
	%seam
	\draw [very thick, red, directed=.7] (.75,-4) to [out=90,in=0] (0,-2.5) to [out=180,in=90] (-.75,-4);
	%vertical edges
	\draw[very thick] (1.25,-4) -- (1.25,-1);
	\draw[very thick] (-1.25,-4) -- (-1.25,-1);
	%bottom web
	\draw[very thick,directed=.65] (1.25,-4) -- (.75,-4);
	\draw[very thick,directed=.65] (-.75,-4) -- (-1.25,-4);
	\draw[very thick,directed=.55] (.75,-4) .. controls (.5,-3.5) and (-.5,-3.5) .. (-.75,-4);
	\draw[very thick,directed=.55] (.75,-4) .. controls (.5,-4.5) and (-.5,-4.5) .. (-.75,-4);
	%labels
	\node [black, opacity=1]  at (0,-1.5) {\tiny{$_{a+b}$}};
	\node[black, opacity=1] at (-.25,-3.4) {\tiny{$b$}};
	\node[black, opacity=1] at (-.25,-4.1) {\tiny{$a$}};
\end{tikzpicture}
};
\endxy
\; , \;
\xy
(0,0)*{
\begin{tikzpicture} [scale=.5,fill opacity=0.2]
	%shading
%	\node[opacity=1] at (.29,1.5) {$+$};
	\path[fill=green] (-2.5,4) to [out=0,in=135] (-.75,3.5) to [out=270,in=90] (.75,.25)
		to [out=135,in=0] (-2.5,1);
	\path[fill=green] (-.75,3.5) to [out=270,in=125] (.29,1.5) to [out=55,in=270] (.75,2.75) 
		to [out=135,in=0] (-.75,3.5);
	\path[fill=green] (-.75,-.5) to [out=90,in=235] (.29,1.5) to [out=315,in=90] (.75,.25) 
		to [out=225,in=0] (-.75,-.5);
	\path[fill=green] (-2,3) to [out=0,in=225] (-.75,3.5) to [out=270,in=125] (.29,1.5)
		to [out=235,in=90] (-.75,-.5) to [out=135,in=0] (-2,0);
	\path[fill=green] (-1.5,2) to [out=0,in=225] (.75,2.75) to [out=270,in=90] (-.75,-.5)
		to [out=225,in=0] (-1.5,-1);
	\path[fill=green] (2,3) to [out=180,in=0] (.75,2.75) to [out=270,in=55] (.29,1.5)
		to [out=305,in=90] (.75,.25) to [out=0,in=180] (2,0);
	%bottom web
	\draw[very thick, directed=.55] (2,0) to [out=180,in=0] (.75,.25);
	\draw[very thick, directed=.55] (.75,.25) to [out=225,in=0] (-.75,-.5);
	\draw[very thick, directed=.55] (.75,.25) to [out=135,in=0] (-2.5,1);
	\draw[very thick, directed=.55] (-.75,-.5) to [out=135,in=0] (-2,0);
	\draw[very thick, directed=.55] (-.75,-.5) to [out=225,in=0] (-1.5,-1);
	%seams
	\draw[very thick, red, rdirected=.85] (-.75,3.5) to [out=270,in=90] (.75,.25);
	\draw[very thick, red, rdirected=.75] (.75,2.75) to [out=270,in=90] (-.75,-.5);	
	%vertical edges
	\draw[very thick] (-1.5,-1) -- (-1.5,2);	
	\draw[very thick] (-2,0) -- (-2,3);
	\draw[very thick] (-2.5,1) -- (-2.5,4);	
	\draw[very thick] (2,3) -- (2,0);
	% top web
	\draw[very thick, directed=.55] (2,3) to [out=180,in=0] (.75,2.75);
	\draw[very thick, directed=.55] (.75,2.75) to [out=135,in=0] (-.75,3.5);
	\draw[very thick, directed=.65] (.75,2.75) to [out=225,in=0] (-1.5,2);
	\draw[very thick, directed=.55]  (-.75,3.5) to [out=225,in=0] (-2,3);
	\draw[very thick, directed=.55]  (-.75,3.5) to [out=135,in=0] (-2.5,4);
	%labels
	\node[black, opacity=1] at (-2.25,3.375) {\tiny$c$};
	\node[black, opacity=1] at (-1.75,2.75) {\tiny$b$};	
	\node[black, opacity=1] at (-1.25,1.75) {\tiny$a$};
	\node[black, opacity=1] at (0,2.75) {\tiny$_{b+c}$};
	\node[black, opacity=1] at (0,.25) {\tiny$_{a+b}$};
	\node[black, opacity=1] at (1.35,2.5) {\tiny$_{a+b}$};	
	\node[black, opacity=1] at (1.35,2) {\tiny$_{+c}$};	
\end{tikzpicture}
};
\endxy
\; , \;
\xy
(0,0)*{
\begin{tikzpicture} [scale=.5,fill opacity=0.2]
	%shading
%	\node[opacity=1] at (-.35,1.5) {$+$};
	\path[fill=green] (-2.5,4) to [out=0,in=135] (.75,3.25) to [out=270,in=90] (-.75,.5)
		 to [out=135,in=0] (-2.5,1);
	\path[fill=green] (-.75,2.5) to [out=270,in=125] (-.35,1.5) to [out=45,in=270] (.75,3.25) 
		to [out=225,in=0] (-.75,2.5);
	\path[fill=green] (-.75,.5) to [out=90,in=235] (-.35,1.5) to [out=315,in=90] (.75,-.25) 
		to [out=135,in=0] (-.75,.5);	
	\path[fill=green] (-2,3) to [out=0,in=135] (-.75,2.5) to [out=270,in=125] (-.35,1.5) 
		to [out=235,in=90] (-.75,.5) to [out=225,in=0] (-2,0);
	\path[fill=green] (-1.5,2) to [out=0,in=225] (-.75,2.5) to [out=270,in=90] (.75,-.25)
		to [out=225,in=0] (-1.5,-1);
	\path[fill=green] (2,3) to [out=180,in=0] (.75,3.25) to [out=270,in=45] (-.35,1.5) 
		to [out=315,in=90] (.75,-.25) to [out=0,in=180] (2,0);				
	%bottom web
	\draw[very thick, directed=.55] (2,0) to [out=180,in=0] (.75,-.25);
	\draw[very thick, directed=.55] (.75,-.25) to [out=135,in=0] (-.75,.5);
	\draw[very thick, directed=.55] (.75,-.25) to [out=225,in=0] (-1.5,-1);
	\draw[very thick, directed=.45]  (-.75,.5) to [out=225,in=0] (-2,0);
	\draw[very thick, directed=.35]  (-.75,.5) to [out=135,in=0] (-2.5,1);	
	%seams
	\draw[very thick, red, rdirected=.75] (-.75,2.5) to [out=270,in=90] (.75,-.25);
	\draw[very thick, red, rdirected=.85] (.75,3.25) to [out=270,in=90] (-.75,.5);
	%vertical edges
	\draw[very thick] (-1.5,-1) -- (-1.5,2);	
	\draw[very thick] (-2,0) -- (-2,3);
	\draw[very thick] (-2.5,1) -- (-2.5,4);	
	\draw[very thick] (2,3) -- (2,0);
	% top web
	\draw[very thick, directed=.55] (2,3) to [out=180,in=0] (.75,3.25);
	\draw[very thick, directed=.55] (.75,3.25) to [out=225,in=0] (-.75,2.5);
	\draw[very thick, directed=.55] (.75,3.25) to [out=135,in=0] (-2.5,4);
	\draw[very thick, directed=.55] (-.75,2.5) to [out=135,in=0] (-2,3);
	\draw[very thick, directed=.55] (-.75,2.5) to [out=225,in=0] (-1.5,2);
	%labels
	\node[black, opacity=1] at (-2.25,3.75) {\tiny$c$};
	\node[black, opacity=1] at (-1.75,2.75) {\tiny$b$};	
	\node[black, opacity=1] at (-1.25,1.75) {\tiny$a$};
	\node[black, opacity=1] at (-.125,2.25) {\tiny$_{a+b}$};
	\node[black, opacity=1] at (-.125,.75) {\tiny$_{b+c}$};
	\node[black, opacity=1] at (1.35,2.75) {\tiny$_{a+b}$};
	\node[black, opacity=1] at (1.35,2.25) {\tiny$_{+c}$};	
\end{tikzpicture}
};
\endxy
\]
modulo certain local relations and isotopies, for which we refer the reader to \cite[Section 3.1]{QR}. 
The local foam relations ensure that the defining linear relations between webs in $\NWeb$---which are \textit{identities} of elements of a $\Z[q^{\pm 1}]$-module---can be lifted to \textit{isomorphisms} between $1$-morphisms in $\foam{N}$.

The facets of these foams carry labellings by elements in $\{1,\ldots,N\}$, and a $k$-labelled facet may also be decorated by 
elements from a ring of symmetric polynomials in $k$ variables. For later use we list a family of foam relations, see \cite[Relation (3.9)]{QR}, which govern \emph{decoration migration} between the three facets adjacent to a \emph{seam} singularity:
\begin{equation}\label{eqn-decmigr}
\xy
(0,0)*{
\begin{tikzpicture} [scale=.5,fill opacity=0.2]
	%shading
	\path[fill=green] (2.25,3) to (.75,3) to (.75,0) to (2.25,0);
	\path[fill=green] (.75,3) to [out=225,in=0] (-.5,2.5) to (-.5,-.5) to [out=0,in=225] (.75,0);
	\path[fill=green] (.75,3) to [out=135,in=0] (-1,3.5) to (-1,.5) to [out=0,in=135] (.75,0);	
	%bottom web
	\draw [very thick,directed=.55] (2.25,0) to (.75,0);
	\draw [very thick,directed=.55] (.75,0) to [out=135,in=0] (-1,.5);
	\draw [very thick,directed=.55] (.75,0) to [out=225,in=0] (-.5,-.5);
	%seam
	\draw[very thick, red, directed=.55] (.75,0) to (.75,3);
	%vertical edges
	\draw [very thick] (2.25,3) to (2.25,0);
	\draw [very thick] (-1,3.5) to (-1,.5);
	\draw [very thick] (-.5,2.5) to (-.5,-.5);
	%top web
	\draw [very thick,directed=.55] (2.25,3) to (.75,3);
	\draw [very thick,directed=.55] (.75,3) to [out=135,in=0] (-1,3.5);
	\draw [very thick,directed=.55] (.75,3) to [out=225,in=0] (-.5,2.5);
	%labels
	\node [black, opacity=1]  at (1.5,2.75) {\tiny{$_{a+b}$}};
	\node[black, opacity=1] at (-.75,3.25) {\tiny{$b$}};
	\node[black, opacity=1] at (-.25,2.25) {\tiny{$a$}};	
	\node[opacity=1] at (1.5,1.5) {\small$\pi_\gamma$};
\end{tikzpicture}
};
\endxy
\quad = \sum_{\alpha,\beta} c_{\alpha,\beta}^\gamma \quad
\xy
(0,0)*{
\begin{tikzpicture} [scale=.5,fill opacity=0.2]
	%shading
	\path[fill=green] (2.25,3) to (.75,3) to (.75,0) to (2.25,0);
	\path[fill=green] (.75,3) to [out=225,in=0] (-.5,2.5) to (-.5,-.5) to [out=0,in=225] (.75,0);
	\path[fill=green] (.75,3) to [out=135,in=0] (-1,3.5) to (-1,.5) to [out=0,in=135] (.75,0);	
	%bottom web
	\draw [very thick,directed=.55] (2.25,0) to (.75,0);
	\draw [very thick,directed=.55] (.75,0) to [out=135,in=0] (-1,.5);
	\draw [very thick,directed=.55] (.75,0) to [out=225,in=0] (-.5,-.5);
	%seam
	\draw[very thick, red, directed=.55] (.75,0) to (.75,3);
	%vertical edges
	\draw [very thick] (2.25,3) to (2.25,0);
	\draw [very thick] (-1,3.5) to (-1,.5);
	\draw [very thick] (-.5,2.5) to (-.5,-.5);
	%top web
	\draw [very thick,directed=.55] (2.25,3) to (.75,3);
	\draw [very thick,directed=.55] (.75,3) to [out=135,in=0] (-1,3.5);
	\draw [very thick,directed=.55] (.75,3) to [out=225,in=0] (-.5,2.5);
	%labels
	\node[black, opacity=1]  at (1.5,2.75) {\tiny{$_{a+b}$}};
	\node[black, opacity=1] at (-.75,3.25) {\tiny{$b$}};
	\node[black, opacity=1] at (-.25,2.25) {\tiny{$a$}};	
	\node[opacity=1] at (.125,1.25) {\small$\pi_\alpha$};
	\node[opacity=1] at (-.25,3) {\small$\pi_\beta$};	
\end{tikzpicture}
};
\endxy
\end{equation} 
Here $\pi_{\alpha}$ denotes the Schur polynomial corresponding to the Young diagram $\alpha$ and the 
$c_{\alpha,\beta}^\gamma$ are the corresponding Littlewood--Richardson coefficients. Imposing this family of relations has the effect of identifying the union of the alphabets of variables on the left side facets with the alphabet on the right side facet. Analogous relations hold across all seams.

The 2-morphisms in $\foam{N}$ admit a grading, which is first defined on undecorated foams as a certain weighted Euler characteristic, see \cite[Definition 3.3]{QR}, to which is added twice the sum of the degrees of the decorating symmetric polynomials in the general case. This grading is additive under disjoint union $\sqcup$, under composition in the 1-morphism direction $\otimes$ and under composition in the 2-morphism direction $\circ$. All relations in $\foam{N}$ are homogeneous, so this grading descends to $\foam{N}$. 

We now introduce versions of foam 2-categories used in the definition of link homologies:
\begin{itemize}

\item $\grFoam{N}$ is the 2-category with the same objects as $\foam{N}$, with 1-morphisms being formal direct sums of webs as in $\foam{N}$, but equipped with an additional formal $\Z$-grading. I.e. corresponding to a web $W$ from $\foam{N}$ we have webs $q^d W$ in $\grFoam{N}$, where $d\in \Z$ indicates the formal degree, and the automorphism shifting the grading by $1$ is given by formal multiplication by $q$. 2-morphisms are matrices of $\C$-linear combinations of foams $F\colon q^{d_1} W_1 \to q^{d_2} W_1$ of grading $d_2-d_1$.\footnote{In \cite{QR} $\grFoam{N}$ is denoted by $\foam{N}$.}

\item $\foam{N}^\bullet$ is the 2-category which is obtained from $\foam{N}$ by imposing the following additional relation on 2-morphisms:
\begin{equation}
\label{eqn_dot1}
\xy
(0,0)*{
\begin{tikzpicture} [fill opacity=0.2,decoration={markings, mark=at position 0.6 with {\arrow{>}}; }, scale=.4]
	%% draw the sheet
	\draw [fill=green] (1,1) -- (-1,2) -- (-1,-1) -- (1,-2) -- cycle;
	%% draw the dot
	\node [opacity=1] at (0,0) {$\bullet^{N}$};
	\node[ opacity=1] at (-.75,1.5) {\tiny{$1$}};
\end{tikzpicture}};
\endxy
 \quad = \quad 0
\end{equation}
This relation is to be understood as saying that whenever a foam contains a $1$-labelled facet decorated by the $N^{th}$ power of the facet variable (indicated by the dot), it is equal to zero. Note that this relation is homogeneous and the morphisms in $\foam{N}^\bullet$ can be equipped with a grading. Thus it also makes sense to define:

\item $\grFoam{N}^\bullet$, the 2-category which is obtained from $\grFoam{N}$ by imposing relation \eqref{eqn_dot1} on 2-morphisms.

\item Let $\Sigma$ be a multiset of $N$ complex numbers. $\foam{N}^\Sigma$ is the 2-category obtained from $\foam{N}$ by imposing the following additional relation on 2-morphisms:
\begin{equation}
\label{eqn_dot2}
\xy
(0,0)*{
\begin{tikzpicture} [fill opacity=0.2,decoration={markings, mark=at position 0.6 with {\arrow{>}}; }, scale=.4]
	%% draw the sheet
	\draw [fill=green] (1,1) -- (-1,2) -- (-1,-1) -- (1,-2) -- cycle;
	%% draw the dot
	\node [opacity=1] at (0,0) {$\bullet^{N}$};
	\node[ opacity=1] at (-.75,1.5) {\tiny{$1$}};
\end{tikzpicture}};
\endxy
 \quad = \quad  \sum_{i=0}^{N-1} (-1)^{N-i-1} e_{N-i}(\Sigma) ~\xy
(0,0)*{
\begin{tikzpicture} [fill opacity=0.2,decoration={markings, mark=at position 0.6 with {\arrow{>}}; }, scale=.4]
	%% draw the sheet
	\draw [fill=green] (1,1) -- (-1,2) -- (-1,-1) -- (1,-2) -- cycle;
	%% draw the dot
	\node [opacity=1] at (0,0) {$\bullet^{i}$};
	\node[ opacity=1] at (-.75,1.5) {\tiny{$1$}};
\end{tikzpicture}};
\endxy
\end{equation}
Here $e_k(\Sigma)$ the $k^{th}$ symmetric polynomial in $N$ variables, evaluated at $\Sigma$. Note that, in general, this relation is not homogeneous with respect to the grading on foams. The only exception is  $\foam{N}^{\{0,\dots, 0\}}=\foam{N}^\bullet$. If we denote by $P\in\C[X]$ the monic polynomial with root multiset $\Sigma$, then \eqref{eqn_dot2} can be abbreviated to $P(\bullet)=0$.
\end{itemize}

\subsection{Link homology}
\label{sec-2-2}
In this section, we recall how foam 2-categories can be used to define $\glnn{N}$ link homologies. For the details, the reader is referred to Queffelec--Rose \cite[Section 4]{QR}. 

Let $\cal{T}$ be a tangle diagram inside an axis-parallel rectangle in $\R^2$, possibly with boundary on the left and right boundary of the rectangle, and with strands labelled by natural numbers. The following construction associates to the labelled tangle diagram $\cal{T}$ a so-called \emph{cube of resolutions} complex $\LC{\cal{T}\hspace{0.5pt}}{}{}$ in the category $\Kom(\foam{N})$ of chain complexes in $\foam{N}$. In \cite[Theorem 4.8]{QR} it is shown that the complex $\LC{\cal{T}\hspace{0.5pt}}{}{}$, considered as an element of the homotopy category $\Kom_h(\foam{N})$, is an invariant of the labelled tangle represented by $\cal{T}$. 

Every labelled tangle admits a diagram given as the horizontal composition $\otimes$ of tangles which are the disjoint union $\sqcup$ of \emph{generating tangles}: caps, cups, crossings and horizontally directed identity tangles. In fact, it suffices to define $\LC{\cdot}{}{}$ on such generating tangles because the operations $\otimes$ and $\sqcup$ on $\foam{N}$ naturally extend to bounded complexes $\Kom^b(\foam{N})$: On chain spaces they act as $\sqcup$ and $\otimes$ respectively and the structure of differentials is defined similarly as in a tensor product complex. Further, it is a standard result that these operations respect homotopies and, hence, are well defined on the homotopy category $\Kom^b_h(\foam{N})$ of bounded complexes in $\foam{N}$.  
 
On 1-strand identity tangles, $\LC{\cdot}{}{}$ is defined as
\begin{equation*}
\label{eqn_reverseEdge}
\left \llbracket \;\;
\xy
(0,0)*{\begin{tikzpicture} [scale=.5]
\draw[very thick, directed=.55] (2,0) to (0,0);
\node at (2.5,0) {$_a$};
\end{tikzpicture}};
\endxy
\right \rrbracket = \xy
(0,0)*{\begin{tikzpicture} [scale=.5]
\draw[very thick, directed=.55] (2,0) to (0,0);
\node at (2.5,0) {$_a$};
\end{tikzpicture}};
\endxy
\quad ,\quad 
\left \llbracket \;\;
\xy
(0,0)*{\begin{tikzpicture} [scale=.5]
\draw[very thick, rdirected=.55] (2,0) to (0,0);
\node at (2.5,0) {$_a$};
\end{tikzpicture}};
\endxy
\right \rrbracket = \xy
(0,0)*{\begin{tikzpicture} [scale=.5]
\draw[very thick, directed=.55] (2,0) to (0,0);
\node at (2.8,0) {$_{N-a}$};
\end{tikzpicture}};
\endxy .
\end{equation*}
Note that we replace a rightward strand of label $a$ by the leftward strand of complementary label $N-a$. This is motivated by the isomorphism of $U_q(\slnn{N})$ representations $(\bV^a_q(\C_q^N))^*\cong\bV^{N-a}_q(\C_q^N)$, and it allows us to exclusively work with leftward oriented webs. The trade-off is that we have the following slightly unnaturally-looking definitions for cap and cup tangles:
\begin{align*}
\left \llbracket
\xy
(0,0)*{\begin{tikzpicture} [scale=.5]
\draw[very thick, directed=.99] (0,1) to [out=180,in=90] (-1.25,.5) to [out=270,in=180] (0,0);
\node at (-1.5,1) {$_a$};
\end{tikzpicture}};
\endxy 
\;\; \right \rrbracket
\;\; = \;\;
\xy
(0,0)*{\begin{tikzpicture} [scale=.5]
 \draw [double] (-1.25,.5) to (-2.25,.5);
 \draw [very thick, directed=.55] (0,0) to [out=180,in=300] (-1.25,.5);
\draw [very thick, directed=.55] (0,1) to [out=180,in=60] (-1.25,.5);
 \node at (.875,0) {$_{N-a}$};
 \node at (.5,1) {$_a$};
 \node at (-2.75,.5) {$_N$}; 
\end{tikzpicture}};
\endxy
\;\; &, \;\;
\left \llbracket \;\;
\xy
(0,0)*{\begin{tikzpicture} [scale=.5]
\draw[very thick, directed=.99] (0,0) to [out=0,in=270] (1.25,.5) to [out=90,in=0] (0,1);
\node at (1.5,1) {$_a$};
\end{tikzpicture}};
\endxy
\right \rrbracket
\;\; = \;\;
\xy
(0,0)*{\begin{tikzpicture} [scale=.5]
 \draw [double] (2.25,.5) -- (1.25,.5);
 \draw [very thick, directed=.55] (1.25,.5) to [out=240,in=0] (0,0);
\draw [very thick, directed=.55] (1.25,.5) to [out=120,in=0] (0,1);
 \node at (-.875,0) {$_{N-a}$};
 \node at (-.5,1) {$_a$};
 \node at (2.75,.5) {$_N$}; 
\end{tikzpicture}};
\endxy \\
\left \llbracket
\xy
(0,0)*{\begin{tikzpicture} [scale=.5]
\draw[very thick, rdirected=.05] (0,1) to [out=180,in=90] (-1.25,.5) to [out=270,in=180] (0,0);
\node at (-1.5,1) {$_a$};
\end{tikzpicture}};
\endxy
\;\; \right \rrbracket
\;\; = \;\;
\xy
(0,0)*{\begin{tikzpicture} [scale=.5]
 \draw [double] (-1.25,.5) to (-2.25,.5);
 \draw [very thick, directed=.55] (0,0) to [out=180,in=300] (-1.25,.5);
\draw [very thick, directed=.55] (0,1) to [out=180,in=60] (-1.25,.5);
 \node at (.875,1) {$_{N-a}$};
 \node at (.5,0) {$_a$};
 \node at (-2.75,.5) {$_N$}; 
\end{tikzpicture}};
\endxy
\;\; &, \;\;
\left \llbracket \;\;
\xy
(0,0)*{\begin{tikzpicture} [scale=.5]
\draw[very thick, rdirected=.05] (0,0) to [out=0,in=270] (1.25,.5) to [out=90,in=0] (0,1);
\node at (1.5,1) {$_a$};
\end{tikzpicture}};
\endxy
\right \rrbracket
\;\; = \;\;
\xy
(0,0)*{\begin{tikzpicture} [scale=.5]
 \draw [double] (2.25,.5) -- (1.25,.5);
 \draw [very thick, directed=.55] (1.25,.5) to [out=240,in=0] (0,0);
\draw [very thick, directed=.55] (1.25,.5) to [out=120,in=0] (0,1);
 \node at (-.875,1) {$_{N-a}$};
 \node at (-.5,0) {$_a$};
 \node at (2.75,.5) {$_N$}; 
\end{tikzpicture}};
\endxy
\end{align*}
The $N$-labelled edges that appear here are drawn as doubled lines---erasing them recovers the tangle.
 
Finally, for labelled, left-directed crossings one uses foam-analogues of Chuang--Rouquier's \cite{CR} Rickard complexes, see \cite[Section 4.2]{QR}. For example, the complex assigned to a negative crossing of strands labelled by $a\geq b$ is
\begin{align}
\label{eqn-crossingcx}
\left \llbracket
\xy (0,0)*{
\tikzcrossing[1]{a}{b}{}{}{}{}
};
(12,-3)*{_b};(12,3)*{_a};
\endxy \right \rrbracket  = \uwave{\xy
(0,0)*{
\begin{tikzpicture} [scale=.5,fill opacity=0.2]
	%bottom web
	\draw[very thick, directed=.55]  (2,0) to (1.25,0);
	\draw[very thick]  (1.25,0) to (0.5,0);
	\draw[very thick, directed=.55]  (-1.25,0) to (-2,0);
	\draw[very thick]  (-0.5,0) to (-1.25,0);
	\draw[very thick, directed=.55]  (2,1) to (1.25,1);
	\draw[very thick, directed=.55]  (-1.25,1) to (-2,1);
	%\draw[very thick] (1.25,1) to (-1.25,1);
	%\draw[very thick, directed=.55]  (0.5,1) to (-0.5,1);
	\draw[very thick, directed=.55]  (0.5,0) to (-0.5,0);
	\draw[very thick, directed=.55] (1.25,1) to [out=180,in=0](0.5,0) ;
	\draw[very thick, directed=.55] (-0.5,0) to [out=180,in=0](-1.25,1) ;
	\node[opacity=1] at (2.25,1) {${_a}$};
	\node[opacity=1] at (2.25,0) {${_b}$};
		\node[opacity=1] at (-2.25,1) {${_b}$};
	\node[opacity=1] at (-2.25,0) {${_a}$};
	\node[opacity=1] at (-1.35,0.45) {\tiny ${b}$};
\end{tikzpicture}
};
\endxy} \to \;\;q\!\!\!
\xy (0,0)*{
\begin{tikzpicture} [scale=.5,fill opacity=0.2]
	%bottom web
	\draw[very thick, directed=.55]  (2,0) to (1.25,0);
	\draw[very thick]  (1.25,0) to (0.5,0);
	\draw[very thick, directed=.55]  (-1.25,0) to (-2,0);
	\draw[very thick]  (-0.5,0) to (-1.25,0);
	\draw[very thick, directed=.55]  (2,1) to (1.25,1);
	\draw[very thick, directed=.55]  (-1.25,1) to (-2,1);
	\draw[very thick] (1.25,1) to (-1.25,1);
	\draw[very thick, directed=.55]  (0.5,1) to (-0.5,1);
	\draw[very thick, directed=.55]  (0.5,0) to (-0.5,0);
	\draw[very thick, directed=.55] (1.25,1) to [out=180,in=0](0.5,0) ;
	\draw[very thick, directed=.55] (-0.5,0) to [out=180,in=0](-1.25,1) ;
	\node[opacity=1] at (2.25,1) {$_{a}$};
	\node[opacity=1] at (2.25,0) {$_{b}$};
	\node[opacity=1] at (-2,0.45) {\tiny $\quad b-1$};
		\node[opacity=1] at (-2.25,1) {$_{b}$};
	\node[opacity=1] at (-2.25,0) {$_{a}$};
\end{tikzpicture}}
\endxy
\to \cdots 
 \to \;\;q^b\xy
(0,0)*{
\begin{tikzpicture} [scale=.5,fill opacity=0.2]
	%bottom web
	\draw[very thick, directed=.55]  (2,0) to (1.25,0);
	\draw[very thick]  (1.25,0) to (0.5,0);
	\draw[very thick, directed=.55]  (-1.25,0) to (-2,0);
	\draw[very thick]  (-0.5,0) to (-1.25,0);
	\draw[very thick, directed=.55]  (2,1) to (1.25,1);
	\draw[very thick, directed=.55]  (-1.25,1) to (-2,1);
	\draw[very thick] (1.25,1) to (-1.25,1);
	\draw[very thick, directed=.55]  (0.5,1) to (-0.5,1);
	\draw[very thick, directed=.55]  (0.5,0) to (-0.5,0);
	%\draw[very thick, directed=.55] (1.25,1) to [out=180,in=0](0.5,0) ;
	\draw[very thick, directed=.55] (1.25,1) to [out=180,in=0](0.5,0) ;\node[opacity=1] at (2.25,1) {${_a}$};
	\node[opacity=1] at (2.25,0) {${_b}$};
		\node[opacity=1] at (-2.25,1) {${_b}$};
	\node[opacity=1] at (-2.25,0) {${_a}$};
	%\node[opacity=1] at (1.35,0.4) {\tiny ${j}$};
\end{tikzpicture}
};
\endxy
\end{align}
where the \uwave{underlined} term is in homological degree zero, every component of the differential is given by a foam containing a single $1$-facet bridging the middle \emph{square} and we include $q$-degree shift that make the differentials grading preserving, allowing the complex to be defined over $\grFoam{N}$. The following illustrates a typical component of the differential:

\[\xy
(0,0)*{
\begin{tikzpicture} [scale=.5,fill opacity=0.1]
%% shading
\path[fill=green, opacity=0.3] (1,.75) to [out=240,in=330] (0,.75) to (-2,2.25) to [out=90,in=180] (0,3.5) to [out=0,in=90] (2,2.25) to (1,0.75) ;
\path[fill=green] (0,0) to (0,3) to (-2,4) to (-2,1) to (0,0);
\path[fill=green] (1,0) to (1,3) to (2,4) to (2,1) to (1,0);
\path[fill=green] (-3,1) to (3,1) to (3,4) to (-3,4);
\path[fill=green] (-2.5,0) to (3.5,0) to (3.5,3) to (-2.5,3);
%% bottom web
	\draw[very thick, directed=.44] (3,1) to (-3,1);
	\draw[very thick, directed=.52] (3.5,0) to (-2.5,0);
	\draw[very thick, directed=.55] (2,1) to  (1,0);
	\draw[very thick, directed=.55] (0,0) to  (-2,1);
%% semi-circle seams	
	\draw[very thick, red, directed=.55] (0,0.75) to [out=330,in=240] (1,0.75);
	\draw[very thick, red, rdirected=.55] (-2,2.25) to [out=90,in=180] (0,3.5) to [out=0,in=90] (2,2.25);
%% seams
\draw[very thick, red, directed=.55] (1,3) to (1,0);
\draw[very thick, red, directed=.85] (-2,4) to (-2,1);
\draw[very thick, red, directed=.55] (0,0) to (0,3);
\draw[very thick, red, directed=.25] (2,1) to (2,4);
\draw[very thick, red, rdirected=.55] (2,2.25) to  (1,0.75);
\draw[very thick, red, rdirected=.55] (0,0.75) to  (-2,2.25);
%% top web
	\draw[very thick, directed=.44] (3,4) to (-3,4);
	\draw[very thick, directed=.52] (3.5,3) to (-2.5,3);
	\draw[very thick, directed=.55] (2,4) to  (1,3);
	\draw[very thick, directed=.55] (0,3) to  (-2,4);
%% middle web 1
	%\draw[dashed] (3,2.25) to (-3,2.25);
	%\draw[dashed] (3,1.25) to (-2,1.25);
%% verticals
\draw[very thick] (-3,4) to (-3,1);
\draw[very thick] (3,4) to (3,1);
\draw[very thick] (-2.5,3) to (-2.5,0);
\draw[very thick] (3.5,3) to (3.5,0);
\node[opacity=1] at (-1.7,.5) {${_k}$};
	\end{tikzpicture}
};
\endxy
\]

These assignments, together with the extension of $\otimes$ and $\sqcup$ to $\Kom^b(\foam{N})$, define $\LC{\cdot}{}{}$ on all tangle diagrams.\footnote{ For technical reasons, one usually assumes that all $N$-labelled boundary points occur at the top. This can be achieved by horizontally composing with webs of the following type: 
$
\xy
(0,0)*{\begin{tikzpicture} [scale=.3]
 \draw [double] (1.25,0) -- (0,0);
\draw [very thick, directed=.55] (1.25,1) -- (-.5,1);
\draw[very thick, directed=.55] (0,0) to (-1.75,0);
\draw[double] (-.5,1) to (-1.75,1);
 % diagonal
 \draw[very thick, directed=.75] (0,0) to (-.5,1); 
 \node at (-2.25,0) {$_a$};
 \node at (-2.25,1) {$_N$}; 
 \node at (1.75,1) {$_a$};
 \node at (1.75,0) {$_N$}; 
\end{tikzpicture}};
\endxy 
\;\;,\;\;
\xy
(0,0)*{\begin{tikzpicture} [scale=.3]
 \draw [double] (3,1) -- (1.75,1);
 \draw[very thick, directed=.55] (3,0) to (1.25,0);
 \draw [double] (1.25,0) -- (0,0);
\draw [very thick, directed=.55] (1.75,1) -- (0,1);
 % diagonal
 \draw [very thick, directed=.75] (1.75,1) -- (1.25,0);
 \node at (3.5,0) {$_a$};
 \node at (3.5,1) {$_N$}; 
  \node at (-0.5,0) {$_N$};
 \node at (-0.5,1) {$_a$}; 
\end{tikzpicture}};
\endxy.$} 
The resulting complex can alternatively be regarded as lying in $\Kom^b(\foam{N})$ or in the graded version $\Kom^b(\grFoam{N})$. We use the following notation for the images of this complex under the natural quotient 2-functors: \begin{itemize}
\item $\LCa{\cal{T}}{\{0,\dots, 0\}}\in \Kom(\foam{N}^\bullet)$ for the singly-graded undeformed version,
\item $\LCa{\cal{T}}{\Sigma}\in \Kom(\foam{N}^\Sigma)$ for the singly-graded $\Sigma$-deformed version and
\item $\LCa{\cal{T}}{\glnn{N}}\in \Kom(\grFoam{N}^\bullet)$ for the bi-graded undeformed version.
\end{itemize}

\begin{thm} \cite[c.f.~Theorem 4.8]{QR}.
\label{thm_tangleinv} Given a labelled tangle diagram $\cal{T}$, the complex $\LC{\cal{T}\hspace{0.5pt}}{}{}$ in $\Kom(\foam{N})$ depends, up to homotopy, only on the labelled tangle represented by $\cal{T}$. The same holds for the versions $\LCa{\cal{T}}{\{0,\dots, 0\}}$, $\LCa{\cal{T}}{\Sigma}$ and $\LCa{\cal{T}}{\glnn{N}}$.
\end{thm}

In the case that the tangle is actually a link, represented by a diagram $\cal{L}$, all of the boundary points in the complex $\LC{\cal{L}}{}{}$ are $N$-labelled and all webs in it are endomorphisms of a \emph{highest weight object} of the form $\mathbf{o}^{top}:=(N,\dots, N)$. Hence, one can apply the representable functor
\begin{equation*}
\grtaut \colon \Kom(\End_{\grFoam{N}^\bullet}(\mathbf{o}^{top})) \to \Kom(\grVect_\C)
\end{equation*} 
given by $\grtaut(\cdot) := \bigoplus_{d\in \Z}\Hom(q^d~1_{\mathbf{o}^{top}}, \cdot )$ to $\LC{\cal{L}}{\glnn{N}}{}$ to obtain a complex of graded vector spaces. Then one defines the bi-graded vector space
\begin{equation}
\label{def_undeformedli}
\LH{\cal{L}}{\glnn{N}}:= H_*(C^{\glnn{N}}(\cal{L})) \quad \text{ where } C^{\glnn{N}}(\cal{L}):=\grtaut(\LC{\cal{L}}{\glnn{N}}{}).
\end{equation}

\begin{thm}~\cite[Theorem 4.12]{QR}
\label{thm_itsKhR}
Up to shifts in homological and \emph{quantum} grading, $\LH{\cal{L}}{\glnn{N}}$ is isomorphic as a bi-graded $\C$-vector space to Wu's and Yonezawa's colored Khovanov--Rozansky homology of the mirror link $\cal{L}^{\#}$.\footnote{In the following, we suppress this difference in conventions and always consider Wu's and Yonezawa's construction applied to the mirror link and with adjusted gradings.}
\end{thm}

In the $\Sigma$-deformed case, we can use the representable functor 
\begin{align*}
\taut \colon \Kom(\End_{\foam{N}^\Sigma}(\mathbf{o}^{top})) \to \Kom(\Vect_\C)
\end{align*} 
given by $\taut(\cdot) := \Hom(1_{\mathbf{o}^{top}}, \cdot )$ to $\LC{\cal{L}}{\Sigma}{}$ to obtain a complex of finite-dimensional vector spaces and we define the graded $\C$-vector space
\begin{equation}
\label{def_deformedli}
\LH{\cal{L}}{\Sigma}:= H_*(C^\Sigma(\cal{L})) \quad \text{ where } C^\Sigma(\cal{L}):=\taut(\LC{\cal{L}}{\Sigma}{}).
\end{equation}

\begin{thm} ~\cite[Theorem 2]{RW}
\label{thm_itsderKhR}
Up to shifts in homological degree, $\LH{\cal{L}}{\Sigma}$ is isomorphic as a singly-graded $\C$-vector space to Wu's deformed Khovanov--Rozansky homology with respect to deformation parameters $\Sigma$.
\end{thm}

\begin{rem} $\LH{\cal{L}}{\glnn{N}}$ is isomorphic to $\LH{\cal{L}}{\{0,\dots,0\}}$ as a singly-graded $\C$-vector space.
\end{rem}

\subsection{Module structure}
\label{sec-module}
In this section, we show that (deformed) colored $\glnn{N}$ link homologies can be regarded as modules over corresponding unknot homologies. For this, let $\cal{T}$ be a labelled 2-ended tangle diagram with boundary label $k$, i.e. a labelled tangle diagram with one right and one left endpoint, which represents a labelled link $\cal{L}$ with a $k$-labelled component cut open.

\begin{equation}
\label{eqn:closure}
\xy
(0,0)*{
\begin{tikzpicture}[scale=.4]
	\draw [very thick,directed=.65] (1.5,.75) to[out=180,in=45]  (.75,0);
	\draw [very thick,directed=.55] (-.75,0) to [out=135,in=0] (-1.5,.75);
	\draw[thick] (1,0) rectangle (-1,-1.5);
	\node at (0,-.75) {$\cal{T}$};
	\node at (1.75,.75) {\tiny $k$};
	\node at (-1.75,.75) {\tiny $k$};
\end{tikzpicture}
};
\endxy\;
\xrightarrow{\textrm{closure}}\;\;  \cal{L}=
\xy
(0,0)*{
\begin{tikzpicture}[scale=.4]
	\draw [very thick,directed=.55] (-.75,0) to [out=135,in=270] (-1,.5) to [out=90,in=180] (0,1) to [out=0,in=90] (1,.5) to [out=270,in=45] (.75,0);
	\draw[thick] (1,0) rectangle (-1,-1.5);
	\node at (0,-.75) {$\cal{T}$};
	\node at (1.25,.75) {\tiny $k$};
\end{tikzpicture}
};
\endxy 
\end{equation}

We write $\mathbf{o}^{cut}:=(N,\dots,N,k)$ for the boundary labels of objects in the complex $\LCa{\cal{T}}{}$ and $H^N_k:= H^*(Gr(k,N),\C)$ for the (graded) cohomology ring (with coefficients in $\C$) of the Grassmannian of complex $k$-planes in $\C^N$. This ring has the following presentation as the quotient of a ring of symmetric polynomials in a $k$-element alphabet $\X$ with coefficients in $\C$
\begin{equation}
\label{eqn_pres1}
H_k^N \cong \frac{\Sym(\X)}{\la h_{N-k+i}(\X)\mid i> 0\ra},
\end{equation}
where $h_i(\X)$ denotes the $i^{th}$ complete symmetric polynomial in $\X$, see \cite{Las}. We write $H(\X)=\sum_{i\geq 0} h_i(\X)t^i = \prod_{x\in X}(1-xt)^{-1}$ for the generating function of the complete symmetric polynomials. We will also need the (ungraded) \emph{$\Sigma$-deformed Grassmannian cohomology ring} $H_k^\Sigma$, given by the presentation
\begin{equation}
\label{eqn_pres2}
H_k^\Sigma := \frac{\Sym(\X)}{\la h_{N-k+i}(\X-\Sigma)\mid i> 0\ra}
\end{equation}
where the complete symmetric polynomials $h_{i}(\X-\Sigma)$ in the difference of the alphabets $\X$ and $\Sigma$ are given by the generating function:
\[H(\X-\Sigma):=H(\X)H(\Sigma)^{-1}=\sum_{i\geq 0}  h_{i}(\X-\Sigma) t^i := \frac{\prod_{\lambda\in \Sigma}(1-\lambda t)}{\prod_{x\in \X}(1-x t)}.\]

\begin{rem} \label{rem-defunknotprop} $H_k^N$ and $H_k^\Sigma$ have (homogeneous) $\C$-bases with elements represented by the Schur polynomials $\pi_\lambda\in \Sym(\X)$ indexed by Young diagrams $\lambda$ that fit into a $k\times (N-k)$ box, see e.g.  \cite[Theorem 2.10]{Wu2}.
\end{rem}

\begin{lem}
\label{lem-endos} There are $\C$-algebra isomorphisms 
\[\End_{\foam{N}^\Sigma}(1_{\mathbf{o}^{cut}})\cong H_k^\Sigma \quad \text{and}\quad \End_{\foam{N}^\Sigma}(1_{\mathbf{o}^{top}})\cong \C\]and graded $\C$-algebra isomorphisms 
\[\End_{\grFoam{N}^\bullet}(1_{\mathbf{o}^{cut}})\cong H^N_k \quad \text{and}\quad \End_{\grFoam{N}^\bullet}(1_{\mathbf{o}^{top}})\cong \C .\]
\end{lem}
\begin{proof}
In \cite[Lemma 23]{RW} it was proven that 
\[\End_{\foam{N}^\Sigma}(1_{(k)})\cong H_k^\Sigma\] where $H_k^\Sigma$ can be interpreted as algebra of decorations on a $k$-labelled foam facet. In particular $\End_{\foam{N}^\Sigma}(1_{(0)}) \cong \C \cong \End_{\foam{N}^\Sigma}(1_{(N)})$. Since $\End_{\foam{N}^\Sigma}(1_{\mathbf{o}^{cut}})$ clearly contains the subalgebra $H_k^\Sigma$ of decorated identity foams, it suffices to prove that these span $\End_{\foam{N}^\Sigma}(1_{\mathbf{o}^{cut}})$. This can be verified as in \cite[Proposition 4.3]{QR}. The undeformed, graded case is analogous.
\end{proof} 

We define the representable functors
\begin{align*}
\grtaut_{cut} \colon \Kom(\End_{\grFoam{N}^\bullet}(\mathbf{o}^{cut})) &\to \Kom(H^N_k\dmod)\\
\taut_{cut} \colon \Kom(\End_{\foam{N}^\Sigma}(\mathbf{o}^{cut})) &\to \Kom(H_k^\Sigma\dmod)
\end{align*} 
 given by $\grtaut_{cut}(\cdot) := \bigoplus_{d\in \Z}\Hom(q^d~1_{\mathbf{o}^{cut}}, \cdot )$ and $\taut_{cut}(\cdot) := \Hom(1_{\mathbf{o}^{cut}}, \cdot )$.

\begin{defi} For 2-ended  tangle diagrams $\cal{T}$ with boundary label $k$, the colored Khovanov--Rozansky homology and its deformed version are defined as
\begin{align*}
\LH{\cal{T}}{\glnn{N}} := \mathrm{H}_*(C^{\glnn{N}}(\cal{T})) \quad &\text{ where } C^{\glnn{N}}(\cal{T}) := \grtaut_{cut}(\LC{\cal{T}\hspace{0.5pt}}{\glnn{N}}{}),\\
\LH{\cal{T}}{\Sigma} := \mathrm{H}_*(C^\Sigma(\cal{T})) \quad &\text{ where } C^\Sigma(\cal{T}) := \taut_{cut}(\LC{\cal{T}\hspace{0.5pt}}{\Sigma}{}).
\end{align*}
\end{defi}

Note that $C^{\glnn{N}}(\cal{T}))$ is a chain complex of graded $H^N_k$-modules and $C^\Sigma(\cal{T}))$ is a chain complex of $H^\Sigma_k$-modules. $\LH{\cal{T}}{\glnn{N}}$ and $\LH{\cal{T}}{\Sigma}$ thus inherit actions of $H^N_k$ and $H^\Sigma_k$ respectively. Before considering the naturality of these actions, we record the following corollary of Theorem~\ref{thm_tangleinv}.
\begin{cor}\label{cor_inv}  $\LH{\cdot}{{\glnn{N}}}$ and $\LH{\cdot}{\Sigma}$, as (bi-)graded vector spaces, are invariant under Reidemeister moves and their tangle analogues and, thus, define invariants of 2-ended tangles.
\end{cor}

In the following, we will show that the Khovanov--Rozansky homologies of 2-ended tangles and their closures are isomorphic. For this, we need a technical lemma.

\begin{lem}\label{lem-straighten} Every web in $\End_{\grFoam{N}^\bullet}(\mathbf{o}^{cut})$ is isomorphic to a direct sum of grading shifted copies of the identity web $1_{\mathbf{o}^{cut}}$. 
\end{lem}
\begin{proof} We first claim that every web $W$ in $\End_{\grFoam{N}^\bullet}(\mathbf{o}^{cut})$---when considered as a 1-term complex in $\Kom_h(\End_{\grFoam{N}^\bullet}(\mathbf{o}^{top}))$---is homotopy equivalent to a chain complex $C(W)$, whose chain groups are direct sums of grading shifted copies of the identity web $1_{\mathbf{o}^{cut}}$.  On the decategorified level, in $\NWeb$, it is well-known that for any such web $W$, there exists a sequence of web relations which can be used to write this web as a linear combination of identity webs $1_{\mathbf{o}^{cut}}$. This can be deduced from the main theorem of Cautis--Kamnitzer--Morrison \cite[Theorem 3.3.1]{CKM}, but we also give an explicit algorithm in Proposition~\ref{prop-annulareval}. Each web relation used in this sequence can be written in a form that lifts to an isomorphism in $\grFoam{N}$, which we would like to use to iteratively simplify $W$. If $V$ is a web appearing at an intermediate stage in the simplification and the next lifted relation is of the form $V\cong \bigoplus_j q^{d_j} V_j$, then we can use this isomorphism to simplify $V$ in $\grFoam{N}$.  However, sometimes the lifted relations are of the form $V\oplus \bigoplus_i q^{d_i} V'_i \cong \bigoplus_j q^{d_j} V_j$. In this case, we use the isomorphism $V \cong \textrm{Cone}\left(\bigoplus_i q^{d_i} V'_i\to \bigoplus_j q^{d_j} V_j\right)$ in the homotopy category instead. 

By Lemma~\ref{lem-endos}, we see that all components of the differential in $C(W)$ are decorated identity foams on $1_{\mathbf{o}^{cut}}$. We may assume that all components have positive degree, since a component of degree zero would be an identity foam that can be cancelled by Gaussian elimination. Then there are no non-trivial homotopies between chain endomorphisms of $C(W)$, as they would have components of negative degree, which is impossible by Lemma~\ref{lem-endos}. This implies that the homotopy equivalence $W \to C(W)$ is in fact an isomorphism of chain complexes, and so $C(W)$ is concentrated in degree zero, and $W$ is isomorphic to its only non-trivial chain group, which proves the lemma.
\end{proof}

\begin{lem}
\label{lem_tangisknot} If $\cal{T}$ is a 2-ended tangle diagram with boundary label $k$ and $\cal{L}$ its closure, then there are isomorphisms $C^{\glnn{N}}(\cal{T}) \cong C^{\glnn{N}}(\cal{L})$ and $C^\Sigma(\cal{T})\cong C^\Sigma(\cal{L})$ of complexes of (graded) vector spaces. Consequently, $\LH{\cal{T}}{\glnn{N}} \cong \LH{\cal{L}}{\glnn{N}}$ and $\LH{\cal{T}}{\Sigma} \cong \LH{\cal{L}}{\Sigma}$ as (bi-)graded vector spaces.
\end{lem}
\begin{proof} Lemma~\ref{lem-straighten} implies that $\LCa{\cal{T}}{\glnn{N}}$ is isomorphic to a complex $C'(\cal{T})$, whose objects are direct sums of grading shifts of $1_{\mathbf{o}^{cut}}$ and whose differentials are matrices of decorated identity foams on this web. After closure as in \eqref{eqn:closure}, this induces an isomorphism between $\LCa{\cal{L}}{\glnn{N}}$ and a complex $C'(\cal{L})$ whose objects are direct sums of grading shifts of bigon webs $B_{k,N-k}$ 
\[1_{\mathbf{o}^{cut}}=
\begin{tikzpicture} [scale=.5,anchorbase]
 \draw [very thick, directed=.55] (1.25,.5) to (-1.25,.5);
 \draw [double] (1.25,-.5) to (-1.25,-.5);
 \draw [double] (1.25,-1.5) to (-1.25,-1.5);
 \node at (-1.75,.5) {$_k$};
 \node at (1.75,.5) {$_k$};
 \node at (-1.75,-.5) {$_N$};
 \node at (1.75,-.5) {$_N$};
 \node at (-1.75,-1.5) {$_N$};
 \node at (1.75,-1.5) {$_N$};
 \node at (1.25,1.25) {$\quad$}; 
  \node at (0,-.8) {$\vdots$}; 
\end{tikzpicture}
\;
\xrightarrow{\textrm{closure}}
\;
\begin{tikzpicture} [scale=.5,anchorbase]
 \draw [double] (-.75,.5) to (-1.25,.5);
 \draw [double] (1.25,-.5) to (-1.25,-.5);
 \draw [double] (1.25,-1.5) to (-1.25,-1.5);
  \draw [double] (.75,.5) to (1.25,.5);
 \draw [very thick, directed=.55] (.75,.5) to [out=240,in=0](0,0) to [out=180,in=300] (-.75,.5);
\draw [very thick, directed=.55] (.75,.5) to [out=120,in=0](0,1) to [out=180,in=60] (-.75,.5);
 \node at (-.75,0) {$_k$};
 \node at (1.25,1) {$_{N-k}$};
 \node at (-1.75,.5) {$_N$};
 \node at (1.75,.5) {$_N$};
 \node at (-1.75,-.5) {$_N$};
 \node at (1.75,-.5) {$_N$};
 \node at (-1.75,-1.5) {$_N$};
 \node at (1.75,-1.5) {$_N$}; 
  \node at (0,-.8) {$\vdots$}; 
\end{tikzpicture}:=
B_{k,N-k}
\]
and whose differential is given by matrices of identity foams decorated on the $k$-labelled facet. Since $B_{k,N-k} \cong 1_{\mathbf{o}^{top}}\otimes H^N_k$ we have a natural isomorphism $\grtaut(B_{k,N-k})\cong \grtaut_{cut}(1_{\mathbf{o}^{cut}})$, which extends to an isomorphism $\grtaut(C'(\cal{L})) \cong \grtaut_{cut}(C'(\cal{T}))$. The deformed case is analogous.
\end{proof}

 In fact, in the proof of Lemma~\ref{lem_tangisknot} we have seen the following.
\begin{cor}
\label{cor-hoffree} $\LH{\cal{L}}{\glnn{N}}$ is the homology of a complex of free $H_k^N$-modules and $\LH{\cal{L}}{\Sigma}$ is the homology of a complex of free $H_k^\Sigma$-modules.
\end{cor}

\begin{thm}
\label{thm-moduleinv}
If $\cal{K}$ is a $k$-labelled knot. Then $C^{\glnn{N}}(\cal{K}) \in \Kom_h(H^N_k\dmod)$, $\LH{\cal{K}}{\glnn{N}}\in H^N_k\dmod$, $C^{\Sigma}(\cal{K}) \in \Kom_h(H_k^\Sigma\dmod)$ and $\LH{\cal{K}}{\Sigma}\in H_k^\Sigma\dmod$ are well-defined knot invariants, in the sense that different representations of $\cal{K}$ as closure of a 2-ended tangle diagram give isomorphic invariants.
\end{thm}
\begin{proof} It is a classical fact that one-component 2-ended tangles are isotopic if and only if they close to isotopic knots, see e.g. \cite[Section 3]{Kho6}. Then it follows from Corollary~\ref{cor_inv} that the module structure is also an invariant.
\end{proof}

Theorem~\ref{thm-moduleinv} also holds for links, although we do not use this fact here. It can be proved via the interpretation of the module structure in Section~\ref{sec-HH} and \cite[Proposition 60]{RW}.

\begin{exa} The unknot homologies are given by $\LH{\bigcirc^k}{\glnn{N}}\cong H^N_k$, $\LH{\bigcirc^k}{\Sigma}\cong H^\Sigma_k$ and, in particular, $\LH{\bigcirc^k}{{\{0,\dots,0\}}}\cong H^{\{0,\dots,0\}}_k$.
\end{exa}

\subsection{Decomposing deformations}
\label{sec-2-4}
Let $\cal{K}$ be a knot given as the closure of a 2-ended tangle diagram $\cal{T}$ and let $\Sigma=\{\lambda_1^{N_1},\dots \lambda_l^{N_l}\}$ be an $N$-element multiset consisting of distinct complex numbers $\lambda_j$ occurring with multiplicities $N_j$. 

\begin{lem}
\label{lem-defunknotprop}\cite[Theorem 13]{RW} There is an isomorphism of $\C$-algebras 
\[H_k^\Sigma \cong \bigoplus_{\sum a_j=k} \bigotimes_{j=1}^l H^{N_j}_{a_j}\] and the direct summands are local algebras that are indexed by $k$-element multisubsets $A=\{\lambda_1^{a_1},\dots \lambda_l^{a_l}\}$ of $\Sigma$. 
\end{lem}

\begin{defi} Let $A$ be a $k$-element multisubset of $\Sigma$ and denote by $1_A$ the corresponding idempotent in $H_k^\Sigma$. Let $\cal{T}$ be a 2-ended tangle diagram with boundary label $k$, then we define \emph{the $A$-localized complex}
\begin{align*}
C^\Sigma_A(\cal{T}) :=  1_A \taut_{cut}(\LC{\cal{T}\hspace{0.5pt}}{\Sigma}{}) 
\end{align*} 
which is a direct summand of $C^\Sigma(\cal{T})$ and clearly itself a knot invariant in $\Kom_h(H_k^\Sigma\dmod)$. Then the \emph{$A$-localized $\Sigma$-deformed Khovanov--Rozansky homology of $\cal{K}$} is defined to be
\[
\mathrm{KhR}^\Sigma_A(\cal{K}):= H_*(C^\Sigma_A(\cal{T})).
\] 
Note that $\mathrm{KhR}^\Sigma(\cal{K}) \cong \bigoplus_{A\subset \Sigma, |A|=k} \mathrm{KhR}^\Sigma_A(\cal{K})$.
\end{defi}

\begin{thm}
\label{thm_localizeddecomp} Let $\cal{K}$ be a knot and write $\cal{K}^k$ for its $k$-labelled version if $k\in \N$. Let $A=\{\lambda_1^{a_1},\dots \lambda_l^{a_l}\}$ be a $k$-element multisubset of $\Sigma$. There is an isomorphism of singly-graded vector spaces:
\begin{equation}
\mathrm{KhR}^\Sigma_A(\cal{K}^k) \cong \bigotimes_{j=1}^l \mathrm{KhR}^{\glnn{N_j}}(\cal{K}^{a_j})
 \end{equation}
\end{thm}
\begin{proof}
The proof is very similar to the proof of~\cite[Theorem 1]{RW}. We sketch it here because we need a variation in the next section, but readers interested in the details are referred to~\cite{RW}.

The proof uses a slightly enlarged 2-category $\hat{\foam{N}{}^{\Sigma}}$, in which idempotent-decorated identity foams are split idempotents---in other words, webs are allowed to be colored by idempotents as well. Then we have that  $C^\Sigma_A(\cal{T}) = 1_A \taut_{cut}(\LC{\cal{T}\hspace{0.5pt}}{\Sigma}{} \cong \taut_{cut}(\LC{\cal{T}}{\Sigma}{A})$, where $\LC{\cal{T}}{\Sigma}{A}$ is the complex in $\hat{\foam{N}{}^{\Sigma}}$ which is associated to the 2-ended tangle diagram $\cal{T}$, with the boundary $k$-labelled web edges colored by the idempotent corresponding to the multisubset $A$.\footnote{From now on we will abbreviate this and say ``colored by $A$'' instead.}

In~\cite[Definition 46 and Theorem 65]{RW} we (that is Rose and the author) define web splitting functors $\phi$, which can be used to split an idempotent-colored complex in $\hat{\foam{N}{}^{\Sigma}}$ into its root-colored components:
\[\phi(\LC{\cal{T}}{\Sigma}{A}) = \bigotimes_{j=1}^l \LC{\cal{T}}{\Sigma}{\lambda_j\in A}\] 
Here we use the shorthand $\LC{\cal{T}}{\Sigma}{\lambda_j\in A}:= \LC{\cal{T}^{a_j}}{\Sigma}{A_j}$ for the complex in $\hat{\foam{N}{}^{\Sigma}}$ assigned to the $a_j$-labelled tangle $\cal{T}$, with the boundary web edges labelled by $A_j:=\{\lambda_j^{a_j}\}$.
Furthermore, there exist representable functors $\taut_{split}$ and $\taut_{\lambda_j}$ such that
\begin{equation}
\label{eqn-tautdecomp}
\taut_{cut}(\LC{\cal{T}}{\Sigma}{A}) \cong \taut_{split}(\phi(\LC{\cal{T}}{\Sigma}{A})) \cong  \bigotimes_{j=1}^l \taut_{\lambda_j}(\LC{\cal{T}}{\Sigma}{\lambda_j\in A}).
\end{equation}
Finally, the complexes $\LC{\cal{T}}{\Sigma}{\lambda_j\in A}$ in the tensor factors on the right-hand side of \eqref{eqn-tautdecomp} live in 2-subcategories $\hat{\foam{N}{}^{\lambda_j\in \Sigma}}$ of $\hat{\foam{N}{}^{\Sigma}}$, which consist of webs and foams that are colored by multisubsets that only contain $\lambda_j$. In~\cite[Proposition 68]{RW} it is shown that $\hat{\foam{N}{}^{\lambda_j\in \Sigma}}$ is isomorphic to the undeformed foam 2-category $\foam{N_j}{}^\bullet$. Moreover, under this isomorphism, the complex $\LC{\cal{T}}{\Sigma}{\lambda_j\in A}$ corresponds to the undeformed complex $\LC{\cal{T}^{a_j}}{\{0,\dots,0\}}{}$ and, hence, $\taut_{\lambda_j}$ applied to $\LC{\cal{T}}{\Sigma}{\lambda_j\in A}$ produces a complex of (ungraded) vector spaces isomorphic to the complex computing $\mathrm{KhR}^{\glnn{N_j}}(\cal{K}^{a_j})$ with the $q$-grading forgotten.
\end{proof}

A version of Theorem~\ref{thm_localizeddecomp} for links with a marked component can be proved analogously, but we will not need it here.

\subsection{Reduced colored $\glnn{N}$ homology}
\label{sec-reducedN}
The purpose of this section is to define reduced colored $\glnn{N}$ homologies and to study their $\Sigma$-deformed versions, which depend on the choice of a $k$-element multiset $A\in\Sigma$. For the sake of brevity, we only consider knots $\cal{K}$, whose $k$-labelled versions we denote by $\cal{K}^k$. However, all definitions and results in this section have straightforward generalizations for links with a marked component.

\begin{defi}Consider the presentation~\eqref{eqn_pres1} of $H_k^N$ and the highest degree element $\pi^N_k:=e_k(\X)^{N-k}$ in $H_k^N$. Then $\la \pi^N_k\ra $ is a $1$-dimensional ideal\footnote{This is proven more generally in Lemma~\ref{lem_piA}} in $H_k^N$ and the \emph{reduced colored $\glnn{N}$ Khovanov--Rozansky homology} of a $k$-labelled knot $\cal{K}^k$ (given as the closure of a $k$-labelled 2-ended tangle $\cal{T}$) is defined as
\[
\overline{\mathrm{KhR}}^{\glnn{N}}(\cal{K}^k):= H_*(\overline{C}^{\glnn{N}}(\cal{T}))\quad \text{ where } \overline{C}^{\glnn{N}}(\cal{T}) := \la\pi^N_k\ra\grtaut_{cut}(\LC{\cal{T}\hspace{0.5pt}}{\glnn{N}}{}).
\]
This produces bi-graded knot invariants, which take the value $\C$ on the unknot.
\end{defi}

\begin{defi} Let $A$ be a  $k$-element multisubset of $\Sigma$. We define the symmetric polynomial
\[\pi_A:= \prod_{\lambda\in \Sigma \setminus A}\prod_{x\in \X} (\lambda-x)\] and denote by the same symbol the corresponding element in $H_k^\Sigma$. 
\end{defi}

\begin{lem}
\label{lem_piA}
$\la \pi_A\ra$ is a $1$-dimensional ideal in $H_k^\Sigma$ and all such are of the form $\la \pi_A\ra$ for some $k$-element multisubset $A$ of $\Sigma$. Moreover, the top degree homogeneous component of $\pi_A$ as a symmetric polynomial in $\X$ is $\pm \pi^N_k$. 
\end{lem}
\begin{proof}
First of all, note that $\pi_A$ has the expansion
\[ \pi_A= \prod_{\lambda\in \Sigma \setminus A}(\lambda^k - \lambda^{k-1}e_1(\X)+\cdots +(-1)^{k-1}\lambda e_{k-1}(\X)+(-1)^{k}e_k(\X)) = \pm e_k(\X)^{N-k} + \text{lower order terms.} \]  
The leading term $\pm e_k(\X)^{N-k}$ of $\pi_A$ is (up to a sign) the Schur polynomial corresponding to the maximal $k\times (N-k)$-box Young diagram. In particular, $\pi_A$ is not zero in $H_k^\Sigma$, see Remark~\ref{rem-defunknotprop}.

Checking 1-dimensionality of $\la \pi_A\ra$ is equivalent to showing that $\pi_A$ is a simultaneous eigenvector for the $e_i(\X)$:
For a multiset of complex numbers $S$ and the variable set $\X$, we consider the ideal $I_S:=\la h_{|S|-|X|+i}(\X-S)|i>0\ra \subset \Sym(\X)$. Via a straightforward computation with generating functions for complete symmetric polynomials in differences of alphabets, we can see that $\prod_{x\in \X} (\lambda-x)I_S \subset I_{S\cup \{\lambda\}}$ for any $\lambda\in \C$. We first note that $\prod_{x\in \X} (\lambda-x)=h_{|\X|}(\lambda-\X)$. We need to check that $h_{|\X|}(\lambda-\X)h_{|S|-|\X|+i}(\X-S)\in I_{S\cup\{\lambda\}}$ for every $i>0$. Writing $S'=S\cup \{\lambda\}$, we note that \[H(S)^{-1}= H(\lambda)^{-1} H(\lambda-\X)H(\X-S)=H(\lambda-\X)H(\X-S')\] is a polynomial of degree $|S|$. Using $h_{k}(\X-S')= h_{k}(\X-S) - \lambda h_{k-1}(\X-S)$ for $k>0$, we get the following identity in degree $|S|+i$.
\begin{align}
\label{eqn:genfctn1}
0 &= h_{|\X|}(\lambda - \X)h_{|S|-|\X|+i}(\X-S) + \sum_{j=0}^{|\X|-1}h_j(\lambda - \X) h_{|S|+i-j}(\X-S') 
\\
\label{eqn:genfctn2}
 &-\lambda h_{|\X|}(\lambda - \X)h_{|S|-|\X|+i-1}(\X-S) +  \sum_{j=|\X|+1}^{|S|+i} h_j(\lambda - \X) h_{|S|+i-j}(\X-S')
\end{align}
The second line \eqref{eqn:genfctn2} can be further rewritten as:
\begin{align*}
%&-\lambda h_{|\X|}(\lambda - \X)h_{|S|-|\X|+i-1}(\X-S) +   \sum_{j=|\X|+1}^{|S|+i} h_j(\lambda - \X) (h_{|S|+i-j}(\X-S) - \lambda h_{|S|+i-j-1}(\X-S))\\
&   \sum_{j=|\X|+1}^{|S|+i} h_j(\lambda - \X) h_{|S|+i-j}(\X-S) - \lambda \sum_{j=|\X|}^{|S|+i-1} h_j(\lambda - \X) h_{|S|+i-j-1}(\X-S))\\
%=&   \sum_{j=|\X|+1}^{|S|+i} h_j(\lambda - \X) h_{|S|+i-j}(\X-S) - \lambda \sum_{j=|\X|+1}^{|S|+i+1} h_{j-1}(\lambda - \X) h_{|S|+i-j}(\X-S))\\
=&   \sum_{j=|\X|+1}^{|S|+i} (h_j(\lambda - \X)-\lambda h_{j-1}(\lambda - \X))  h_{|S|+i-j}(\X-S) =0
\end{align*}
since $h_j(\lambda - \X)=\lambda h_{j-1}(\lambda - \X)$ for $j>|\X|$. Then \eqref{eqn:genfctn1} implies that $h_{|\X|}(\lambda-\X)h_{|S|-|\X|+i}(\X-S)\in I_{S\cup\{\lambda\}}$.

Then, by induction, we also have $\pi_A I_A \subset I_\Sigma$ and since $I_A$ contains all polynomials of the form $e_i(\X)-e_i(A)$, it follows that $\pi_A$ is an  eigenvector with eigenvalue $e_i(A)$ for multiplication by $e_i(\X)$ in $H_k^\Sigma$. Also, it shows that $\pi_A\in 1_A H_k^\Sigma$ because $\pi_A$ is annihilated by every $e_i(\X)-e_i(A)$, at least one of which would need to act as a unit in any other summand $1_B H_k^\Sigma$ by Lemma~\ref{lem_unitcrit}.

Finally, for every $1$-dimensional ideal $I=\C\pi$, the homomorphism $w_I\colon H_k^\Sigma \to \C$ given by $x \pi = w_I(x) \pi$ has as kernel a maximal ideal in $H_k^\Sigma$, of which there are exactly as many as local direct summands of $H_k^\Sigma$, which have been identified as $1_A H_k^\Sigma$. Thus, we have $I=\la \pi_A \ra $ for some $k$-element multisubset $A$ of $\Sigma$.
\end{proof}

\begin{lem}\cite[Corollary 18]{RW}
\label{lem_unitcrit} With respect to the presentation~\eqref{eqn_pres2} of $H_k^\Sigma$, a symmetric polynomial $p\in \Sym(\X)$ projects to a unit in the summand $1_A H_k^\Sigma$ if and only if $p(A)\neq 0$.
\end{lem}

\begin{rem} For a multisubset $A\subset \Sigma$, define $\overline{A}:=\{\lambda\in \Sigma \mid \lambda \text{ appears at least once in } A\}$ and $\overline{\pi}_A:= \prod_{\lambda\in \overline{A} \setminus A}\prod_{x\in \X} (\lambda-x) $. Then $\la \pi_A \ra = \la \overline{\pi}_A 1_A \ra$ in $H_k^\Sigma$. This is easily seen, since the additional factors in the definition of $\pi_A$ are units in $1_A H_k^\Sigma$ by Lemma~\ref{lem_unitcrit}.
\end{rem}

\begin{defi} Let $A$ be a $k$-element multisubset of $\Sigma$. We define the $A$-\emph{reduced} $\Sigma$\emph{-deformed Khovanov--Rozansky homology}
\[
\overline{\mathrm{KhR}}^\Sigma_A(\cal{K}):= H_*(\overline{C}^\Sigma_A(\cal{T}))\quad \text{where } \overline{C}_A^\Sigma(\cal{T}):=\la\pi_{A}\ra \taut_{cut}{\LC{\cal{T}\hspace{0.5pt}}{\Sigma}{}} .
\]
By Theorem~\ref{thm-moduleinv} this is a singly-graded knot invariant. It takes value $\C$ on the unknot.
\end{defi}

In the uncolored case, where $A=\{\lambda\}$, these reduced deformed homologies have been previously defined by Lewark-Lobb \cite{LL}, who also show that they lead to new obstructions to sliceness. 

\begin{thm}
\label{thm_reduceddecomp}  Let $\Sigma=\{\lambda_1^{N_1},\dots \lambda_l^{N_l}\}$ be an $N$-element multiset of distinct complex numbers $\lambda_j$, occurring with multiplicities $N_j$. Let $A=\{\lambda_1^{a_1},\dots \lambda_l^{a_l}\}$ be a $k$-element multisubset of $\Sigma$. Then there is an isomorphism of singly-graded vector spaces:
\begin{equation}
\overline{\mathrm{KhR}}^\Sigma_A(\cal{K}^k) \cong \bigotimes_{j=1}^l \overline{\mathrm{KhR}}^{\glnn{N_j}}(\cal{K}^{a_j})
 \end{equation}
\end{thm}
\begin{proof}

The reduced deformed homology is computed by the chain complex $\la\pi_{A}\ra \taut_{cut}{\LC{\cal{T}\hspace{0.5pt}}{\Sigma}{}}$, which is a chain complex whose chain groups are spanned by foams $F \colon 1_{\mathbf{o}^{cut}} \to W$, which are decorated by $\pi_A$ on the $k$-facet adjacent to the domain web $1_{\mathbf{o}^{cut}}$, and $W$ is a web appearing in the cube of resolutions chain complex $\LC{\cal{T}\hspace{0.5pt}}{\Sigma}{}$. The differentials in $\la\pi_{A}\ra \taut_{cut}{\LC{\cal{T}\hspace{0.5pt}}{\Sigma}{}}$ act by post-composing with foams $G\colon W\to W'$ from $\LC{\cal{T}\hspace{0.5pt}}{\Sigma}{}$. 

Now we apply a splitting functor $\phi$ to all webs in the complex to separate out individual roots. This works exactly as in the proof of Theorem~\ref{thm_localizeddecomp}, except that the decoration by $\pi_A$ on the $k$-facet needs to be distributed onto the $l$ parallel $a_j$-facets colored by $A_{j}:=\{\lambda_j^{a_j}\}$. We claim that if an $A$-colored facet is split into facets colored by multisubsets $A=B\uplus C$ which are disjoint $B\cap C=0$, then $\pi_A$ splits into the product of decorations $u_B \pi_B$ and $u_C \pi_C$ on the $B$- and $C$-colored facets respectively, where $u_B$ and $u_C$ are units. 

To see this, denote the alphabet on the $A$-, $B$- and $C$-colored facets with $\X$, $\X_1$ and $\X_2$ respectively. Then the decoration migration relations~\eqref{eqn-decmigr} impose $\X = \X_1\uplus \X_2$ and we get:
\[ \pi_B \pi_C = \prod_{\lambda\in \Sigma \setminus B}\prod_{x\in \X_1} (\lambda-x) \prod_{\lambda\in \Sigma \setminus C}\prod_{x\in \X_2} (\lambda-x) = \underbrace{\prod_{\lambda\in \Sigma \setminus A}\prod_{x\in \X} (\lambda-x)}_{= \pi_A} \underbrace{\prod_{\lambda\in C}\prod_{x\in \X_1} (\lambda-x)}_{=: u_B^{-1}} \underbrace{\prod_{\lambda\in B}\prod_{x\in \X_2} (\lambda-x)}_{=: u_C^{-1}}\]
Using the criterion from Lemma~\ref{lem_unitcrit} it is clear that the elements which we have suggestively denoted by $u_B^{-1}$ and $u_C^{-1}$ are indeed units. The following equation illustrates the action of the foam splitting functor $\phi$ in a simple case. The green sheet on the left is decorated by $\pi_A$ and colored with the idempotent $1_A$ with $A=\{\lambda_1^{a_1},\lambda_1^{a_2}\}$ containing only two distinct roots. The functor $\phi$ acts by splitting the sheet along the boundary into two parallel sheets colored with the idempotents $1_B$ (blue) and $1_C$ (red), where $B=\{\lambda_1^{a_1}\}$ and $C=\{\lambda_2^{a_2}\}$, c.f. \cite[Example 4.27]{RW}. The decoration $\pi_A$ is split into factors $u_B\pi_B$ and $u_C\pi_C$ as computed above.
\begin{equation}
\phi\left (
\xy
(0,0)*{
\begin{tikzpicture} [scale=.5,fill opacity=0.2]
	%shading	
%% red
	\path[fill=green]  (-2,0.5) to (-2,4.5) to (2,4.5) to (2, 0.5) to (-2,0.5);
%bottom web
	\draw[very thick, directed=.55] (2,0.5) to (-2,0.5);
%vertical edges
	\draw[very thick] (2,0.5) to (2,4.5);
	\draw[very thick] (-2,0.5) to (-2,4.5);
%middle web
	\draw[green, dashed] (-2,2.5) to  (2,2.5);
%top web
	\draw[very thick, directed=.55] (2,4.5) to (-2,4.5);
%labels
	\node[opacity=1] at (0,2.75) {\small $\pi_A$};
\end{tikzpicture}
};
\endxy 
 \right )
  \quad = \quad
\sum_r \xy
(0,0)*{
\begin{tikzpicture} [scale=.5,fill opacity=0.2]
	%shading	
\path[fill=blue]  (2,5) to [out=90,in=270] (2,1) to [out=180,in=0] (0.25,1) to  (0.25,1.75) to [out=0, in=270] (1,2.5) to [out=90, in=0] (0.25,3.25) to [out=180, in=90] (-0.5,2.5) to [out=270,in=180] (0.25,1.75) to (0.25,1)  to (-2,1) to (-2,5) to (2,5);
\path[fill=red]  (2.5,4) to [out=90,in=270] (2.5,0) to [out=180,in=0] (0.25,0) to  (0.25,1.75) to [out=0, in=270] (1,2.5) to [out=90, in=0] (0.25,3.25) to [out=180, in=90] (-0.5,2.5) to [out=270,in=180] (0.25,1.75) to (0.25,0)  to (-1.5,0) to (-1.5,4) to (2.5,4);
\path[fill=green] (0.25,1.75) to [out=0, in=270] (1,2.5) to [out=90, in=0] (0.25,3.25) to [out=180, in=90] (-0.5,2.5) to [out=270,in=180] (0.25,1.75);
% bottom web
	\draw[very thick, directed=.65] (2,1) to [out=180,in=0] (-2,1);
	\draw[very thick, directed=.55] (2.5,0) to [out=180,in=0] (-1.5,0);
%vertical edges
	\draw[very thick] (2,1) to (2,5);
	\draw[very thick] (2.5,0) to (2.5,4);
	\draw[very thick] (-1.5,0) to (-1.5,4);
	\draw[very thick] (-2,1) to (-2,5);	
%middle web
	\draw[dashed, blue] (2,3) to [out=180,in=45] (1,2.5);
	\draw[red, dashed] (2.5,2) to [out=180,in=315] (1,2.5);
	\draw[green, dashed] (1,2.5) to (-.5,2.5);
	\draw[red, dashed] (-.5,2.5) to [out=225,in=0] (-1.5,2);
	\draw[dashed, blue] (-.5,2.5) to [out=135,in=0] (-2,3);
%seams
	\draw[very thick, red, directed=.65] (1,2.5) to [out=270,in=0]  (0.25,1.75) to [out=180, in = 270] (-0.5,2.5) to [out=90, in = 180] (0.25,3.25) to [out=0, in = 90] (1,2.5);
% top web
	\draw[very thick, directed=.65] (2,5) to [out=180,in=0] (-2,5);
	\draw[very thick, directed=.55] (2.5,4) to [out=180,in=0] (-1.5,4);
%labels
	\node[opacity=1] at (-1.55,4.55) {\small $g_r$};
	\node[opacity=1] at (-1.05,0.5) {\small $f_r$};
	\node[opacity=1] at (.25,2.75) {\small $\pi_A$};
\end{tikzpicture}
};
\endxy   
   \quad = \quad
\xy
(0,0)*{
\begin{tikzpicture} [scale=.5,fill opacity=0.2]
	%shading	
\path[fill=blue]  (2,5) to [out=90,in=270] (2,1) to (-2,1) to (-2,5) to (2,5);
\path[fill=red]  (2.5,4) to [out=90,in=270] (2.5,0)  to (-1.5,0) to (-1.5,4) to (2.5,4);
% bottom web
	\draw[very thick, directed=.65] (2,1) to [out=180,in=0] (-2,1);
	\draw[very thick, directed=.55] (2.5,0) to [out=180,in=0] (-1.5,0);
%vertical edges
	\draw[very thick] (2,1) to (2,5);
	\draw[very thick] (2.5,0) to (2.5,4);
	\draw[very thick] (-1.5,0) to (-1.5,4);
	\draw[very thick] (-2,1) to (-2,5);	
%middle web
	\draw[dashed, blue] (2,3) to (-2,3);
	\draw[red, dashed] (2.5,2) to (-1.5,2);
% top web
	\draw[very thick, directed=.65] (2,5) to [out=180,in=0] (-2,5);
	\draw[very thick, directed=.55] (2.5,4) to [out=180,in=0] (-1.5,4);
%labels
		\node[opacity=1] at (0,4.55) {\small $u_B \pi_B$};
	\node[opacity=1] at (0.5,0.5) {\small $u_C \pi_C$};
\end{tikzpicture}
};
\endxy.
\end{equation}

If $A$ contains more than two distinct roots, then an inductive argument allows to split $A$-facets into parallels of $A_j$-facets and to distribute $\pi_A$ into decorations by $\pi_{A_j}$ (up to units). This argument provides the necessary modification of \eqref{eqn-tautdecomp}:
\begin{equation}
\label{eqn-tautdecomp2}
\la \pi_A\ra \taut_{cut}(\LC{\cal{T}}{\Sigma}{A}) \cong  \bigotimes_{j=1}^l \la \pi_{A_j} \ra \taut_{\lambda_j}(\LC{\cal{T}}{\Sigma}{\lambda_j\in A}).
\end{equation}
Next, note that $\la \pi_{A_j} \ra = \la \overline{\pi}_{A_j}\ra$ is the unique maximal ideal in $H_{a_j}^{\Sigma_j}$, where $\Sigma_j = \{\lambda_j^{N_j}\}$. Under the isomorphism of $\hat{\foam{N}{}^{\lambda_j\in \Sigma}}$ with the undeformed foam 2-categories $\foam{N_j}{}^\bullet$, this ideal corresponds to the maximal ideal $\la \pi^{N_j}_{a_j}\ra $ in $H_{a_j}^{N_j}$, which is used in the definition of reduced, undeformed $a_j$-colored $\glnn{N_j}$ homology. Thus, the tensor factors on the right-hand side of \eqref{eqn-tautdecomp2} compute knot homologies isomorphic to the singly-graded versions of $\overline{\mathrm{KhR}}^{\glnn{N_j}}(\cal{K}^{a_j})$.
\end{proof}

\subsection{Deformation spectral sequences}
\label{sec-2-6}
In this section we study Wu's quantum filtration on deformed colored Khovanov--Rozansky homology and its induced deformation spectral sequence \cite{Wu4, Wu2}, and introduce versions for reduced colored Khovanov--Rozansky homology. We again restrict to the case of knots and leave the straightforward generalization to the case of links to the reader.

\begin{defi} A $\C$-vector space $V$ is called \emph{filtered} if it is equipped with a sequence of vector subspaces $\{\cal{F}^p V\}_{p\in \Z}$---the \emph{filtration}---satisfying
\begin{gather*} 
\cdots \subset\cal{F}^p V \subset \cal{F}^{p+1} V \subset  \cal{F}^{p+2} V \subset \cdots\quad ,\quad 
\bigcap_{p\in \Z} \cal{F}^p V = \emptyset\quad ,\quad \bigcup_{p\in \Z} \cal{F}^p V = V.
\end{gather*}
A $\C$-algebra $A$ is \emph{filtered} if its underlying $\C$-vector space is filtered and the multiplication satisfies
\[  \cal{F}^{p} A \cdot \cal{F}^{q} A \subset \cal{F}^{p+q} A \text{ for any } p,q\in \Z.\] 
A module $M$ over a graded, commutative $\C$-algebra $A$ is \emph{filtered} if its underlying $\C$-vector space is filtered and the $A$-action satisfies
\[ x\cdot \cal{F}^{p} M \subset \cal{F}^{p+\mathrm{deg} x} M \text{ for any } x\in A, p\in \Z.\] 
From a filtered $\C$-algebra or an $A$-module $X$ we can construct the \emph{associated graded} object
\[\mathrm{Gr}(X)=\bigoplus_{p\in \Z} {\cal{F}^p X}/{\cal{F}^{p-1} X}\]
which inherits the structure of a graded $\C$-algebra or a graded $A$-module respectively.
We use powers of $q$ to indicate shifts in filtration:
\[\cal{F}^{p} (q^d M):= \cal{F}^{p+d} M.\]
An $A$-module homomorphism $f\colon M\to N$ between filtered $A$-modules is \emph{filtered} if 
\[ f(\cal{F}^{p} M) \subset \cal{F}^{p} N.\]
A chain complex of filtered $A$-modules is \emph{filtered} if the differential is filtered. The homology of such a chain complex is again a filtered $A$-module and its associated graded module can be computed via a spectral sequence induced by the filtration on the chain complex, see e.g. \cite[Theorem 2.6]{McC}.
\end{defi}

\begin{exa}
\label{exa-filt} $H_k^\Sigma$ is a filtered $\C$-algebra and a filtered $\Sym(\X)$-module with respect to the filtration
\[\cal{F}^p H_k^\Sigma := \{x| x \text{ has a representative in } \Sym(\X) \text{ of degree} \leq p\} \]
defined via the presentation~\eqref{eqn_pres2}. In fact, $H_k^\Sigma \cong \End_{\foam{N}^\Sigma}(1_{\mathbf{o}^{cut}})$ is also a filtered module over the larger algebra $R:=\Sym(\X|\X_1|\cdots|\X_s)$ of polynomials separately symmetric in the $k$-element alphabet $\X$ and in $N$-element alphabets $\X_i$, one for each occurrence of $N$ in $\mathbf{o}^{cut}$. The action can be interpreted as placing additional decorations on (decorated) identity foams on $1_{\mathbf{o}^{cut}}$.
\end{exa}

\begin{lem}
\label{lem-filt}
The associated graded algebra of $H_k^\Sigma$  with respect to the filtration introduced in Example~\ref{exa-filt} is $\mathrm{Gr}(H_k^\Sigma)\cong H_k^N$.
\end{lem}
\begin{proof}
Consider the algebra $\Sym(\X)$ and let $I_0:=\la h_{N-k+i}(\X)|i>0\ra$ and $I_\Sigma:=\la h_{N-k+i}(\X-\Sigma)|i>0\ra$ be the ideals used in the presentations \eqref{eqn_pres1} and \eqref{eqn_pres2} of $H_k^N$ and $H_k^\Sigma$ respectively. $I_0$ is a homogeneous ideal, the quotient algebra $H_k^N$ is graded, and so it is canonically isomorphic to its own associated graded $H_k^N\cong \mathrm{Gr}(H_k^N)$. We shall show $\mathrm{Gr}(H_k^N) \cong \mathrm{Gr}(H_k^\Sigma)$. 

Let $I$ be any ideal of $\Sym(\X)$, then an elementary computation identifies the degree $p$ part of the associated graded of $\Sym(\X)/I$ as:
\begin{equation} 
\label{eqn-gradedpiece}
\mathrm{Gr}(\Sym(\X)/I)_p \cong \frac{\cal{F}^{p}\Sym(\X)}{\cal{F}^{p-1}\Sym(\X)+ \cal{F}^{p}\Sym(\X)\cap I}
\end{equation}
Now it is sufficient to note that in the case of the ideals $I_0$ and $I_\Sigma$, the corresponding denominators are equal. To see this, recall that $\Sym(\X)$ has two triangularly equivalent bases $\{\pi_\lambda(\X)=\mathrm{det}(h_{\lambda_i-i+j}(\X))\}$ and $\pi_\lambda(\X-\Sigma)=\mathrm{det}(h_{\lambda_i-i+j}(\X-\Sigma))$ of Schur polynomials, which are indexed by Young diagrams $\lambda=(\lambda_1\geq \lambda_2 \geq \dots)$ with at most $k$ rows. Moreover, by \cite[Theorem 2.10]{Wu2}, we have that $I_0$ and $I_\Sigma$ are spanned by those $\pi_\lambda(\X)$ and $\pi_\lambda(\X-\Sigma)$ respectively, whose Young diagrams $\lambda$ do not fit inside the $k\times (N-k)$ box. Since the top degree part of $\pi_\lambda(\X-\Sigma)$ is precisely $\pi_\lambda(\X)$, it follows that $\cal{F}^{p-1}\Sym(\X)+ \cal{F}^{p}\Sym(\X)\cap I_0 =\cal{F}^{p-1}\Sym(\X)+ \cal{F}^{p}\Sym(\X)\cap I_\Sigma$. 

Then \eqref{eqn-gradedpiece} shows that the graded parts of $\mathrm{Gr}(H_k^N)$ and $\mathrm{Gr}(H_k^\Sigma)$ are canonically isomorphic and that the algebra structures agree since they are both inherited from $\Sym(\X)$.
\end{proof}

The existence of deformation spectral sequences for colored Khovanov--Rozansky homology has been established via Theorems~\ref{thm_filt} and~\ref{thm_ss}, which are due to Wu. Their basis is the observation that the chain complexes computing colored Khovanov--Rozansky homology have chain groups that are graded, but the differentials are only filtered. The filtration thus induced by the grading is known as the \textit{quantum filtration}.

\begin{thm}\label{thm_filt} \cite[Theorem 9.3]{Wu2} Let $\cal{K}$ be a labelled knot diagram. Then the filtered homotopy type of $C^{\Sigma}(\cal{K})$ with respect to the quantum filtration is a knot invariant. 
\end{thm}

\begin{thm}\label{thm_ss} \cite[Theorem 9.9]{Wu2} Let $\cal{K}$ be a labelled knot and $\Sigma$ an $N$-element multiset of complex numbers. Then the \emph{deformation spectral sequence} induced by the quantum filtration on $C^{\Sigma}(\cal{K})$ satisfies:
\begin{itemize}
\item the $E_0$ page is isomorphic to the undeformed bi-graded Khovanov--Rozansky complex $C^{\glnn{N}}(\cal{K})$
\item the $E_1$ page is isomorphic to the undeformed bi-graded Khovanov--Rozansky homology $\LH{\cal{K}}{\glnn{N}}$
\item the $E_\infty$ page is isomorphic as a bi-graded $\C$-vector space to the associated graded of the $\Sigma$-deformed Khovanov--Rozansky homology: 
\[E_\infty \cong \mathrm{Gr}\left( \LH{\cal{K}}{\Sigma} \right )\]
and as a singly-graded $\C$-vector space, it is isomorphic to $\LH{\cal{K}}{\Sigma}$ itself.
\end{itemize} 
\end{thm}

\begin{conv}
Given a spectral sequence as in Theorem~\ref{thm_ss}, we abuse notation and say there is a spectral sequence:
$\mathrm{KhR}^{\glnn{N}}(\cal{K}) \quad \rightsquigarrow \quad \mathrm{KhR}^\Sigma(\cal{K})$.
\end{conv}

We will now give an explicit description of the quantum filtration in the framework of foams and a sketch proof of Theorem~\ref{thm_ss}.

\begin{proof}[Sketch proof of Theorem~\ref{thm_ss}]
We work with a 2-ended tangle diagram $\cal{T}$ whose closure is $\cal{K}$. This has the advantage of allowing an analogous  proof of Theorem~\ref{thm_reducedss}.

The original cube of resolutions complex $\LC{\cal{T}\hspace{0.5pt}}{}{}$ can be regarded as living in $\grFoam{N}$, i.e. it is $q$-graded with grading preserving differentials. Analogously as in the proof of Lemma~\ref{lem_tangisknot}, this complex can be simplified to a (graded) homotopy equivalent complex, which consists only of identity webs $1_{\mathbf{o}^{cut}}$, and decorated, grading preserving identity foams between them\footnote{The main difference here is that $N$-labelled foam facets can carry decorations, which can be regarded as ``equivariant'' parameters in the sense of \cite{Kra} and \cite{Wu2}.}---by abuse of notation we also denote this by $\LC{\cal{T}\hspace{0.5pt}}{}{}$. Next we consider the corresponding complex in $\foam{N}$, but we remember the $q$-grading shifts of the objects in the complex. Proceeding to the quotient $\LC{\cal{T}\hspace{0.5pt}}{\Sigma}{}$ in $\foam{N}^\Sigma$ has the effect of equating decorations by high-degree symmetric polynomials on $k$-facets with decorations by symmetric polynomials of lower degree and specializing decorations on $N$-facets to multiplication by complex scalars. Thus, the formal degree of any foam $F\colon q^{d_1} 1_{\mathbf{o}^{cut}} \to q^{d_2} 1_{\mathbf{o}^{cut}}$ appearing in a differential in the new complex $\LC{\cal{T}\hspace{0.5pt}}{\Sigma}{}$ is actually less than or equal to $d_2-d_1$. When applying the functor $\taut_{cut}$ to $\LC{\cal{T}\hspace{0.5pt}}{\Sigma}{}$ every $q^d 1_{\mathbf{o}^{cut}}$ in the complex will contribute a copy of $q^d H_k^\Sigma$ to the chain groups of $C^{\Sigma}(\cal{T})$ according to Lemma~\ref{lem-endos}. Now, however, we consider this copy as being equipped with the filtration from Example~\ref{exa-filt} shifted by $d$. This defines filtrations on all chain groups of $C^{\Sigma}(\cal{T})$ and we will check that the differential is filtered too. An element of $\cal{F}^p \taut_{cut}(q^{d_1} 1_{\mathbf{o}^{cut}})$ is given by an identity foam decorated by a polynomial of degree $\leq p+d_1$. By the above discussion, applying the differential $\taut_{cut}(F)$ adds a further decoration by a symmetric polynomial of degree $\leq d_2-d_1$. The resulting total decoration on the identity foam is of degree $\leq p + d_2$, so it lies in $\cal{F}^p \taut_{cut}(q^{d_2} 1_{\mathbf{o}^{cut}})$.

Now we know that $C^{\Sigma}(\cal{T})$ is endowed with a filtration, and since it is clearly bounded, there is an induced spectral sequence which converges to $\mathrm{Gr}\left( \LH{\cal{K}}{\Sigma} \right )$. Finally, we argue that $E_0:= \mathrm{Gr}\left( C^{\Sigma}(\cal{T})\right )$ is isomorphic to $C^{\glnn{N}}(\cal{T})$. For the chain groups, this follows from the fact $\mathrm{Gr}(H_k^\Sigma)\cong H_k^N$, which was observed in Lemma~\ref{lem-filt}. Moreover, it is clear from the above discussion that taking the grading preserving (top degree) part of the differential in $C^{\Sigma}(\cal{T})$ recovers the differential in $C^{\glnn{N}}(\cal{T})$. The statement about the $E_1$ page follows immediately.
\end{proof}

\begin{rem} Theorem~\ref{thm_ss} also holds for links $\cal{L}$, although in this case an analogous proof using closed diagrams (and endomorphism categories of $\mathbf{o}^{top}$ instead of $\mathbf{o}^{cut}$) would appear to be more natural.
\end{rem}

We will need deformation spectral sequences for reduced colored Khovanov--Rozansky homology.

\begin{thm}\label{thm_reducedss} Let $\Sigma$ be an $N$-element multiset of complex numbers and $A$ a $k$-element multisubset of $\Sigma$ . Then there is a \emph{reduced deformation spectral sequence} satisfying the following conditions:
\begin{itemize}
\item the $E_0$ page is isomorphic to the reduced bi-graded Khovanov--Rozansky complex $\overline{C(\cal{K}^k)}^{\glnn{N}}$
\item the $E_1$ page is isomorphic to the reduced bi-graded Khovanov--Rozansky homology $\overline{\mathrm{KhR}}^{\glnn{N}}(\cal{K}^k) $
\item the $E_\infty$ page is isomorphic to the associated graded of the $A$-reduced $\Sigma$-deformed Khovanov--Rozansky homology, which by Theorem~\ref{thm_reduceddecomp} decomposes into a tensor product of undeformed reduced homologies of smaller color and rank:  
\[E_\infty \cong
\mathrm{Gr}\left( \overline{\mathrm{KhR}}^\Sigma_A(\cal{K}^k)\right ) \overset{\text{Theorem }\ref{thm_reduceddecomp}}{\cong}
\mathrm{Gr}\left(  \bigotimes_{j=1}^l \overline{\mathrm{KhR}}^{\glnn{N_j}}(\cal{K}^{a_j})\right ) 
 \]
 if $\Sigma= \{\lambda_1^{N_1},\dots, \lambda_l^{N_l}\}$ and $A= \{\lambda_1^{a_1},\dots, \lambda_l^{a_l}\}$. As a singly-graded $\C$-vector space, the $E_\infty$ page is isomorphic to $\overline{\mathrm{KhR}}^\Sigma_A(\cal{K}^k)\cong \bigotimes_{j=1}^l \overline{\mathrm{KhR}}^{\glnn{N_j}}(\cal{K}^{a_j})$ itself.
\end{itemize} 
\end{thm}
\begin{proof} As before, we work with a 2-ended tangle $\cal{T}$ instead of $\cal{K}$. It is clear that $\overline{C}_A^\Sigma(\cal{T}):=\la\pi_{A}\ra \taut_{cut}{\LC{\cal{T}\hspace{0.5pt}}{\Sigma}{}}$ inherits a quantum filtration from $\taut_{cut}{\LC{\cal{T}\hspace{0.5pt}}{\Sigma}{}}$ and the induced spectral sequence computes $\mathrm{Gr}\left(\overline{\mathrm{KhR}}^\Sigma_A(\cal{K}^k) \right )$. Moreover, the differentials in the complexes $\overline{C}_A^\Sigma(\cal{T})$ and $C^\Sigma(\cal{T}) $ are given by composition by the same foams. Thus, it only remains to show that the associated graded of a chain group of $\overline{C}_A^\Sigma(\cal{T})$ is isomorphic to the corresponding chain group of $\overline{C(\cal{K}^k)}^{\glnn{N}}$. Again, it suffices to check this for the elementary building blocks of the chain groups:
\[\mathrm{Gr}\left ( \la\pi_{A}\ra \taut_{cut}({1_{\mathbf{o}^{cut}}}) \right) \cong \mathrm{Gr}\left ( \la\pi_{A}\ra H_k^\Sigma \right) \cong \la\pi_{k}^N\ra H_k^N \cong \la\pi_{k}^N\ra \grtaut(1_{\mathbf{o}^{cut}}) \]
The middle isomorphism follows from Example~\ref{exa-filt}, Lemma~\ref{lem-filt} and Lemma~\ref{lem_piA}.
\end{proof}

\begin{cor} \label{cor_colss} There exist spectral sequences
\begin{align}
\overline{\mathrm{KhR}}^{\glnn{k N}}(\cal{K}^k)  \quad &\rightsquigarrow \quad  \left( \overline{\mathrm{KhR}}^{\glnn{N}}(\cal{K}^1)\right)^{\otimes k} \quad \text{and}
\\ 
\overline{\mathrm{KhR}}^{\glnn{2 N}}(\cal{K}^k)  \quad &\rightsquigarrow \quad \overline{\mathrm{KhR}}^{\glnn{N}}(\cal{K}^{k-1}) \,\otimes \, \overline{\mathrm{KhR}}^{\glnn{N}}(\cal{K}^{1}).
\end{align}
\end{cor}
\begin{proof}
By Theorem~\ref{thm_reducedss} there is a spectral sequence from $\overline{\mathrm{KhR}}^{\glnn{N}}(\cal{K}^k)$ to $\overline{\mathrm{KhR}}^\Sigma_A(\cal{K}^k)$ for any choice of $k$-element multisubset $A$ of an $N$-element multiset $\Sigma$ of complex deformation parameters. For the first type of spectral sequence we choose $\Sigma$ and $A$ such that $l=k$, $N_1=\dots=N_l=N$ and $a_1 = \dots = a_l=1$ and for the second type we choose $l=2$, $N_1=N_2=N$ and $a_1=k-1$, $a_2=1$.
\end{proof}

\begin{cor} \textbf{(Rank-reducing spectral sequences)}  \label{cor_rankdiff}
For $N\geq M\geq k$ there is a spectral sequence
\begin{align}
\overline{\mathrm{KhR}}^{\glnn{N}}(\cal{K}^k)  \quad &\rightsquigarrow \quad \overline{\mathrm{KhR}}^{\glnn{M}}(\cal{K}^k).
\end{align}
\end{cor}
\begin{proof}
Choose $l=2$ and $N_1=M$, $N_2=N-M$ and $a_1=k$, $a_2=0$ in Theorem~\ref{thm_reducedss} and note that $\overline{\mathrm{KhR}}^{\glnn{N-M}}(\cal{K}^{0})\cong \C$.
\end{proof}

\begin{cor} \label{cor_colss2}
 For $k\geq h$ there is a spectral sequence
\begin{align}
\overline{\mathrm{KhR}}^{\glnn{N+k-h}}(\cal{K}^k)  \quad &\rightsquigarrow \quad  \overline{\mathrm{KhR}}^{\glnn{N}}(\cal{K}^h).
\end{align}
\end{cor}
\begin{proof}
Choose $l=2$ and $N_1=N$, $N_2=k-h$ and $a_1=h$, $a_2=k-h$ in Theorem~\ref{thm_reducedss} and note that $\overline{\mathrm{KhR}}^{\glnn{k-h}}(\cal{K}^{k-h})\cong \C$.
\end{proof}

\begin{rem} \label{rem-cav1}
In general, it is unclear how fast the spectral sequences in the previous corollaries converge. However, it follows from the stability property of Theorem 2 in the introduction (see also Theorem~\ref{thm-gradcollapse}) that the spectral sequences from Corollary~\ref{cor_rankdiff} are degenerate for large $M$. 
\end{rem}

For completeness we mention the following theorem, which follows from Corollary~\ref{cor-hoffree} similarly as in the uncolored case, see e.g. Rasmussen's \cite[Lemma 5.5]{Ras3}.% or Lewark's \cite[Theorem 2]{Lew}.

\begin{thm} There is a spectral sequence 
\[\mathrm{KhR}^{\glnn{N}}(\bigcirc^k) \otimes \overline{\mathrm{KhR}}^{\glnn{N}}(\cal{K}^k) \quad \rightsquigarrow{}\quad  \mathrm{KhR}^{\glnn{N}}(\cal{K}^k)\]
which respects the homological and quantum grading.
\end{thm}

\section{Colored HOMFLY-PT homology}
\label{sec-HH}
In Section~\ref{sec-defiHHH} we give a review of colored HOMFLY-PT homology which utilizes the notation and techniques already introduced in this paper. In Section~\ref{sec-ss} we transition to matrix factorization technology to study spectral sequences connecting colored HOMFLY-PT and $\glnn{N}$ homologies as claimed in Theorem 3. Finally, we introduce reduced colored HOMFLY-PT homology in Section~\ref{sec-reducedHomfly} and prove Theorems~1, 2, 4 and 6.
\subsection{Construction}
\label{sec-defiHHH}

Let $\cal{L}$ be a labelled link presented as the \textit{closure} of a labelled braid diagram $\cal{B}$:
\begin{equation}
\label{eqn:braidclosure}
\xy
(0,0)*{
\begin{tikzpicture}[scale=.4]
	\draw [very thick,rdirected=.65] (-1.5,-.25) to  (-1,-.25);
	\draw [very thick,rdirected=.65] (-1.5,-1.25) to  (-1,-1.25);
	\draw [very thick] (1,-.25) to (1.5,-.25);
	\draw [very thick] (1,-1.25) to  (1.5,-1.25);
	\draw[thick] (1,0) rectangle (-1,-1.5);
	\node at (0,-.75) {$\cal{B}$};
	\node at (-1.25,-.75) {\tiny $\cdot$};
		\node at (1.25,-.75) {\tiny $\cdot$};
\end{tikzpicture}
};
\endxy\;
\xrightarrow{} \;\;
\overline{\cal{B}}=\xy
(0,0)*{
\begin{tikzpicture}[scale=.4]\draw [very thick,directed=.55] (-1,-.25) to [out=180,in=270] (-1.5,.25) to [out=90,in=180] (-1,.75) to (1,.75) to [out=0,in=90] (1.5,.25) to [out=270,in=0] (1,-.25);
\draw [very thick,directed=.53] (-1,-1.25) to [out=180,in=270] (-2,.25) to [out=90,in=180] (-1,1.75) to (1,1.75) to [out=0,in=90] (2,.25) to [out=270,in=0] (1,-1.25);
	\draw[thick] (1,0) rectangle (-1,-1.5);
	\node at (0,-.75) {$\cal{B}$};
		\node at (-1.25,-.75) {\tiny $\cdot$};
		\node at (1.25,-.75) {\tiny $\cdot$};
				\node at (0,.25) {$\bullet$};
\end{tikzpicture}
};
\endxy 
\xrightarrow{}\;\;  \cal{L}=
\xy
(0,0)*{
\begin{tikzpicture}[scale=.4]\draw [very thick,directed=.55] (-1,-.25) to [out=180,in=270] (-1.5,.25) to [out=90,in=180] (-1,.75) to (1,.75) to [out=0,in=90] (1.5,.25) to [out=270,in=0] (1,-.25);
\draw [very thick,directed=.53] (-1,-1.25) to [out=180,in=270] (-2,.25) to [out=90,in=180] (-1,1.75) to (1,1.75) to [out=0,in=90] (2,.25) to [out=270,in=0] (1,-1.25);
	\draw[thick] (1,0) rectangle (-1,-1.5);
	\node at (0,-.75) {$\cal{B}$};
		\node at (-1.25,-.75) {\tiny $\cdot$};
		\node at (1.25,-.75) {\tiny $\cdot$};
\end{tikzpicture}
};
\endxy 
\end{equation}
To be compatible with our conventions for webs and foams, we draw braids horizontally and oriented from right to left. In \eqref{eqn:braidclosure} we illustrate the closure of $\cal{B}$, which produces a link diagram $\cal{L}$ in $\R^2$, and also the intermediate \textit{annular closure}, which produces a link diagram $\overline{\cal{B}}$ in a planar annulus.

We will consider the cube of resolutions chain complex $\LC{\cal{B}}{}{}$, whose construction was sketched in Section~\ref{sec-2-2}, as a chain complex over $\foam{N}$, where we choose $N$ larger than the sum of the labels on the strands of $\cal{B}$. Denote the domain and co-domain object of webs in $\LC{\cal{B}}{}{}$ by $\mathbf{o}_\cal{B}=(o_1,\dots, o_m)$. More generally, a braid, tangle or web whose left and right boundaries carry matching sequences of labels is called \emph{balanced}.

\begin{defi} For a web $W\colon \mathbf{o}_1 \to \mathbf{o}_2$ in $\foam{N}$ we consider the sub-algebra $\Dec(W)$ of the endomorphism algebra $\End_{\foam{N}}(W)$ generated by the decorated identity foams on $W$. For identity foams on an object $\mathbf{o}=(o_1,\dots, o_k)$ we abbreviate $\Dec(\mathbf{o}):=\Dec(1_{\mathbf{o}})$ and note that $\Dec(W)$ is naturally a $\Dec(\mathbf{o}_1)\otimes \Dec(\mathbf{o}_2)$-module by placing decorations on the right and left boundary facets.  

With respect to a braid diagram $\cal{B}$ as above, we define $R_\cal{B}:=\Dec(\mathbf{o}_\cal{B})\cong \Sym(\X_1|\dots|\X_m)$ and remind the reader that it is given by the graded algebra of polynomials separately symmetric in $o_i$-element alphabets $\X_i$. Under the assignment $\Dec$, the webs in the complex $\LC{\cal{B}}{}{}$ correspond to $R_\cal{B}$-bimodules, which are special cases of \emph{singular Soergel bimodules}, see~\cite{Wil}. 
\end{defi}

\begin{exa}\label{exa-bimodules} The explicit description of the following decoration bimodules follows from the decoration migration relations~\eqref{eqn-decmigr} on foams:
\begin{gather*}
\nonumber
\Dec \big(
\xy
(0,0)*{
\begin{tikzpicture} [scale=.3,fill opacity=0.2]
	%bottom web
	\draw[very thick, directed=.55] (1,1) -- (-1,1);	
\end{tikzpicture}
};
\endxy
\big)
\cong 
\frac{\Sym(\V | \Y)}{\langle \V =\Y \rangle} \\
\Dec \big(
\xy
(0,0)*{
\begin{tikzpicture} [scale=.3,fill opacity=0.2]
	\draw[very thick,directed=.55] (2,4) -- (.75,4);
	\draw[very thick,directed=.55] (-.75,4) -- (-2,4);
	\draw[very thick,directed=.55] (.75,4) .. controls (.5,3.5) and (-.5,3.5) .. (-.75,4);
	\draw[very thick,directed=.55] (.75,4) .. controls (.5,4.5) and (-.5,4.5) .. (-.75,4);
\end{tikzpicture}
};
\endxy
\big)
\cong
\left( \frac{\Sym(\V | \L | \M)}{\langle \V = \L \cup \M \rangle} \right) \otimes_{\Sym(\L | \M)} 
\left( \frac{\Sym(\L | \M | \Y)}{\langle \L \cup \M = \Y \rangle} \right) \\
\nonumber
\Dec \big(
\xy
(0,0)*{
\begin{tikzpicture} [scale=.3,fill opacity=0.2]
	\draw [very thick,directed=.55] (.75,0) to (-.75,0);
	\draw [very thick,directed=.75] (-.75,0) to [out=135,in=0] (-2.25,.5);
	\draw [very thick,directed=.75] (-.75,0) to [out=225,in=0] (-2.25,-.5);
	\draw [very thick,directed=.45] (2.25,.5) to [out=180,in=45] (.75,0);
	\draw [very thick,directed=.45] (2.25,-.5) to [out=180,in=315] (.75,0);
\end{tikzpicture}
};
\endxy
\big)
\cong
\left( \frac{\Sym(\V | \W | \L)}{\langle \V \cup \W = \L \rangle} \right) \otimes_{\Sym(\L)} 
\left( \frac{\Sym(\L | \X | \Y)}{\langle \L = \X \cup \Y \rangle} \right) \\
\label{eq-linfac}
\Dec \big(
\xy
(0,0)*{
\begin{tikzpicture} [scale=.3,fill opacity=0.2]
	\draw [very thick,directed=.55] (1,.5) to [out=190,in=350] (-1,.5);
	\draw [very thick,directed=.55] (1,-.5) to [out=170,in=10] (-1,-.5);
\end{tikzpicture}
};
\endxy
\big)
\cong
\frac{\Sym(\V | \X)}{\langle \V =\X \rangle} \otimes_{\C}
\frac{\Sym(\W | \Y)}{\langle \W= \Y \rangle}  \\
\nonumber
\Dec \left(
\xy
(0,0)*{
\begin{tikzpicture} [scale=.3,fill opacity=0.2]
	\draw[very thick, directed=.65] (2,3) to [out=180,in=0] (.75,2.75);
	\draw[very thick, directed=.65] (.75,2.75) to [out=135,in=0] (-.75,3.5);
	\draw[very thick, directed=.65] (.75,2.75) to [out=225,in=0] (-2,2);
	\draw[very thick, directed=.75]  (-.75,3.5) to [out=225,in=0] (-2,3);
	\draw[very thick, directed=.75]  (-.75,3.5) to [out=135,in=0] (-2,4);
\end{tikzpicture}
};
\endxy
\right) \cong
\left( \frac{\Sym(\V | \W | \L)}{\langle \V \cup \W =\L \rangle} \right) \otimes_{\Sym(\L)} 
\left( \frac{\Sym(\L | \X | \Y)}{\langle \L \cup \X = \Y \rangle} \right) \\
\nonumber
\Dec \left(
\xy
(0,0)*{
\begin{tikzpicture} [scale=.3,fill opacity=0.2]
	\draw[very thick, directed=.55] (2,3) to [out=180,in=0] (.75,3.25);
	\draw[very thick, directed=.65] (.75,3.25) to [out=225,in=0] (-.75,2.5);
	\draw[very thick, directed=.65] (.75,3.25) to [out=135,in=0] (-2,4);
	\draw[very thick, directed=.75] (-.75,2.5) to [out=135,in=0] (-2,3);
	\draw[very thick, directed=.75] (-.75,2.5) to [out=225,in=0] (-2,2);
\end{tikzpicture}
};
\endxy
\right)
\cong
\left( \frac{\Sym(\W | \X | \M)}{\langle \W \cup \X = \M \rangle} \right) \otimes_{\Sym(\M)} 
\left( \frac{\Sym(\V | \M | \Y)}{\langle \V \cup \M =\Y \rangle} \right)
\end{gather*}
Here and in the following, we use the suggestive shorthand $\langle \X_1 \cup \cdots \cup \X_{a} = \X_{a+1} \cup \cdots \cup \X_{a+b}\rangle$ for ideals of the form $ \langle e_i(\X_1 \cup \cdots \cup \X_{a}) - e_i(\X_{a+1} \cup \cdots \cup \X_{a+b})| i>0 \rangle$ which have the effect of identifying polynomials in the respective alphabets in the quotient ring.
\end{exa}

\begin{rem}
In \cite{RW}, Rose and the author have used the singular Soergel bimodules from Example~\ref{exa-bimodules} to relate the (deformed) link homologies constructed in foam 2-categories with the original construction using matrix factorizations \cite{KR1, Yon, Wu1, Wu2}. More precisely, we show that, under favourable circumstances\footnote{In the absence of oriented cycles in the web.}, the matrix factorizations assigned to webs in this construction are essentially determined by the bimodules $\Dec(\cdot)$ via the process of stabilization. Similarly, homomorphisms of matrix factorizations are then determined by the corresponding bimodule maps. For details, see \cite[Section 4.4]{RW} and the references therein, in particular \cite{Bec, CM2, CM, Web1}. 
\end{rem}

The main tool from \cite[Section 4.4]{RW} that we need in the following is that $\Dec$ extends to a 2-functor from the 2-category of foams to the 2-category of singular Soergel bimodules. We have already described that on the level of 1-morphisms, the 2-functor $\Dec$ sends the web $W$ to the decoration bimodule $\Dec(W)$. On the level of 2-morphisms, $\Dec$ sends foams to the bimodule homomorphisms specified as follows with respect to the notation in Example~\ref{exa-bimodules}:
\begin{gather*}
\xy
(0,0)*{
\begin{tikzpicture} [scale=.5,fill opacity=0.2]
	%back cup 
	\path[fill=green] (-.75,4) to [out=270,in=180] (0,2.5) to [out=0,in=270] (.75,4) .. controls (.5,4.5) and (-.5,4.5) .. (-.75,4);
	\path[fill=green] (-.75,4) to [out=270,in=180] (0,2.5) to [out=0,in=270] (.75,4) -- (2,4) -- (2,1) -- (-2,1) -- (-2,4) -- (-.75,4);
	%front cup
	\path[fill=green] (-.75,4) to [out=270,in=180] (0,2.5) to [out=0,in=270] (.75,4) .. controls (.5,3.5) and (-.5,3.5) .. (-.75,4);
	\draw[very thick, directed=.55] (2,1) -- (-2,1);
	\path (.75,1) .. controls (.5,.5) and (-.5,.5) .. (-.75,1); 
	%seam
	\draw [very thick, red, directed=.65] (-.75,4) to [out=270,in=180] (0,2.5) to [out=0,in=270] (.75,4);
	%vertical edges
	\draw[very thick] (2,4) -- (2,1);
	\draw[very thick] (-2,4) -- (-2,1);
	\draw[very thick,directed=.55] (2,4) -- (.75,4);
	\draw[very thick,directed=.55] (-.75,4) -- (-2,4);
	\draw[very thick,directed=.55] (.75,4) .. controls (.5,3.5) and (-.5,3.5) .. (-.75,4);
	\draw[very thick,directed=.55] (.75,4) .. controls (.5,4.5) and (-.5,4.5) .. (-.75,4);
	\node [opacity=1]  at (1.5,3.5) {\tiny{$_{a+b}$}};
	\node[opacity=1] at (.25,3.4) {\tiny{$a$}};
	\node[opacity=1] at (-.25,4.1) {\tiny{$b$}};	
\end{tikzpicture}
};
\endxy 
\mapsto
\big(
%\left\{
\overline{1} \mapsto \overline{1} \otimes \overline{1}
%\right\}
\big)
\quad , \quad
\xy
(0,0)*{
\begin{tikzpicture} [scale=.5,fill opacity=0.2]
	%back cup 
	\path[fill=green] (-.75,-4) to [out=90,in=180] (0,-2.5) to [out=0,in=90] (.75,-4) .. controls (.5,-4.5) and (-.5,-4.5) .. (-.75,-4);
	%sheet
	\path[fill=green] (-.75,-4) to [out=90,in=180] (0,-2.5) to [out=0,in=90] (.75,-4) -- (2,-4) -- (2,-1) -- (-2,-1) -- (-2,-4) -- (-.75,-4);
	%front cup
	\path[fill=green] (-.75,-4) to [out=90,in=180] (0,-2.5) to [out=0,in=90] (.75,-4) .. controls (.5,-3.5) and (-.5,-3.5) .. (-.75,-4);
	\draw[very thick, directed=.55] (2,-1) -- (-2,-1);
	\path (.75,-1) .. controls (.5,-.5) and (-.5,-.5) .. (-.75,-1); %for spacing symmetry
	%seam
	\draw [very thick, red, directed=.65] (.75,-4) to [out=90,in=0] (0,-2.5) to [out=180,in=90] (-.75,-4);
	%vertical edges
	\draw[very thick] (2,-4) -- (2,-1);
	\draw[very thick] (-2,-4) -- (-2,-1);
	%bottom web
	\draw[very thick,directed=.55] (2,-4) -- (.75,-4);
	\draw[very thick,directed=.55] (-.75,-4) -- (-2,-4);
	\draw[very thick,directed=.55] (.75,-4) .. controls (.5,-3.5) and (-.5,-3.5) .. (-.75,-4);
	\draw[very thick,directed=.55] (.75,-4) .. controls (.5,-4.5) and (-.5,-4.5) .. (-.75,-4);
	\node [opacity=1]  at (1.25,-1.25) {\tiny{$_{a+b}$}};
	\node[opacity=1] at (-.25,-3.4) {\tiny{$b$}};
	\node[opacity=1] at (.25,-4.1) {\tiny{$a$}};
\end{tikzpicture}
};
\endxy
\mapsto
\left(
\overline{\pi_{\lambda}^{\L}} \otimes \overline{1} \mapsto 
\left\{ \begin{array}{l} 
\overline{1} \quad \text{if} \quad \lambda=b^a, \\
0 \quad \text{all other } \lambda \in P(a,b),
\end{array} \right.
\right)
\\
\xy
(0,0)*{
\begin{tikzpicture} [scale=.5,fill opacity=0.2]
	%shading
	\path [fill=green] (4.25,-.5) to (4.25,2) to [out=165,in=15] (-.5,2) to (-.5,-.5) to 
		[out=0,in=225] (.75,0) to [out=90,in=180] (1.625,1.25) to [out=0,in=90] 
			(2.5,0) to [out=315,in=180] (4.25,-.5);
	\path [fill=green] (3.75,.5) to (3.75,3) to [out=195,in=345] (-1,3) to (-1,.5) to 
		[out=0,in=135] (.75,0) to [out=90,in=180] (1.625,1.25) to [out=0,in=90] 
			(2.5,0) to [out=45,in=180] (3.75,.5);
	\path[fill=green] (.75,0) to [out=90,in=180] (1.625,1.25) to [out=0,in=90] (2.5,0);
	%bottom web
	\draw [very thick,directed=.55] (2.5,0) to (.75,0);
	\draw [very thick,directed=.55] (.75,0) to [out=135,in=0] (-1,.5);
	\draw [very thick,directed=.55] (.75,0) to [out=225,in=0] (-.5,-.5);
	\draw [very thick,directed=.55] (3.75,.5) to [out=180,in=45] (2.5,0);
	\draw [very thick,directed=.55] (4.25,-.5) to [out=180,in=315] (2.5,0);
	\draw [very thick, red, directed=.75] (.75,0) to [out=90,in=180] (1.625,1.25);
	\draw [very thick, red] (1.625,1.25) to [out=0,in=90] (2.5,0);
	%vertical edges
	\draw [very thick] (3.75,3) to (3.75,.5);
	\draw [very thick] (4.25,2) to (4.25,-.5);
	\draw [very thick] (-1,3) to (-1,.5);
	\draw [very thick] (-.5,2) to (-.5,-.5);
	%top web
	\draw [very thick,directed=.55] (4.25,2) to [out=165,in=15] (-.5,2);
	\draw [very thick, directed=.55] (3.75,3) to [out=195,in=345] (-1,3);
	\node[opacity=1]  at (1.625,.5) {\tiny{$_{a+b}$}};
	\node[opacity=1] at (3.5,2.65) {\tiny{$b$}};
	\node[opacity=1] at (4,1.85) {\tiny{$a$}};		
\end{tikzpicture}
};
\endxy
\mapsto
\big(
\overline{1}\otimes \overline{1} \mapsto \overline{1} \otimes \overline{1}
\big)
\quad , \quad
\xy
(0,0)*{
\begin{tikzpicture} [scale=.5,fill opacity=0.2]
	%shading
	\path [fill=green] (4.25,2) to (4.25,-.5) to [out=165,in=15] (-.5,-.5) to (-.5,2) to
		[out=0,in=225] (.75,2.5) to [out=270,in=180] (1.625,1.25) to [out=0,in=270] 
			(2.5,2.5) to [out=315,in=180] (4.25,2);
	\path [fill=green] (3.75,3) to (3.75,.5) to [out=195,in=345] (-1,.5) to (-1,3) to [out=0,in=135]
		(.75,2.5) to [out=270,in=180] (1.625,1.25) to [out=0,in=270] 
			(2.5,2.5) to [out=45,in=180] (3.75,3);
	\path[fill=green] (2.5,2.5) to [out=270,in=0] (1.625,1.25) to [out=180,in=270] (.75,2.5);
	%bottom web
	\draw [very thick,directed=.55] (4.25,-.5) to [out=165,in=15] (-.5,-.5);
	\draw [very thick, directed=.55] (3.75,.5) to [out=195,in=345] (-1,.5);
	%seam
	\draw [very thick, red, directed=.75] (2.5,2.5) to [out=270,in=0] (1.625,1.25);
	\draw [very thick, red] (1.625,1.25) to [out=180,in=270] (.75,2.5);
	\draw [very thick] (3.75,3) to (3.75,.5);
	\draw [very thick] (4.25,2) to (4.25,-.5);
	\draw [very thick] (-1,3) to (-1,.5);
	\draw [very thick] (-.5,2) to (-.5,-.5);
	\draw [very thick,directed=.55] (2.5,2.5) to (.75,2.5);
	\draw [very thick,directed=.55] (.75,2.5) to [out=135,in=0] (-1,3);
	\draw [very thick,directed=.55] (.75,2.5) to [out=225,in=0] (-.5,2);
	\draw [very thick,directed=.55] (3.75,3) to [out=180,in=45] (2.5,2.5);
	\draw [very thick,directed=.55] (4.25,2) to [out=180,in=315] (2.5,2.5);
	%labels
	\node [opacity=1]  at (1.625,2) {\tiny{$_{a+b}$}};
	\node[opacity=1] at (3.5,2.65) {\tiny{$b$}};
	\node[opacity=1] at (4,1.85) {\tiny{$a$}};		
\end{tikzpicture}
};
\endxy
\mapsto
\left(
\overline{1} \otimes \overline{1} \mapsto
\sum_{\alpha \in P(a,b)} (-1)^{|\hat{\alpha}|} \overline{\pi_{\hat{\alpha}}^{\V}} \otimes \overline{\pi_{\alpha}^{\Y}}
\right)
\\
\xy
(0,0)*{
\begin{tikzpicture} [scale=.5,fill opacity=0.2]
	\path[fill=green] (-2.5,4) to [out=0,in=135] (-.75,3.5) to [out=270,in=90] (.75,.25)
		to [out=135,in=0] (-2.5,1);
	\path[fill=green] (-.75,3.5) to [out=270,in=125] (.29,1.5) to [out=55,in=270] (.75,2.75) 
		to [out=135,in=0] (-.75,3.5);
	\path[fill=green] (-.75,-.5) to [out=90,in=235] (.29,1.5) to [out=315,in=90] (.75,.25) 
		to [out=225,in=0] (-.75,-.5);
	\path[fill=green] (-2,3) to [out=0,in=225] (-.75,3.5) to [out=270,in=125] (.29,1.5)
		to [out=235,in=90] (-.75,-.5) to [out=135,in=0] (-2,0);
	\path[fill=green] (-1.5,2) to [out=0,in=225] (.75,2.75) to [out=270,in=90] (-.75,-.5)
		to [out=225,in=0] (-1.5,-1);
	\path[fill=green] (2,3) to [out=180,in=0] (.75,2.75) to [out=270,in=55] (.29,1.5)
		to [out=305,in=90] (.75,.25) to [out=0,in=180] (2,0);
	%bottom web
	\draw[very thick, directed=.55] (2,0) to [out=180,in=0] (.75,.25);
	\draw[very thick, directed=.55] (.75,.25) to [out=225,in=0] (-.75,-.5);
	\draw[very thick, directed=.55] (.75,.25) to [out=135,in=0] (-2.5,1);
	\draw[very thick, directed=.55] (-.75,-.5) to [out=135,in=0] (-2,0);
	\draw[very thick, directed=.55] (-.75,-.5) to [out=225,in=0] (-1.5,-1);
	%seams
	\draw[very thick, red, rdirected=.85] (-.75,3.5) to [out=270,in=90] (.75,.25);
	\draw[very thick, red, rdirected=.75] (.75,2.75) to [out=270,in=90] (-.75,-.5);	
	%vertical edges
	\draw[very thick] (-1.5,-1) -- (-1.5,2);	
	\draw[very thick] (-2,0) -- (-2,3);
	\draw[very thick] (-2.5,1) -- (-2.5,4);	
	\draw[very thick] (2,3) -- (2,0);
	\draw[very thick, directed=.55] (2,3) to [out=180,in=0] (.75,2.75);
	\draw[very thick, directed=.55] (.75,2.75) to [out=135,in=0] (-.75,3.5);
	\draw[very thick, directed=.65] (.75,2.75) to [out=225,in=0] (-1.5,2);
	\draw[very thick, directed=.55]  (-.75,3.5) to [out=225,in=0] (-2,3);
	\draw[very thick, directed=.55]  (-.75,3.5) to [out=135,in=0] (-2.5,4);
	%labels
	\node[opacity=1] at (-2.25,3.375) {\tiny$c$};
	\node[opacity=1] at (-1.75,2.75) {\tiny$b$};	
	\node[opacity=1] at (-1.25,1.75) {\tiny$a$};
	\node[opacity=1] at (0,2.75) {\tiny$_{b+c}$};
	\node[opacity=1] at (0,.25) {\tiny$_{a+b}$};
\end{tikzpicture}
};
\endxy
\mapsto
\big(
\overline{1} \otimes \overline{1} \mapsto \overline{1} \otimes \overline{1}
\big)
\quad , \quad
\xy
(0,0)*{
\begin{tikzpicture} [scale=.5,fill opacity=0.2]
	\path[fill=green] (-2.5,4) to [out=0,in=135] (.75,3.25) to [out=270,in=90] (-.75,.5)
		 to [out=135,in=0] (-2.5,1);
	\path[fill=green] (-.75,2.5) to [out=270,in=125] (-.35,1.5) to [out=45,in=270] (.75,3.25) 
		to [out=225,in=0] (-.75,2.5);
	\path[fill=green] (-.75,.5) to [out=90,in=235] (-.35,1.5) to [out=315,in=90] (.75,-.25) 
		to [out=135,in=0] (-.75,.5);	
	\path[fill=green] (-2,3) to [out=0,in=135] (-.75,2.5) to [out=270,in=125] (-.35,1.5) 
		to [out=235,in=90] (-.75,.5) to [out=225,in=0] (-2,0);
	\path[fill=green] (-1.5,2) to [out=0,in=225] (-.75,2.5) to [out=270,in=90] (.75,-.25)
		to [out=225,in=0] (-1.5,-1);
	\path[fill=green] (2,3) to [out=180,in=0] (.75,3.25) to [out=270,in=45] (-.35,1.5) 
		to [out=315,in=90] (.75,-.25) to [out=0,in=180] (2,0);				
	%bottom web
	\draw[very thick, directed=.55] (2,0) to [out=180,in=0] (.75,-.25);
	\draw[very thick, directed=.55] (.75,-.25) to [out=135,in=0] (-.75,.5);
	\draw[very thick, directed=.55] (.75,-.25) to [out=225,in=0] (-1.5,-1);
	\draw[very thick, directed=.45]  (-.75,.5) to [out=225,in=0] (-2,0);
	\draw[very thick, directed=.35]  (-.75,.5) to [out=135,in=0] (-2.5,1);	
	%seams
	\draw[very thick, red, rdirected=.75] (-.75,2.5) to [out=270,in=90] (.75,-.25);
	\draw[very thick, red, rdirected=.85] (.75,3.25) to [out=270,in=90] (-.75,.5);
	%vertical edges
	\draw[very thick] (-1.5,-1) -- (-1.5,2);	
	\draw[very thick] (-2,0) -- (-2,3);
	\draw[very thick] (-2.5,1) -- (-2.5,4);	
	\draw[very thick] (2,3) -- (2,0);
	% top web
	\draw[very thick, directed=.55] (2,3) to [out=180,in=0] (.75,3.25);
	\draw[very thick, directed=.55] (.75,3.25) to [out=225,in=0] (-.75,2.5);
	\draw[very thick, directed=.55] (.75,3.25) to [out=135,in=0] (-2.5,4);
	\draw[very thick, directed=.55] (-.75,2.5) to [out=135,in=0] (-2,3);
	\draw[very thick, directed=.55] (-.75,2.5) to [out=225,in=0] (-1.5,2);
	\node[opacity=1] at (-2.25,3.75) {\tiny $c$};
	\node[opacity=1] at (-1.75,2.75) {\tiny $b$};	
	\node[opacity=1] at (-1.25,1.75) {\tiny $a$};
	\node[opacity=1] at (-.125,2.25) {\tiny $_{a+b}$};
	\node[opacity=1] at (-.125,.75) {\tiny $_{b+c}$};
\end{tikzpicture}
};
\endxy
\mapsto
\big(
\overline{1} \otimes \overline{1} \mapsto \overline{1} \otimes \overline{1}
\big)
\end{gather*}
Here $\bar{f}$ denotes an element of a quotient ring (as shown in Example~\ref{exa-bimodules}) represented by a polynomial $f$,  $\pi_{\lambda}^{\W}$ denotes the Schur 
polynomial in the alphabet $\W$ corresponding to the Young diagram $\lambda$. $P(a,b)$ is the set of Young diagrams fitting into a box of size $a\times b$ and its maximal element is  $b^a$. For a Young diagram $\lambda\in P(a,b)$ with rows $\lambda_i$ for $0\leq i \leq a$, the complementary diagram $\lambda^c$ has rows $\lambda^c_{i}=b-\lambda_{a+1-i}$ for $0\leq i \leq a$, and $\hat{\lambda}$ denotes the transpose of $\lambda^c$.
\begin{prop}
\label{prop-bimodconstr} The above assignments uniquely define bimodule maps between the corresponding decoration bimodules. $\Dec$ extends naturally to a monoidal 2-functor from $\foam{N}$ to the monoidal 2-category $\Bimod$ of singular Soergel bimodules. For a balanced labelled braid diagram $\cal{B}$, the complex $\Dec(\LC{\cal{B}}{}{})$ is isomorphic to the complex of singular Soergel bimodules constructed by Mackaay--Sto\v{s}i\'c--Vaz \cite{MSV} and Webster--Williamson \cite{WW}.
\end{prop}
\begin{proof} Uniqueness is clear since the maps are defined on a generating set and the bimodule map condition can be checked explicitly. All arising decoration bimodules are singular Soergel bimodules which form a 2-category $\Bimod$ and \cite[Section 4.4]{RW} contains a proof that the images of foam relations under $\Dec$ hold in $\Bimod$. Moreover, $\Dec$ respects the monoidal structure given by disjoint union on foams and by tensor product over $\C$ on bimodules. A quick comparison with \cite{MSV} shows that the complex constructed there and $\Dec(\LC{\cal{B}}{}{})$ have the same bimodules as objects. Furthermore, in both cases, the differentials are constructed locally for each crossing and it is well-known that the differentials in a crossing complex of singular Soergel bimodules are uniquely determined (up to non-zero scalars), see e.g. \cite[Corollary 5.7.]{WW}. Thus, the complexes are isomorphic. 
\end{proof}

Next, one needs a functor to translate $\Dec(\LC{\cal{B}}{}{})$ into a complex of triply-graded vector spaces.

\begin{defi} Let $\cal{R}$ be a commutative, graded $\C$-algebra. The Hochschild homology of a graded $\cal{R}$-bimodule $M$ (equivalently, a $\cal{R}\otimes \cal{R}$-module) is defined to be:
\[ HH_*(M):=\Tor^{\cal{R}\otimes \cal{R}}_{*}(\cal{R},M)\]
\end{defi}
In fact, $HH_*(\cdot)$ is covariant functor from the category of graded $\cal{R}$-bimodules to the category of bi-graded $\C$-vector spaces. 

\begin{defi}
With $\cal{R}$ as in the previous definition, let $\mathbf{x}=\{x_1,\dots, x_n\}$ be a sequence of homogeneous elements of $\cal{R}$, then the the Koszul complex of $\mathbf{x}$ is defined as
\[\cal{Z}_\mathbf{x} := \bigotimes_{i=1}^n \left ( \cal{R}\xrightarrow{x_i} \uwave{\cal{R}}\right ) \]
where we assume that the copies of $\cal{R}$ have been shifted in (internal) grading to make all differentials homogeneous of degree $2$. The homological grading in this complex is called the \emph{horizontal grading}. We use the convention that the the differential in this complex increases the horizontal grading by one and that the \uwave{underlined} term is in horizontal and internal grading zero. 
\end{defi}

In the following let $W\colon \mathbf{o}_\cal{B} \to \mathbf{o}_\cal{B}$ be a web and $\Dec(W)$ its associated $\cal{R}_\cal{B}$-bimodule. Recall that $R_\cal{B}\cong \Sym(\X_1|\dots|\X_m)$ with $|\X_i|=o_i$. Let $e_j(\X_i)$ denote the $j^{th}$ elementary symmetric polynomial in $\X_i$, which is of internal degree $2 j$. We also identify $\cal{R}_\cal{B}\otimes \cal{R}_\cal{B}\cong \Sym(\X'_1|\dots|\X'_m|\X_1|\dots|\X_m)$. Next, we recall how the Hochschild homology of bimodules over (symmetric) polynomial rings is computed via Koszul resolutions. For more details see e.g. \cite{Kho3}.

\begin{lem}\label{lem-Koszul}
Consider the sequence $\mathbf{x}=\{e_j(\X_i)- e_j(\X'_i)\}$ in $\cal{R}_\cal{B}\otimes \cal{R}_\cal{B}$ where the indices range in $1\leq i\leq m$ and $1\leq j\leq o_i$. Then the Koszul complex $\cal{Z}_\mathbf{x}$ is a resolution of $\cal{R}_\cal{B}$ by free $\cal{R}_\cal{B}\otimes \cal{R}_\cal{B}$-modules and we can compute the Hochschild homology of a $\cal{R}_\cal{B}$-bimodule $M$ as
\begin{equation}
\label{eqn-HH} HH_*(M) \cong H_*(M\otimes \cal{Z}_\mathbf{x}).
\end{equation}
\end{lem}

In other words, the Hochschild homology $\mathrm{HH}_*(\Dec(W))$ of $\Dec(W)$ can be computed as the homology of a complex in the shape of the 1-skeleton of a $\sum_i  o_i$-dimensional hypercube, with copies of $\Dec(W)$ on all vertices and differentials given by multiplication by $\pm (e_j(\X_i)- e_j(\X'_i))$ on the edges, see also Example~\ref{exa-2unknot}. 

\begin{defi} Let $\cal{B}$ be a balanced labelled braid diagram as above. Then the triply-graded colored HOMFLY-PT homology $\LH{\cal{B}}{\infty}{} = \bigoplus_{i,j,h} \LHH{\cal{B}}{i,j,h}$ is defined as:
\begin{equation}
\LHH{\cal{B}}{*,j,h}:= H_h(HH_j(\Dec(\LC{\cal{B}}{}{})))
\end{equation}
The three gradings are the internal grading (index $i$), the horizontal grading (index $j$) and the homological, a.k.a. \emph{vertical grading} inherited from the complex $\LC{\cal{B}}{}{}$ (index $h$). 
\end{defi}

Modulo Proposition~\ref{prop-bimodconstr} and grading conventions, this definition is due to Mackaay--Sto\v{s}i\'c--Vaz~\cite{MSV} and Webster--Williamson~\cite{WW}, building on work of Khovanov~\cite{Kho3} in the uncolored setting, who also proved the predecessor of the following theorem.

\begin{thm} \cite[c.f.~Theorem 1.1.]{WW}. Let $\cal{L}$ be a labelled link and $\cal{B}$ any diagram for a braid representative of $\cal{L}$. Then, as a triply-graded vector space, $\LH{\cal{L}}{\infty}{}:=\LH{\cal{B}}{\infty}{}$ only depends on $\cal{L}$ and not on $\cal{B}$ (up to an overall grading shift). Furthermore $\LH{\cal{L}}{\infty}{}$ categorifies the unreduced colored HOMFLY-PT polynomial $P^\mathrm{HOMFLY-PT}(\cal{L})$ of the link $\cal{L}$:
\[\sum_{i,j,h}(-1)^h(-a^2)^j q^i \dim_\C \LHH{\cal{L}}{i,j,h} = P^\mathrm{HOMFLY-PT}(\cal{L}) \]
up to overall multiplication by a monomial in $a$ and $q$.  
\end{thm}

\begin{exa}\label{exa-2unknot} We compute the HOMFLY-PT homology of the $2$-labelled unknot via its trivial 1-strand braid diagram. It is concentrated in homological degree zero and given by the Hochschild homology of $\Dec \big(
\xy
(0,0)*{
\begin{tikzpicture} [scale=.3,fill opacity=0.2]
	\draw[very thick, directed=.55] (1,1) -- (-1,1);
	\node [opacity=1]  at (1.5,1) {\tiny{$_{2}$}};
\end{tikzpicture}
};
\endxy
\big)$, which is a bimodule over $\Sym(\X)\cong \C[e_1(\X),e_2(\X)]$ where $|\X|=2$. As before we distinguish the left- and right-actions by using different sets of variables $\X',\X$ and write:
\[\Dec \big(
\xy
(0,0)*{
\begin{tikzpicture} [scale=.3,fill opacity=0.2]
	\draw[very thick, directed=.55] (1,1) -- (-1,1);
	\node [opacity=1]  at (1.5,1) {\tiny{$_{2}$}};
\end{tikzpicture}
};
\endxy
\big)= 
\frac{\Sym(\X' | \X)}{\langle \X' =\X \rangle} \cong \frac{\C[e_1(\X'),e_2(\X'),e_1(\X),e_2(\X)]}{\langle e_1(\X')-e_1(\X),e_2(\X')-e_2(\X) \rangle}\]

The following equation shows the complex computing the corresponding Hochschild homology, where we have indicated grading shifts by powers of $q$ and the slightly unnaturally looking variable $-a^{2}$. It is clear from the above description that the differentials are all zero.
\[\mathrm{HH}_*\left(\Dec \big(
\xy
(0,0)*{
\begin{tikzpicture} [scale=.3,fill opacity=0.2]
	\draw[very thick, directed=.55] (1,1) -- (-1,1);
	\node [opacity=1]  at (1.5,1) {\tiny{$_{2}$}};
\end{tikzpicture}
};
\endxy
\big) \right) 
= 
a^{-4}\mathrm{H}_*\left( \xy
(0,0)*{
\begin{tikzpicture} [scale=.3,fill opacity=0.2]
\draw[->] (3.5,0.9) to (6,1.4);
\draw[->] (3.5,-0.9) to (6,-1.4);
\draw[->] (15.5,1.4) to (17.5,0.9);
\draw[->] (15.5,-1.4) to (17.5,-0.9);
\node [opacity=1]  at (4.5,1.5) {\tiny $0$};
\node [opacity=1]  at (4.5,-1.5) {\tiny $0$};
\node [opacity=1]  at (16.5,1.5) {\tiny $0$};
\node [opacity=1]  at (16.5,-1.5) {\tiny $0$};
	\node [opacity=1]  at (21.5,0) {$a^4 \Dec \big(
\xy
(0,0)*{
\begin{tikzpicture} [scale=.3,fill opacity=0.2]
	\draw[very thick, directed=.55] (1,1) -- (-1,1);
	\node [opacity=1]  at (1.5,1) {\tiny{$_{2}$}};
\end{tikzpicture}
};
\endxy
\big)$	};
\node [opacity=1]  at (0,0) {$q^2\Dec \big(
\xy
(0,0)*{
\begin{tikzpicture} [scale=.3,fill opacity=0.2]
	\draw[very thick, directed=.55] (1,1) -- (-1,1);
	\node [opacity=1]  at (1.5,1) {\tiny{$_{2}$}};
\end{tikzpicture}
};
\endxy
\big)$	};
\node [opacity=1]  at (11,1.5) {$-a^2 q^2\Dec \big(
\xy
(0,0)*{
\begin{tikzpicture} [scale=.3,fill opacity=0.2]
	\draw[very thick, directed=.55] (1,1) -- (-1,1);
	\node [opacity=1]  at (1.5,1) {\tiny{$_{2}$}};
\end{tikzpicture}
};
\endxy
\big)$	};
\node [opacity=1]  at (11,-1.5) {$-a^2 \Dec \big(
\xy
(0,0)*{
\begin{tikzpicture} [scale=.3,fill opacity=0.2]
	\draw[very thick, directed=.55] (1,1) -- (-1,1);
	\node [opacity=1]  at (1.5,1) {\tiny{$_{2}$}};
\end{tikzpicture}
};
\endxy
\big)	$};
\end{tikzpicture}
};
\endxy \right )\]
The Hochschild homology of $\Dec \big(
\xy
(0,0)*{
\begin{tikzpicture} [scale=.3,fill opacity=0.2]
	\draw[very thick, directed=.55] (1,1) -- (-1,1);
	\node [opacity=1]  at (1.5,1) {\tiny{$_{2}$}};
\end{tikzpicture}
};
\endxy
\big)$ thus consists of four grading shifted copies of $\Dec \big(
\xy
(0,0)*{
\begin{tikzpicture} [scale=.3,fill opacity=0.2]
	\draw[very thick, directed=.55] (1,1) -- (-1,1);
	\node [opacity=1]  at (1.5,1) {\tiny{$_{2}$}};
\end{tikzpicture}
};
\endxy
\big)$ itself, whose underlying $\C$-space is isomorphic to a polynomial algebra with generators of internal degree $2$ and $4$. Thus, the corresponding Hilbert--Poincar\'e series (with respect to the Hochschild grading) has the expansion 
\[\mathrm{HP}\left (\mathrm{HH}_*\left(\Dec \big(
\xy
(0,0)*{
\begin{tikzpicture} [scale=.3,fill opacity=0.2]
	\draw[very thick, directed=.55] (1,1) -- (-1,1);
	\node [opacity=1]  at (1.5,1) {\tiny{$_{2}$}};
\end{tikzpicture}
};
\endxy
\big) \right)  \right ) = a^{-4} \frac{(a^4 - a^2 q^2- a^2 + q^2)}{(1-q^2)(1-q^4)} 
= a^{-4} \frac{(a^2-1)(a^2-q^2)}{(1-q^2)(1-q^4)}
\] 
Up to overall multiplication by a monomial, this agrees with the HOMFLY-PT polynomial of the $2$-labelled unknot.
\end{exa}

\subsection{HOMFLY-PT to $\glnn{N}$ spectral sequences}
\label{sec-ss}
Following Webster \cite{Web1}, we explain how the Hochschild homology functor can be replaced by an alternative functor in order to compute the $\Sigma$-deformed Khovanov--Rozansky homology instead of HOMFLY-PT homology from $\Dec{\LC{\cal{B}}{}{}}$. This functor takes the form 
\begin{equation}
WH^\Sigma_*(\cdot):=H_{tot}(\cdot \otimes \cal{Z}_{\mathbf{y},\mathbf{x}}).
\end{equation}
see Theorem~\ref{thm-Webster}. Here, the Koszul complex $ \cal{Z}_{\mathbf{x}}$ has been replaced by a \emph{Koszul matrix factorization} $\cal{Z}_{\mathbf{y},\mathbf{x}}$. We now recall some background on matrix factorizations and their relevance for $\glnn{N}$ link homologies.

\begin{defi} Let $\cal{R}$ be a commutative $\C$-algebra. A $\Z$-graded matrix factorization over $\cal{R}$ with potential $\phi\in \cal{R}$ is a $\Z$-graded $\cal{R}$-module $M$ together with a map $d_{tot}=d_+ +d_-\colon M \to M$ composed of differentials $d_\pm$ of degree $\pm 1$ such that $d_{tot}^2$ acts by multiplication by $\phi$. A morphism of matrix factorizations is a graded $\cal{R}$-module map which intertwines the respective $d_+$'s and $d_-$'s. Moreover, matrix factorizations over $\cal{R}$ admit a tensor product which is additive on potentials. We call the $\Z$-grading the \emph{horizontal} grading and denote it by $\mathrm{gr}_h$. Usually, the matrix factorizations we consider will have an additional internal polynomial grading denoted $\mathrm{gr}_q$. 
\end{defi}
 We use the shorthand notation $M^+ =(M, d_+)$ for $M$ considered as a chain complex with differential $d_+$. A matrix factorization with potential zero is also a chain complex with respect to $d_{tot}=d_+ + d_-$, which implies that $d_+$ and $d_-$ anti-commute. In this case, we write $H_{tot}(M) := H_*(M, d_{tot})$ for the total homology, but we could also consider just the positive homology with its induced $d_-^*$ differential $(H^+(M),d_-^*):=(H_*(M^+),d_-^*)$. 

\begin{defi}\label{def-KoszulMF} Let $\mathbf{a}=\{a_1,\dots,a_n\}$ and $\mathbf{b}=\{b_1,\dots,b_n\}$ be sequences of elements of $\cal{R}$, then the \emph{Koszul matrix factorization} of $\mathbf{a}$ and $\mathbf{b}$ is defined as:
\[\cal{Z}_{\mathbf{a},\mathbf{b}}:=\bigotimes_{i=1}^n \left(\cal{R} \mathrel{\mathop{\rightleftarrows}^{b_i}_{a_i}} \uwave{\cal{R}} \right)\]
This is a matrix factorization with potential $\sum_{i=1}^n a_ib_i$. If $\cal{R}$ is graded and the $b_i$ homogeneous, we always assume that degrees are shifted to make $d_+$ homogeneous of internal degree $2$. Note that then $\cal{Z}^+_{\mathbf{a},\mathbf{b}}=\cal{Z}_{\mathbf{b}}$. It is common to encode the differentials of $\cal{Z}_{\mathbf{a},\mathbf{b}}$ in a $n \times 2$ matrix with columns $\mathbf{a}$ and $\mathbf{b}$. Then the \emph{rows} of this matrix correspond to the tensor factors in the defining equation. 
\end{defi}

\begin{lem} \cite[Section 3.3, \emph{row operations}]{Ras1} \label{lem-rowops} Let $a_1,a_2,b_1,b_2,c \in \cal{R}$. Then there is an isomorphism of $\Z$-graded matrix factorizations: $\cal{Z}_{a_1,b_1} \otimes \cal{Z}_{a_2,b_2} \cong \cal{Z}_{a_1+c a_2,b_1}\otimes \cal{Z}_{a_2,b_2-c b_1}$.
\end{lem}

More generally, we can obtain isomorphisms of larger Koszul matrix factorizations by applying a sequence of such row operations on pairs of tensor factors in them. Recall that a sequence of elements $\mathbf{b}=\{b_1,\dots, b_k\}$ of $\cal{R}$ is \emph{regular} if for $1\leq i\leq k$ the element $b_i$ is not a zero-divisor in $\cal{R}/\la b_1,\dots b_{i-1}\ra$. It is easy to see that $H^+(\cal{Z}_{\mathbf{b}}) \cong \cal{R}/\la \mathbf{b}\ra$, concentrated in horizontal degree zero. A generalization of this fact is provided by the following result of Webster.

\begin{prop}\cite[Section 1.2]{Web1} \label{prop-simplify} Let $\cal{Z}_{\mathbf{a},\mathbf{b}}$ be a Koszul matrix factorization over $\cal{R}$ with potential $\phi$. Suppose that $\mathbf{b}$ is a regular sequence in $\cal{R}$. Let $M$ be a projective matrix factorization over $\cal{R}$ with potential $-\phi$. Then the chain map $\pi \colon \cal{Z}^+_{\mathbf{a},\mathbf{b}} \to H^+(\cal{Z}_{\mathbf{a},\mathbf{b}})\cong \cal{R}/\la \mathbf{b}\ra $ induces the following isomorphisms of $\cal{R}$-modules:
\begin{gather*}
H^+(\cal{Z}_{\mathbf{a},\mathbf{b}} \otimes M) \cong H^+(\cal{R}/ \la \mathbf{b}\ra \otimes M )\\
H_{tot}(\cal{Z}_{\mathbf{a},\mathbf{b}} \otimes M) \cong H_{tot}(\cal{R}/ \la \mathbf{b}\ra \otimes M )
\end{gather*}
\end{prop}

\begin{cor}\label{cor_plusminus} Under the assumptions of Proposition~\ref{prop-simplify}, there is an isomorphism of chain complexes of $\cal{R}$-modules:
\[(H^+(\cal{Z}_{\mathbf{a},\mathbf{b}} \otimes M),d_-^*) \cong (H^+(\cal{R}/ \la \mathbf{b}\ra \otimes M ),d_-^*)\]
\end{cor}
\begin{proof} The chain map $\pi \colon \cal{Z}^+_{\mathbf{a},\mathbf{b}} \to H^+(\cal{Z}_{\mathbf{a},\mathbf{b}})\cong \cal{R}/\la \mathbf{b}\ra $ also intertwines the negative differential $d_-$ on $\cal{Z}_{\mathbf{a},\mathbf{b}}$ with the trivial negative differential on $\cal{R}/\la \mathbf{b}\ra$, i.e. it is a morphism of matrix factorizations. Then $\pi\otimes 1\colon \cal{Z}_{\mathbf{a},\mathbf{b}} \otimes M \to \cal{R}/ \la \mathbf{b}\ra \otimes M$ also intertwines the negative differentials. The induced map on positive homology is an isomorphism on modules which intertwines the induced differentials $d_-^*$, i.e. it is an isomorphism of chain complexes.
\end{proof}

\begin{defi} Let $\cal{Z}_{\mathbf{a},\mathbf{b}}$ be as in Proposition~\ref{prop-simplify} and $M_1$ a projective matrix factorization over $\cal{R}$ (not necessarily with potential $-\phi$), then we consider $\cal{R}/ \la \mathbf{b}\ra \otimes M_1$ as a matrix factorization over $\cal{R}/ \la \mathbf{b}\ra$ and call it a \emph{simplification} of $\cal{Z}_{\mathbf{a},\mathbf{b}} \otimes M_1$.
\end{defi} 

Proposition~\ref{prop-simplify} and Corollary~\ref{cor_plusminus} imply that Koszul matrix factorizations appearing as tensor factors in a (projective) matrix factorization with potential zero can be replaced by their simplifications without changing the $H^+$ homology with its induced negative differential or the $H_{tot}$ homology or the $\cal{R}$-module structures on them. 

We now recall how to assign Koszul matrix factorizations to webs, which can be used to compute $\glnn{N}$ link homologies, see \cite{Wu1, Wu2}\footnote{Note however that we use $\Z$-graded matrix factorizations instead of $\Z_2$-graded ones.}. For this, let $W$ be a web and associate to every edge in $W$ with label $l$ an alphabet $\Y$ with $|\Y|=l$. The alphabets adjacent to the left and right boundary we again denote by $\X_i$ and $\X'_i$. We consider the ring $\cal{R}_W$ of polynomials separately symmetric in these alphabets. For a polynomial $T(X)=\sum_{i=0}^k c_i X^i \in \C[X]$ we set $T(\X) = \sum_{i=0}^k c_i p_i(\X)$ where $p_i(\X)$ denotes the $i^{th}$ power sum symmetric polynomial in the alphabet $\X$. Recall that we denote by $P(X)\in \C[X]$ the monic polynomial with root multiset $\Sigma$ and let $Q(X)\in \C[X]$ be such that $Q'(X)=(N+1)P(X)$ and $Q(0)=0$. 

\begin{defi} To every web $W$ we associate a Koszul matrix factorization $\Mf(W)$ over $\cal{R}_W$ with potential $Q(\bigsqcup \X'_j)-Q(\bigsqcup \X_j)$. We first describe it in special cases: 
\begin{itemize}
\item If $W$ is a single strand of label $a$ without incident vertices, then $\Mf(W)$ has positive sequence $e_j(\X'_1 ) - e_j(\X_1)$ for $1\leq j \leq a$ and the negative sequence is chosen\footnote{It is easy to see that this is possible \cite[Section 3.2]{Wu2}, the specific choice is eirrelevant for our discussion.} to generate the potential $Q(\X'_1)-Q(\X_1)$.
\item If $W$ is a merge web of incoming strands of thickness $a$ and $b$ then $\Mf(W)$ has the positive sequence consisting of the entries $e_j(\X'_1 ) - e_j(\X_1 \sqcup \X_2)$ for $1\leq j \leq a+b$ and the negative sequence is chosen to generate the potential $Q(\X'_1)-Q(\X_1\sqcup \X_2)$. 

\item Analogously, the Koszul matrix factorization associated to the split web has the positive sequence with entries $e_j(\X'_1 \sqcup \X'_2)-e_j(\X_1 )$ and an appropriate negative sequence.
\end{itemize}
For any other web $W$, the Koszul matrix factorization $\Mf(W)$ is defined over $\cal{R}_W$ and consists of all the rows appearing in the Koszul factorizations associated to the merge and split vertices and non-interacting strands in $W$. 
\end{defi}

\begin{lem}\label{lem-singlegrad} Let $W$ be a web in $\LC{\cal{B}}{}{}$. Then $\Mf(W)$ simplifies to $H^+(\Mf(W))\cong \Dec(W)$  and this simplification is concentrated in horizontal degree zero. 
\end{lem}
\begin{proof} It is easy to check that the positive sequence  of $\Mf(W)$ is regular, and so the positive homology is concentrated in degree zero and isomorphic to the quotient of the ground ring by the ideal generated by the elements in the sequence. Identifying the ring $\cal{R}_W$ with the possible placements of decorations on the identity foam on $W$, the entries of the positive sequence generate the same ideal as relations~\eqref{eqn-decmigr}. This identifies $H^+(\Mf(W))$ with $\Dec(W)$. 
\end{proof}

\begin{rem}
Lemma~\ref{lem-singlegrad} shows that we can view $\Dec(W)$ as a matrix factorization over $\cal{R}_{\cal{B}}\otimes \cal{R}_\cal{B}$ with potential $\phi := Q(\bigsqcup \X_j)-Q(\bigsqcup \X'_j)$ and trivial differentials. To see this directly, note that since $H^+(\Mf(W))$ is concentrated in horizontal degree zero, we have $\Dec(W) \cong \Mf_0(\cal{B})/\mathrm{Im}~d_+$. Then $\phi\,\mathrm{id}_{\Mf_h(\cal{B})}= d_+ d_-$ and so the element $\phi$ is in the image of $d_+$ and acts by zero.
\end{rem}

In \cite[Section 4.4]{RW} Rose and the author show that there exists a 2-functor from $\Foam{N}^\Sigma$ to a 2-category of ($\Z_2$-graded) matrix factorizations with potential modelled on $Q$. Under this 2-functor, the web $W$ is sent to the matrix factorization $\Mf(W)$ (with the horizontal grading taken modulo $2$). More generally, for a braid $\cal{B}$, the complex $\LC{\cal{B}}{\Sigma}{}$ is sent to a \emph{cube of resolutions} complex $\Mf(\cal{B})$ of matrix factorizations and morphisms between them. This complex agrees up to isomorphism and grading conventions with the complex used in Wu's construction of (deformed) colored $\glnn{N}$ link homologies in \cite{Wu1,Wu2}. We will not describe the differentials in this complex more explicitly, as Lemma~\ref{lem-singlegrad} implies that it suffices to work with $\Dec(\LC{\cal{B}}{}{})$, see Theorem~\ref{thm-Webster}. \comm{Why? The $\Z$-graded matrix factorizations are free resolutions of Soergel bimodules. Thus, maps of the modules induce unique (up to homotopy) maps between the resolutions.}

We fix the notation $\cal{Z}_{\mathbf{y},\mathbf{x}}:= \Mf(1^{op}_{\mathbf{o}_\cal{B}})$ for the Koszul matrix factorization assigned to the oppositely oriented identity web on $\mathbf{o}_\cal{B}$, i.e. with with alphabets $\X_j$ at the left and alphabets $\X'_j$ at the right boundary. It is defined over $\cal{R}_{\cal{B}}\otimes \cal{R}_{\cal{B}}$, has potential $Q(\bigsqcup_j \X_j)-Q(\bigsqcup_j\X'_j)$ and the positive sequence $\mathbf{x}$ consists of entries $e_l(\X_j)-e_l(\X'_j)$ for $1\leq j \leq m$ and $1\leq l \leq o_j$.

\begin{defi} 
The matrix factorization associated to the labelled link diagram $\cal{L}$ given as closure of a balanced labelled braid diagram $\cal{B}$ is a complex of Koszul matrix factorizations with potential zero defined as
\[\Mf(\cal{L}):= \Mf(\cal{B})\otimes \cal{Z}_{\mathbf{y},\mathbf{x}}.\] 
 Further, we define the functor
\[WH^\Sigma_*(\cdot) := H_{tot}(\cdot \otimes \cal{Z}_{\mathbf{y},\mathbf{x}})\] from the category of matrix factorizations over $\cal{R}_{\cal{B}}\otimes \cal{R}_{\cal{B}}$ with potential $Q(\bigsqcup_j \X'_j)-Q(\bigsqcup_j\X_j)$ to the category of (bi-)graded $\cal{R}_{\cal{B}}$-modules. 
\end{defi}

The functor $WH^\Sigma_*(\cdot)$ is a matrix factorization analogue of the Hochschild homology functor $HH_*(\cdot)$ which depends on the multiset $\Sigma$, see Lemma~\ref{lem-Koszul}. It has been introduced by Webster in \cite{Web1}. Clearly, $WH^\Sigma_*(\cdot)$ extends to chain complexes of matrix factorizations, in which case we denote the induced \emph{vertical differential} by $d_v^*$ and the new homological, a.k.a.~\emph{vertical grading} by $\mathrm{gr}_v$.

\begin{thm}\label{thm-Webster} Let $\cal{B}$ be a balanced labelled braid diagram and $\cal{L}$ its closure. Then we have:
\begin{equation}
\label{eqn-Webster}
\LH{\cal{L}}{\Sigma}{} \cong H_*(H_{tot}(\Mf(\cal{L})),d_v^*) \cong H_*(WH^\Sigma_*(\Dec(\LC{\cal{B}}{}{})),d_v^*)
\end{equation}
\end{thm}
\begin{proof}
First of all, we note that~\cite[Theorem 2]{RW} identifies $\LH{\cal{L}}{\Sigma}{} $ with Wu's $\Sigma$-deformed Khovanov Rozansky $\glnn{N}$ homology, which can be computed as
\[\LH{\cal{L}}{\Sigma}{} \cong H_*(H_{tot}(\Mf(\cal{L})),d_v^*)\]
where the inner homology is with respect to the total matrix factorization differential $d_{tot}=d_++d_-$ and the outer homology is with respect to the induced vertical differential $d_v^*$. This establishes the first isomorphism. For the second isomorphism we use that the vertical complex $\Mf(\cal{B})$ simplifies to $H^+(\Mf(\cal{B}))\cong \Dec(\LC{\cal{B}}{}{})$. 
\[H_{tot}(\Mf(\cal{L})) = H_{tot}(\Mf(\cal{B})\otimes \cal{Z}_{\mathbf{y},\mathbf{x}})\cong H_{tot}(\Dec(\LC{\cal{B}}{}{}) \otimes \cal{Z}_{\mathbf{y},\mathbf{x}}) = WH^\Sigma_*(\Dec(\LC{\cal{B}}{}{})\]
and the second claimed isomorphism follows after taking $d_v^*$-homology.
\end{proof}

\begin{rem}
\label{rem-gradings} We need to comment on the behaviour of gradings under the identifications in~\eqref{eqn-Webster}. Up to a global shift, the homological grading of $\LH{\cal{L}}{\Sigma}{}$ simply agrees with the homological grading on $H_*(WH^\Sigma_*(\Dec(\LC{\cal{B}}{}{})),d_v^*)$. In the undeformed case where $\Sigma=\{0,\dots,0\}$, we also would like to recover the quantum grading of $\LH{\cal{L}}{\Sigma}{} \cong \LH{\cal{L}}{\glnn{N}}{}$ on the right-hand side of~\eqref{eqn-Webster}. Note however, that $d_{tot}=d_+ + d_-$ is not even homogeneous with respect to the internal polynomial grading $gr_q$ on $H_*(WH^\Sigma_*(\Dec(\LC{\cal{B}}{}{})),d_v^*)$. Instead one considers the new grading $\mathrm{gr}_N:= \mathrm{gr}_q +  (N-1) \mathrm{gr}_h$, which is preserved by $d_{tot}$ and thus descends to homology, where it can be shown to correspond (up to a global shift) to the quantum grading on the left-hand side of~\eqref{eqn-Webster}, see Rasmussen's~\cite[Proposition 3.12]{Ras1}.
\end{rem}

\begin{thm}\label{thm-Hss} There is a spectral sequence 
\[\LHH{\cal{L}}{} \quad \rightsquigarrow \quad \LH{\cal{L}}{\Sigma}{}\]
whose $k^{th}$ differential lowers the horizontal and homological grading by $k$ and $k-1$ respectively. If $\Sigma=\{0,\dots,0\}$, it increases the internal degree by $2k N$.
\end{thm}

\begin{proof}[Sketch proof following Rasmussen \cite{Ras1}] The matrix factorization $\Mf(\cal{L})$ associated to $\cal{L}$ is a equipped with the three anti-commuting differentials $d_+$, $d_-$ and $d_v$. Its $H^+$-homology is a triply-graded chain complex (internal grading, horizontal grading, homological grading), with induced differentials $d_-^*$ and $d_v^*$, which lower horizontal and homological grading by one respectively, but preserve the other grading. The horizontal filtration on this double complex induces a spectral sequence with $E_0$ page $(H^+(\Mf(\cal{L})),d_v^*)$, which is isomorphic to $(HH_*(\Dec(\LC{\cal{B}}{}{})),d_v^*)$ by Proposition~\ref{prop-simplify}. Consequently, the $E_1$ page is isomorphic to $H_*(HH_*(\Dec(\LC{\cal{B}}{}{})),d_v^*) = \LHH{\cal{L}}{}$. 
The spectral sequence converges to the total homology  $H_*((H^+(\Mf(\cal{L})),d_-^*+d_v^*)$ and it remains to identify this with $\LH{\cal{L}}{\Sigma}{}=H_*((H_{tot}(\Mf(\cal{L})),d_v^*)$.
The proof of this is split into two steps. First, we need \[H_*((H^+(\Mf(\cal{L})),d_-^*+d_v^*) \cong H_*(H_*((H^+(\Mf(\cal{L})),d_-^*),d_v^*),\] which can be proved exactly as in \cite[Proposition 5.10]{Ras1}. And second \[ H_*(H_*((H^+(\Mf(\cal{L})),d_-^*),d_v^*) \cong H_*(H_{tot}(\Mf(\cal{L})),d_v^*),\] which follows as in \cite[Lemma 5.11, Lemma 5.12, Proposition 5.13]{Ras1}. The proofs of these isomorphisms, however, use a critical fact, which is a generalization of Rasmussen's \cite[Proposition 5.8 and Corollary 5.9]{Ras1} and which we prove in the separate Lemma~\ref{lem-singlehor}.

As regards gradings, the $k^{th}$ differential in the spectral sequence lowers horizontal degree $\mathrm{gr_h}$ by $k$ and homological degree $\mathrm{gr_v}$ by $k-1$ (and it raises internal degree $\mathrm{gr_q}$ by $2k N$ in the undeformed case). In particular, all differentials in the spectral sequence preserve the quantity $\mathrm{gr_-}=\mathrm{gr_v}-\mathrm{gr_h}$ (and $\mathrm{gr_N'}= \mathrm{gr_q} +2 N \mathrm{gr_h}$), inducing a grading on the $E_\infty$ page. Careful comparison of gradings as in \cite[Proposition 5.14]{Ras1} shows that this agrees with homological grading (and quantum grading) on $\mathrm{KhR}^{\Sigma}(\cal{K})$ (and $\mathrm{KhR}^{\glnn{N}}(\cal{K})$ respectively) as identified in Remark~\ref{rem-gradings}.
\end{proof}

\begin{lem} $H_*(H^+(\Mf(\cal{L})),d_-^*)$ is concentrated in a single horizontal grading. 
\label{lem-singlehor}
\end{lem}
\begin{proof}
An analogue of Rasmussen's \cite[Corollary 5.9]{Ras1} shows that it suffices to prove the claim for a matrix factorization $\Mf(L)$ in place of $\Mf(\cal{L})$, where $L$ is the closure (as in \eqref{eqn:braidclosure}) of a web $W$ appearing in $\LC{\cal{B}}{}{}$. The key ingredient in Rasmussen's proof of this fact in the uncolored setting in \cite[Proposition 5.8]{Ras1} is a recursive algorithm for computing $H_*(H^+(\Mf(L)),d_-^*)$, see \cite[Sections 4.2 and 4.3]{Ras1}. 

A suitable replacement for this algorithm in the colored case is an analogue of the annular web evaluation algorithm of Queffelec and Rose~\cite[Proposition 5.1 and Lemma 5.2]{QR2}, which we review in detail in Section~\ref{sec-app}. In its simplest, decategorified version, the algorithm expands the annular closure of a web in $\NWeb$ as a linear combination of the annular closures of identity webs---which can be thought of as collections of nested labelled essential circles in the annulus. It does so by successively rewriting annular webs as linear combinations of annular webs of lower complexity, using two types of relations: relations that already hold for webs in $\NWeb$, namely \textit{rung combination}, \textit{the Frobenius relations} and \textit{square switch relations} from Lemma~\ref{lem-annulareval}), as well as the \textit{trace relations} which witness that the annular closure $\overline{W_1W_2}$ of a composite web $W_1W_2$ is isotopic to the annular closure $\overline{W_2W_1}$ of $W_2W_1$.

Using arguments as in the proof of Lemma~\ref{lem-straighten}, we can mimick the the annular web evaluation on the categorified level to exhibit $H_*(H^+(\Mf(L)),d_-^*)$ as a direct summand of a direct sum of triply-graded vector spaces of the form $H_*(H^+(\Mf(C)),d_-^*)$, where $C$ is a collection of nested labelled essential circles. Indeed, it follows directly from the definition of $\Mf(\cdot)$ and Proposition~\ref{prop-simplify} that $\Mf(\overline{W_1W_2})\cong \Mf(\overline{W_2W_1})$, i.e. the functor $\Mf(\overline{\cdot})$ is trace-like, and so is $H_*(H^+(\Mf(\cdot)),d_-^*)$. To see that $H_*(H^+(\Mf(\cdot)),d_-^*)$ also satisfies categorified analogues of the other required relations, we write
\[(H^+(\Mf(L)),d_-^*) = (H^+(\Mf(W)\otimes \cal{Z}_{\mathbf{y},\mathbf{x}}),d_-^*) \overset{\text{Cor.}~\ref{cor_plusminus}}{\cong} (H^+(\Dec(W)\otimes \cal{Z}_{\mathbf{y},\mathbf{x}}),d_-^*)  \] 
Now it follows that rung combination, the Frobenius relations and square switch relations hold here because they hold for singular Soergel bimodules $\Dec(\cdot)$, as verified e.g. in \cite[Section 4.4]{RW}. More precisely, an isomorphism of Soergel bimodules (i.e. of $\Ru\otimes \Ru$-modules) induces an isomorphism of matrix factorizations after tensoring with $ \cal{Z}_{\mathbf{y},\mathbf{x}}$ over $\Ru\otimes \Ru$. After taking positive homology, we get an isomorphism of complexes with respect to the induced negative differentials.

 We conclude that $H_*(H^+(\Mf(L)),d_-^*)$ is (isomorphic to) a direct summand of a direct sum of objects of the form $H_*(H^+(\Mf(C)),d_-^*)$, where $C$ is a collection of nested labelled essential circles. Such a summand is concentrated in the horizontal grading equal to the negative of the sum of the labels on the circles in $C$, which is equal (to $\sum o_i$) for all $C$ appearing in the annular evaluation algorithm. Thus $H_*(H^+(\Mf(L)),d_-^*)$ and then also $H_*(H^+(\Mf(\cal{L})),d_-^*)$ are concentrated in this grading as well. 
\end{proof}

\begin{rem} The categorified annular web evaluation algorithm of Queffelec--Rose \cite{QR2}, which is essentially independent of $N$, suggests a possible alternative approach to colored HOMFLY-PT homology as an invariant of annular links. This might lead to a new and more illuminating proof of Theorem~\ref{thm-Hss}.
\end{rem}

\subsection{Reduced colored HOMFLY-PT homology}
\label{sec-reducedHomfly}

In this section, we introduce a reduced version of colored HOMFLY-PT homology and relate it to reduced colored $\glnn{N}$ homology. As in the case of $\glnn{N}$ homology, we have to choose a link component at which we want to reduce. To this end, we consider a balanced labelled braid diagram $\cal{B}$---which closes to an oriented, labelled link diagram $\cal{L}$---together with a marked right boundary point with associated alphabet $\X_i$ of size $k=|\X_i|$. In the following, we assume that all strands of $\cal{B}$ have labels $\geq k$ and that $\cal{B}$ does not split into a disjoint union of two braids. 

\begin{rem}
Both of these assumptions are automatically satisfied for braid diagrams representing knots. Unfortunately, however, they currently prevent the treatment of colored HOMFLY-PT homologies of links reduced at a component of non-minimal label, for which interesting \emph{color-shifting properties} have been conjectured in \cite{GNSSS}.
\end{rem}

\begin{defi}\label{defi_Vu}
We cut the labelled planar graph underlying the diagram $\cal{B}$ into labelled arcs between crossings, and we assign alphabets $\X'_1,\dots,\X'_m$ and $\X_1,\dots,\X_m$ to boundary arcs on the left and the right respectively and alphabets $\Y_l$ of the appropriate size to the internal arcs. Note that every such arc corresponds to an edge in any web appearing in $\LC{\cal{B}}{}{}$. Let 
\[\cal{S}_\cal{B}:=\Sym(\X_1|\dots|\X_m|\X'_1|\dots |\X'_m|\Y_1|\dots)\]
be the graded ring of polynomials separately symmetric in these variables. Let 
\[\Vu := \Su / \la e_j(\W_1'\sqcup \W_2')-e_j(\W_1\sqcup \W_2)|1\leq j\leq k \ra\]
where $(\W_1',\W'_2,\W_1,\W_2)$ ranges over all tuples of alphabets adjacent to a crossing site, with the dashed alphabets associated to outgoing edges and the other ones to incoming edges. Finally, we define
\[\Tu := \Ru\otimes\Ru / \la e_j( \bigsqcup_{l=1}^m\X_l')-e_j(\bigsqcup_{l=1}^m\X_l)|1\leq j\leq k \ra\]
where we identify $\Ru\otimes\Ru \cong \Sym(\X_1|\cdots|\X_m|\X'_1|\cdots|\X'_m)$.

\end{defi}

\begin{lem} \label{lem-symthrough}
$\Vu$ is a $\Tu$-module.
\end{lem}
\begin{proof} $\Vu$ is a  $\Ru\otimes\Ru$-module and and it is easy to check (e.g. by induction on the number of crossings in $\cal{B}$) that the relations $e_j( \bigsqcup_{l=1}^m\X'_l)=e_j(\bigsqcup_{l=1}^m\X_l)$ hold in $\Vu$ for $1\leq j\leq k$. 
\end{proof}

\begin{lem}\label{lem-Vufree} $\Tu$ and $\Vu$ are isomorphic to polynomial rings over $\C$ and $\Vu$ is free over $\Tu$. 
\end{lem}
\begin{proof}  The unital $\C$-algebra $\Tu$ has a presentation with commuting generators $e_j(\X_l)$ and $e_j(\X'_l)$ for $1\leq l \leq m$ and $1\leq j \leq o_l = |\X_l|$  and one relation for each $1\leq j\leq k$, which we rewrite as follows:
\[ 0= e_j( \bigsqcup_{l=1}^m\X_l')-e_j(\bigsqcup_{l=1}^m\X_l) = e_j(\X_i') - \left( e_j(\bigsqcup_{l=1}^m\X_l) - \sum_{x=0}^{j-1} e_x(\X_i')e_{j-x}( \bigsqcup_{l\neq i}\X_l')  \right)\]
Thus, the generator $e_j(\X_i')$ can be expressed in terms of other generators. Under this substitution, we obtain a presentation for $\Tu$ with one generator and one relation eliminated. By iteration, we arrive at a presentation of $\Tu$ without any relations and with all generators $e_j(\X_i')$ for $1\leq j\leq k$ eliminated. In other words, the ring homomorphism, which is given by the composition 
\begin{equation}
\label{eqn_iso1}
\Sym(\X_1|\cdots|\X_m|\X'_1|\cdots|\X'_{i-1}|\X'_{i+1}|\cdots|\X'_{m})\hookrightarrow \Ru\otimes\Ru \to \Tu
\end{equation} 
of the canonical inclusion and the defining quotient map, is an isomorphism.

In Definition~\ref{defi_Vu}, $\Vu$ is given by a similar presentation with generators associated to arcs in the braid diagram $\cal{B}$ and $k$ relations for every crossing in $\cal{B}$. For a crossing with adjacent alphabets $(\W_1',\W'_2,\W_1,\W_2)$, we can choose one of these alphabets, call it $\W$, and eliminate the generators $e_j(\W)$ for $1\leq j\leq k$ in pairs with the relations $e_j(\W_1'\sqcup \W_2')-e_j(\W_1\sqcup \W_2)$ for $1\leq j\leq k$. This is possible since we assume that $|W|\geq k$. To eliminate the relations associated to all crossings, we need to choose different alphabets $\W$ at every crossing. It is easy to see that this is possible. In fact, since $\cal{B}$ is a non-split braid diagram, we can arrange to eliminate generators $e_j(\X'_i)$ for $1\leq j \leq k$ first and only generators in internal alphabets $\Y_l$ thereafter. The surviving generators in internal alphabets generate a (possibly trivial) polynomial ring $R$ and the canonical composition
\begin{equation}
\label{eqn_iso2}
\Sym(\X_1|\cdots|\X_m|\X'_1|\cdots|\X'_{i-1}|\X'_{i+1}|\cdots|\X'_{m})\otimes_\C R \hookrightarrow \Su \to \Vu
\end{equation} is an isomorphism. Finally, note that the isomorphisms in \eqref{eqn_iso1} and \eqref{eqn_iso2} are defined as identity maps on representatives. Thus $\Vu\cong \Tu\otimes_\C R$ is free over $\Tu$.
\end{proof}

For the following we write $\cal{Z}_{\mathbf{y},\mathbf{x}} \cong \cal{Z}_{\mathbf{y}_i,\mathbf{x}_i} \otimes_\C \cal{Z}_{\overline{\mathbf{y}},\overline{\mathbf{x}}}$ where $\cal{Z}_{\mathbf{y}_i,\mathbf{x}_i}$ is the Koszul matrix factorization over $\Sym(\X_i|\X'_i)$ with respect to the subsequences of $\mathbf{x}$ and $\mathbf{y}$, which only consist of polynomials involving $\X_i$ and $\X'_i$. $\overline{\mathbf{x}}$ and $\overline{\mathbf{y}}$ are the complementary subsequences, forming the Koszul complex $\cal{Z}_{\overline{\mathbf{y}},\overline{\mathbf{x}}}$ over the subring of $\Ru\otimes\Ru$ generated by all alphabets except $\X_i$ and $\X'_i$. 

\begin{lem}\label{lem-symxi} $H^+(\Tu \otimes \cal{Z}_{\overline{\mathbf{y}},\overline{\mathbf{x}}}) \cong \Tu/ \la \X_j=\X'_j| j\neq i\ra \cong \Ru\otimes \Ru / \la\X_j = \X'_j|\forall j\ra$.
\end{lem}
\begin{proof} From the proof of Lemma~\ref{lem-Vufree} we see that  $\overline{\mathbf{x}}$ is a regular sequence in $\Tu$ and it clearly generates the ideal $\la \X_j=\X'_j| j\neq i\ra$. Thus we get the first isomorphism. For the second isomorphism note that after identifying $\X_j =\X'_j$ we can rewrite the defining relations in $\Tu$ as
\[\sum_{a+b=j} (e_a(\X'_i)-e_a(\X_i) e_b( \bigsqcup_{l\neq i}\X_l)= 0,\quad 1\leq j\leq k\] which impose $e_j(\X'_i)=e_j(\X_i)$ for $1\leq j\leq k$, or in shorthand $\X'_i=\X_i$. 
\end{proof}

In the following, we again consider complexes of matrix factorizations associated to braid and link diagrams. Throughout we work in the undeformed setting $\Sigma = \{0,\dots, 0\}$ and omit $\Sigma$ from the notation. Also note that from now on all matrix factorizations carry the internal polynomial grading $\mathrm{gr}_q$ and all differentials are homogeneous.

\begin{lem} \label{lem-quotientcx} $\Mfu(\cal{B})$ has a simplification $\UMf(\cal{B})$ over $\Vu$. 
\end{lem}
\begin{proof} If $\cal{B}$ has no crossings, i.e. is a single strand, we define $\UMf(\cal{B}):=\Sym(\X_i|\X'_i)/\la \X_i=\X'_i\ra$ which is a simplification by Lemma~\ref{lem-singlegrad}. Otherwise, let $W$ be a (non-trivial) web appearing in $\LC{\cal{B}}{}{}$. We will apply row operations to replace the Koszul matrix factorization $\Mfu(W)$ by an isomorphic matrix factorization whose positive sequence contains entries that impose the defining relations in $\Vu$. It suffices to consider one half of a crossing web appearing in \eqref{eqn-crossingcx}:
\begin{equation}
\label{eqn-crossingresol2}\xy
(0,0)*{
\begin{tikzpicture} [scale=.5,fill opacity=0.2]
	%bottom web
	\draw[very thick, directed=.55]  (2,0) to (1.25,0);
	\draw[very thick]  (1.25,0) to (0.5,0);
	\draw[very thick]  (1.25,1) to (0.5,1);
	\draw[very thick, directed=.55]  (2,1) to (1.25,1);
	\draw[very thick, directed=.55]  (0.5,1) to (-0.25,1);
	\draw[very thick, directed=.55]  (0.5,0) to (-0.25,0);
	\draw[very thick, directed=.55] (1.25,1) to [out=180,in=0](0.5,0) ;
\node[opacity=1] at (0,1.5) {$_{\Y_1}$};
	\node[opacity=1] at (0,-0.5) {$_{\Y_2}$};
	\node[opacity=1] at (1.4,0.45) {\tiny $\A$};
		\node[opacity=1] at (1.7,1.5) {$_{\X_1}$};
	\node[opacity=1] at (1.7,-0.5) {$_{\X_2}$};
\end{tikzpicture}
};
\endxy
\end{equation}
The Koszul matrix factorization associated to this web has positive entries $e_a(\Y_1 \sqcup \A)-e_a(\X_1)$ for $1\leq a \leq |\X_1|$ and $e_b(\Y_2)- e_b(\X_2\sqcup \A)$  for $1\leq b \leq |\Y_2|=|\X_2|+|\A|$, which includes the case $\A=\emptyset$. Now it is easy to see that a sequence of row operations as in Lemma~\ref{lem-rowops} can be used to iteratively replace the first $k$ positive entries $e_a(\Y_1 \sqcup \A)-e_a(\X_1)$ by $e_a(\X_1\sqcup\X_2)-e_a(\Y_1\sqcup\Y_2)$ for $1\leq a \leq k$. Here we again use that all labels are $\geq k$. The new rows contain precisely the generators of the defining ideal of $\Vu$ as positive entries. Since they form a regular sequence, we apply Proposition~\ref{prop-simplify} and denote the result by $\UMf(\cal{B})$.\comm{We will replace $k$ positive entries associated to the splitting vertex by positive entries $e_a(\Y_1 \sqcup \Y_2)-e_a(\X_1 \sqcup \X_2)$ for $1\leq a \leq k$. For $a=1$ we obtain this by adding the degree $1$ relation of the merge vertex to the degree $1$ relation of the split vertex. We now describe the induction step $a-1 \to a$:
\begin{gather*}
e_a(\X_1\sqcup A)-e_a(\Y_1) + \sum_{c+d=a, c<a} e_c(\X_2)(e_d(\X_1\sqcup A)-e_d(\Y_1)) = e_a(\X_1\sqcup\X_2\sqcup\A)-e_a(\Y_1\sqcup\Y_2\sqcup\A)\\ = e_a(\X_1\sqcup\X_2)-e_a(\Y_1\sqcup\Y_2) + \sum_{c+d=a, c<a}   (e_c(\X_1\sqcup\X_2)-e_c(\Y_1\sqcup\Y_2))e_d(\A)   \end{gather*}
In the first line we add multiples of lower order terms (which can be obtained as linear combinations of the new entries) to the old degree $a$ entry $e_a(\X_1\sqcup A)-e_a(\Y_1)$. In the second line we can subtract multiples of new entries to obtain $e_a(\X_1\sqcup\X_2)-e_a(\Y_1\sqcup\Y_2)$. This also shows that the old degree $a$ entry can be obtained as linear combination of new degree $\leq a$ entries.}
\end{proof}

As a Koszul matrix factorization, $\UMf(\cal{B})$ is free over $\Vu$, so the proof of Lemma~\ref{lem-Vufree} together with Proposition~\ref{prop-simplify} imply that $\UMf(\cal{B})\otimes \cal{Z}_{\overline{\mathbf{y}},\overline{\mathbf{x}}}$ simplifies to $\UMf(\cal{B})/\la\X_j=\X'_j|j\neq i\ra$, which is a $\Sym(\X_i|\X'_i)/\la\X_i=\X'_i\ra$-module by Lemma~\ref{lem-symxi}. So we write it more compactly as $\UMf(\cal{B})/\la\X_j=\X'_j\ra$. Further tensoring this with  $\cal{Z}_{\mathbf{y}_i,\mathbf{x}_i}$ is equivalent to tensoring with $\cal{Z}_{\pi(\mathbf{y}_i),0}$, where $\pi\colon \Sym(\X_i|\X'_i)\to\Sym(\X_i|\X'_i)/\la\X_i=\X'_i\ra$ is the obvious quotient map. We summarize:

\begin{cor}
\label{cor-umf} $\Mfu(\cal{L})$ has a simplification of the form:
\begin{equation}
\label{eqn-umf}
 \Mfd(\cal{L}):= \UMf(\cal{B})/ \la \X'_j = \X_j\ra \otimes \cal{Z}_{\pi(\mathbf{y}_i),0} 
 \end{equation}
\end{cor}

\noindent Further, note that we have an isomorphism of $\Sym(\X_i)$-modules:
\[H_{tot}(\cal{Z}_{\pi(\mathbf{y}_i),0})\cong H_{tot}\left( \frac{\Sym(\X_i|\X'_i)}{\la\X_i=\X'_i\ra}\otimes_{\Sym(\X_i|\X'_i)}  \cal{Z}_{\mathbf{y}_i,\mathbf{x}_i}\right) = \LH{\bigcirc^k}{\glnn{N}} \cong H_k^N\]
So, the entries of $\pi(\mathbf{y}_i)\in \Sym(\X_i)$ are generators for the defining ideal of $H_k^N$ in $\Sym(\X_i)$.

\begin{defi} For a marked, labelled link diagram $\cal{L}$ given as the closure of a braid diagram $\cal{B}$ as specified in the introduction to this section, we define the \emph{reduced complex} associated to $\cal{L}$ as:
\[\overline{\Mfu(\cal{L}(i))}:= \UMf(\cal{B}) / \la \X_i=0,\X_j=\X'_j\ra\]
\end{defi}

\begin{rem} \label{rem-z2}
In order to prove the following two results we need to take a short technical detour. In their original work, Khovanov--Rozansky \cite{KR1} and Wu \cite{Wu1} work with categories of $\Z_2$-graded matrix factorizations. The matrix factorizations they associate to webs agree with the ones in this paper after reducing the $\Z$-grading modulo $2$ and, in particular, their total homologies agree. However, the advantage of $\Z_2$-graded matrix factorizations is that they admit a good notion of homotopy equivalence. Working in the corresponding homotopy category, Khovanov--Rozansky and Wu have proved matrix factorization analogues of many of the web isomorphisms that exist in the foam 2-category $\grFoam{N}^\bullet$. More generally, thanks to the 2-functor between $\foam{N}$ and Soergel bimodules, which was studied in \cite[Section 4.4]{RW}, we now know that all web isomorphisms in $\grFoam{N}^\bullet$ imply the analogous relations in the corresponding homotopy 2-category of matrix factorizations. In the following lemma, we will use matrix factorization versions of the web relations employed in web evaluation algorithm from Proposition~\ref{prop-annulareval}, and we will use that these relations preserve the module structure with respect to alphabet assigned to the marked edge. 
\end{rem}

\begin{lem} \label{lem-circlecx}
There is a homotopy equivalence of complexes over the homotopy category of $\Z_2$-graded matrix factorizations over $\Sym(\X_i|\X'_i)$
\[ (\Mfu(\cal{L}),d_v) \rightarrow (C(\cal{L}),d_v),\] where the target is a chain complex whose chain groups are direct sums of terms $\Mfu(\bigcirc^k)$ associated to $k$-labelled unknots. Upon taking total homology, this induces a homotopy equivalence of complexes of $\Sym(\X_i)$-modules: 
\[ (H_{tot}(\Mfu(\cal{L})),d_v^*) \rightarrow (H_{tot}(C(\cal{L})),d_v^*).\] 
\end{lem}
\begin{proof} 
Consider the 2-ended tangle diagram $\cal{T}$ obtained by cutting open the marked edge in $\cal{L}$, and its associated complex of matrix factorizations
\[ \Mfu(\cal{T}):= \Mfu(\cal{B})\otimes \cal{Z}_{\overline{\mathbf{y}},\overline{\mathbf{x}}}.  \]  
The web evaluation algorithm from Proposition~\ref{prop-annulareval} together with Remark~\ref{rem-z2} imply that the complex $\Mfu(\cal{T})$ of matrix factorization is homotopy equivalent (in the category of chain complexes over the homotopy category of $\Z_2$-graded matrix factorizations over $\Sym(\X_i|\X'_i)$) to a complex $O(\cal{L})$ whose chain groups are direct sums of the matrix factorizations $\cal{Z}_{\mathbf{y}_i,-\mathbf{x}_i}$ assigned to the trivial $k$-labelled 1-strand tangle. After closing off the marked edge by tensoring with $\cal{Z}_{\mathbf{y}_i,\mathbf{x}_i}$, we obtain the desired homotopy equivalence:
\[ \Mfu(\cal{L}) = \Mfu(\cal{T})\otimes  \cal{Z}_{\mathbf{y}_i,\mathbf{x}_i} \to  O(\cal{L}) \otimes  \cal{Z}_{\mathbf{y}_i,\mathbf{x}_i} = C(\cal{L}).\]
This preserves the $\Sym(\X_i|\X'_i)$-module structure on the matrix factorizations and induces a homotopy equivalence of $\Sym(\X_i|\X'_i)/ \la\X_i = \X'_i\ra\cong\Sym(\X_i)$-modules after taking total homology.
\end{proof}

In the following, let $\zerorep:=\Sym(\X_i)/ \la \X_i=0 \ra$ and $\zerorep':=\Sym(\X'_i)/ \la \X'_i=0 \ra$.

\begin{prop}\label{prop-innerred} There is a homotopy equivalence of chain complexes with respect to $d_v^*$:
\[ \zerorep\otimes H_{tot}(\Mfu(\cal{L})) \sim H_{tot}(\overline{\Mfu(\cal{L}(i))})\]
\end{prop}
\begin{proof} For the proof we express $H_{tot}(\zerorep \otimes \Mfu(\cal{L}))$ in two different ways. In the case of $\cal{L}=\bigcirc^k$ we get:
\begin{gather*}
 H_{tot}(\zerorep\otimes \Mfu(\bigcirc^k)) = H_{tot}(\zerorep\otimes \cal{Z}_{\mathbf{y}_i,-\mathbf{x}_i}\otimes \cal{Z}_{\mathbf{y}_i,\mathbf{x}_i})  
  \cong H_{tot}((\zerorep\otimes \zerorep') \otimes \cal{Z}_{\mathbf{y}_i,\mathbf{x}_i} ) \\ = H_{tot}((\zerorep\otimes \zerorep') \otimes \cal{Z}_{0,0} )  \cong  \zerorep \otimes H^*(S^1)^{\otimes k} 
\\ \cong  \zerorep \otimes H_{tot}(\Mfu(\bigcirc^k)) \otimes H^*(S^1)^{\otimes k} 
\end{gather*}
Here the first isomorphism is provided by Proposition~\ref{prop-simplify}. More generally, we consider a complex $C(\cal{L})=O(\cal{L})\otimes \cal{Z}_{\mathbf{y}_i,\mathbf{x}_i}$---as in the proof of Lemma~\ref{lem-circlecx}---with vertical differential supported on $O(\cal{L})$, whose chain groups are direct sums of $\cal{Z}_{\mathbf{y}_i,-\mathbf{x}_i}$. Analogously as in the unknot case, we can simplify the direct summands $\zerorep\otimes \cal{Z}_{\mathbf{y}_i,-\mathbf{x}_i}$ of the chain groups of $V\otimes O(\cal{L})$ to terms of the form $\zerorep\otimes \zerorep'$. After tensoring with the single-strand closure $\cal{Z}_{\mathbf{y}_i,\mathbf{x}_i}$ and taking total homology, we get an isomorphism of chain complexes
\[ H_{tot}(\zerorep \otimes C(\cal{L})) \cong \zerorep \otimes H_{tot}(C(\cal{L})) \otimes H^*(S^1)^{\otimes k}   \]
in which the vertical differential acts trivially on the tensor factor $H^*(S^1)^{\otimes k}$ contributed by the single-strand closure $\cal{Z}_{\mathbf{y}_i,\mathbf{x}_i}$. Using the homotopy equivalences from Lemma~\ref{lem-circlecx}, which preserve the $\Sym(\X_i)$-module structures before and after taking total homology, we now get:
\begin{align}
\label{eqn-s11}
H_{tot}(\zerorep \otimes \Mfu(\cal{L})) \sim H_{tot}(\zerorep \otimes C(\cal{L}))
&\cong \zerorep \otimes H_{tot}(C(\cal{L})) \otimes H^*(S^1)^{\otimes k}  
\\ \nonumber
 &\sim  \zerorep \otimes H_{tot}(\Mfu(\cal{L})) \otimes H^*(S^1)^{\otimes k}   
\end{align}
On the other hand, we can tensor~\eqref{eqn-umf} with $\zerorep$ and take $H_{tot}$ homology to get:
\begin{align}
\label{eqn-s12} H_{tot}(\zerorep \otimes \Mfu(\cal{L})) &\cong H_{tot}(\zerorep \otimes \Mfd(\cal{L})) 
 \\ \nonumber &= H_{tot}\left(\zerorep \otimes \frac{\UMf(\cal{B})}{\la \X_j=\X'_j\ra}\otimes \cal{Z}_{\pi(\mathbf{y}_i),0} \right) \cong  H_{tot}(\overline{\Mfu(\cal{L}(i))})\otimes H^*(S^1)^{\otimes k} . 
\end{align}
On the latter complex, the vertical differential acts trivially on the tensor factor $H^*(S^1)^{\otimes k}$, which is again the contribution of the closure of the marked edge. Since the homotopy category of graded chain complexes of $\C$-vector spaces is Krull--Schmidt, we can cancel the factors of $H^*(S^1)^{\otimes k}$ in \eqref{eqn-s11} and \eqref{eqn-s12}, which produces the desired homotopy equivalence.
\end{proof}

\begin{cor} The reduced Khovanov--Rozansky homology is computed by the reduced complex:
\[\overline{\mathrm{KhR}}^{\glnn{N}}(\cal{L}(i)) \cong H_*(H_{tot}(\overline{\Mfu(\cal{L}(i))}),d_v^*)\]
\end{cor}
\begin{proof} Analogously as in \cite[Proof of Theorem 4.12]{QR}, it follows from the 2-functor from foams to matrix factorizations that $(H_{tot}(\Mfu(\cal{L})),d_v^*)$ and $\grtaut(\LC{\cal{L}}{\glnn{N}}{})$ are isomorphic as chain complexes of $H_k^N$-modules. The reduced homology is then obtained from tensoring the chain groups of these complexes with the $H_k^N$-module $\la \pi_k^N\ra$ before taking vertical homology. Clearly, this is equivalent to tensoring with the isomorphic $H_k^N$-module $\zerorep\otimes H_k^N$. Now the claim follows from Proposition~\ref{prop-innerred}.
\end{proof}

\begin{defi} Let $\cal{L}(i)$ be the closure of a balanced labelled braid diagram $\cal{B}$ with a marked lower boundary point $o_i$ of minimal label. Then we define the reduced colored HOMFLY-PT homology of $\cal{L}$ reduced at $i$ as
\begin{align*}
\LHHr{\cal{L}(i)}{}:= H_*(H^+(\overline{\Mfu(\cal{L}(i))})), d_v^*).
\end{align*}
\end{defi}

 Let $\Wu := \Vu / \la \X_j=\X'_j| j\neq i\ra = \Vu / \la \X_j=\X'_j\ra$, which is still isomorphic to a polynomial ring containing a tensor factor $\Sym(\X_i)$. Thus we can write $\Wu \cong \Wu/\la \X_i=0\ra \otimes \Sym(\X_i)$. In the following, we consider $\UMf(\cal{B})/ \la \X_j=\X'_j\ra$ as a (vertical) complex of matrix factorizations over $\Wu/\la \X_i=0\ra$.
 
\begin{lem} \label{lem-cone} As a complex of matrix factorizations over $\Wu/\la \X_i=0\ra$, the reduced complex
$\overline{\Mfu(\cal{L}(i))}$ is homotopy equivalent to the iterated mapping cone $\UMf(\cal{B})/ \la \X_j=\X'_j\ra \otimes \cal{Z}_{\X_i}$. Here $\cal{Z}_{\X_i}$ is the (vertical) Koszul complex over $\Sym(\X_i)$ with respect to the sequence consisting of $e_l(\X_i)$ for $1 \leq l \leq k$.  
\end{lem}
\begin{proof} For $0\leq l \leq k$ define $\Wu^l:=  \Wu/\la e_x(\X_i)=0|x>l\ra$ and $\overline{\Mfu(\cal{L}(i))}_l:= \Wu^l \otimes \UMf(\cal{B})$, which is a free matrix factorization over $\Wu^l$. As special cases we have $\overline{\Mfu(\cal{L}(i))}_{k} = \UMf(\cal{B})/ \la \X_j=\X'_j\ra$ and $\overline{\Mfu(\cal{L}(i))}_0=\overline{\Mfu(\cal{L}(i))}$. Now Rasmussen's one-variable argument from \cite[Lemma 5.15]{Ras1} shows that $\overline{\Mfu(\cal{L}(i))}_l$ is homotopy equivalent, as a complex of matrix factorizations over $\Wu^l$, to $\overline{\Mfu(\cal{L}(i))}_{l+1} \otimes \cal{Z}_{e_l(\X_i)} $. The same holds over sub-rings $\Wu^{l'}$ with $l'<l$, and so we can simultaneously trade quotienting out all generators $e_l(\X_i)$ against tensoring with the Koszul complex $\cal{Z}_{\X_i}$ . 
\end{proof}

After taking positive homology, we get a homotopy equivalence of vertical complexes
\begin{equation}
\label{eqn-cone}H^+(\overline{\Mfu(\cal{L}(i))}) \sim H^+(\UMf(\cal{B})/ \la \X_j=\X'_j\ra) \otimes \cal{Z}_{\X_i}
\end{equation}

\begin{lem}\label{lem-markedshift}  $\LHHr{\cal{L}(i)}{}$ is isomorphic to $\LHHr{\cal{L}(i')}{}$ if the marked points $o_i$ and $o_{i'}$ lie on the same component of $\cal{L}$.
\end{lem}
\begin{proof} Equation \eqref{eqn-cone} shows that $\LHHr{\cal{L}(i)}{}$ is the homology of an iterated mapping cone on endomorphisms of the complex $H^+(\UMf(\cal{B})/ \la \X_j=\X'_j\ra))$, which does not depend on $i$. Since cones on homotopic chain maps are isomorphic in the homotopy category, it suffices to show that the endomorphism of $H^+(\UMf(\cal{B})/ \la \X_j=\X'_j\ra)$ given by multiplying by  $e_l(\X_i)$ is homotopic to the endomorphism given by multiplying by $e_l(\X_{i'})$. In fact, it is enough to check this homotopy between endomorphisms locally, i.e. on the complex associated to a single crossing, where the alphabets $\X_i$ and $\X_{i'}$ are assigned to opposite sides of a single strand. This has been done by Rose and the author in \cite[Proposition 60]{RW} in the framework of webs and foams, and an analogous proof works in the present setting.
\end{proof}

As a by-product we get that the $\Su$ action on $\LHHr{\cal{L}(i)}{}$ factors through the quotient $\Su/\la \Y_j=\Y_{j'}, \X_i=0 \ra$ where we identify alphabets $\Y_j,\Y_{j'}$ which belong to the same link component. In particular, all alphabets on the component containing the marked point $o_i$ act by zero.

\begin{cor} The triply-graded vector space $\LHHr{\cal{L}(i)}{}$ is independent of the diagram of the labelled link with marked component specified by $i$. 
\end{cor}
\begin{proof} From Corollary~\ref{cor-umf} we get an isomorphism of chain complexes with respect of $d_v^*$:
\begin{equation}
\label{eqn-splitoff}
(H^+(\Mfu(\cal{L})),d_v^*) \cong (H^+(\UMf(\cal{B})/ \la \X_j=\X'_j\ra),d_v^*) \otimes H^*(S^1)^{\otimes k}
\end{equation}
and this decomposition does not depend on $i$, only on its label. Since $H^+(\Mfu(\cal{L}))$ is invariant under Reidemeister 2 and 3 and Markov moves up to homotopy, so is $H^+(\UMf(\cal{B})/ \la \X_j=\X'_j\ra)$. The same is true for this complex tensored with $\cal{Z}_{\X_{i}}$, as we may assume that the alphabet $\X_{i}$ is associated to an edge not participating in the move (see proof of Lemma~\ref{lem-markedshift}). 
\end{proof}

\begin{rem} 
$H_*(H^+(\UMf(\cal{B})/ \la \X_j=\X'_j\ra),d_v^*)$ is a triply-graded link invariant which is the colored generalization of what Rasmussen calls the \emph{middle HOMFLY homology}, \cite[Definition 2.9]{Ras1}. It carries exactly the same amount of information as $\LHH{\cal{L}}{}$.
\end{rem}

\begin{rem} 
\label{rem-redunred}
In the uncolored case, Rasmussen shows that $\LHH{\cal{L}}{}\cong \LHHr{\cal{L}(i)}{}\otimes \LHH{\bigcirc}{}$, see \cite[Section 2.8]{Ras1}. The key ingredient for this is an isomorphism
$\UMf(\cal{B})/ \la \X_j=\X'_j\ra) \cong \overline{\Mfu(\cal{L}(i))} \otimes \Sym(\X_i)$ of double complexes with respect to $d$ and $d_+$. We suspect that this also holds in the colored case, but do not know a proof of this fact.
\end{rem}

\begin{thm}\label{thm-redHss}  There is a spectral sequence 
\[\LHHr{\cal{L}(i)}{} \quad \rightsquigarrow \quad\overline{\mathrm{KhR}}^{\glnn{N}}(\cal{L}(i))\]
whose $k^{th}$ differential lowers horizontal grading by $k$, homological grading by $k$ and it increases the internal grading by $2k N$.
\end{thm}
\begin{proof} Analogously as in the unreduced situation of Theorem~\ref{thm-Hss}, we have a vertical complex of matrix factorizations $\overline{\Mfu(\cal{L}(i))}$ with potential zero, from which HOMFLY-PT as well as $\glnn{N}$ homologies can be recovered:  
\begin{align*}
\LHHr{\cal{L}(i)}{}&= H_*(H^+(\overline{\Mfu(\cal{L}(i))})), d_v^*)\\
\overline{\mathrm{KhR}}^{\glnn{N}}(\cal{L}(i)) &\cong H_*(H_{tot}(\overline{\Mfu(\cal{L}(i))}),d_v^*)
\end{align*}
The proof now proceeds as for Theorem~\ref{thm-Hss}, or its uncolored predecessor in \cite{Ras1}. The necessary reduced analogue of Lemma~\ref{lem-singlehor} can be proved by using the web evaluation algorithm introduced in Proposition~\ref{prop-annulareval}, which allows to keep the marked edge undisturbed in the simplification process. Further, it is easy to check by hand that the additional bigon relation \eqref{eqn-addbigon} holds in the categorified framework and that it causes a homogeneous shift in horizontal grading in all webs appearing in the algorithm (just like removing circles of labels that sum up $\sum o_i$).
\end{proof}

\begin{lem} \label{lem-finrank} For any web $W$ appearing in $\LC{\cal{B}}{}{}$, the singular Sorgel bimodule $\Dec(W)$ is finitely generated over $\cal{S}_\cal{B}$.
\end{lem}
\begin{proof}
This would be obvious from the definition of $\Dec(W)$ if $\cal{S}_\cal{B}$ would contain an alphabet for every web edge of $W$. However, $W$ contains additional web edges on top of the ones already present as arcs of the diagram $\cal{B}$---namely the ones introduced in crossing resolutions, e.g. $\W_1, \W_2, \W_3$ and $\W_4$ in \eqref{eqn-crossingresol}. It remains to show that all decorations on these additional web edges can be expressed as $\cal{S}_\cal{B}$-linear combination of only a finite number of basis decorations. For this, we consider a typical web appearing at a crossing site:
\begin{equation}
\label{eqn-crossingresol}\xy
(0,0)*{
\begin{tikzpicture} [scale=.5,fill opacity=0.2]
	%bottom web
	\draw[very thick, directed=.55]  (2,0) to (1.25,0);
	\draw[very thick]  (1.25,0) to (0.5,0);
	\draw[very thick, directed=.55]  (-1.25,0) to (-2,0);
	\draw[very thick]  (-0.5,0) to (-1.25,0);
	\draw[very thick, directed=.55]  (2,1) to (1.25,1);
	\draw[very thick, directed=.55]  (-1.25,1) to (-2,1);
	\draw[very thick] (1.25,1) to (-1.25,1);
	\draw[very thick, directed=.55]  (0.5,1) to (-0.5,1);
	\draw[very thick, directed=.55]  (0.5,0) to (-0.5,0);
	\draw[very thick, directed=.55] (1.25,1) to [out=180,in=0](0.5,0) ;
	\draw[very thick, directed=.55] (-0.5,0) to [out=180,in=0](-1.25,1) ;
\node[opacity=1] at (0,1.5) {$_{\W_3}$};
	\node[opacity=1] at (0,-0.5) {$_{\W_4}$};
	\node[opacity=1] at (-1.7,0.45) {\tiny $\W_2$};
	\node[opacity=1] at (1.7,0.45) {\tiny $\W_1$};
		\node[opacity=1] at (1.7,1.5) {$_{\Y_1}$};
	\node[opacity=1] at (1.7,-0.5) {$_{\Y_2}$};
	\node[opacity=1] at (-1.7,1.5) {$_{\Y'_1}$};
	\node[opacity=1] at (-1.7,-0.5) {$_{\Y'_2}$};
\end{tikzpicture}
};
\endxy
\end{equation}

Here, the edges labelled with the alphabets $\Y_i$ come from arcs in the diagram $\cal{B}$ and the edges labelled by $\W_i$ alphabets are new. After eliminating $\W_4$, we see that the singular Soergel bimodule associated to this web is isomorphic to:

\begin{equation}
\label{eqn-localbimod}\frac{\Sym(\Y'_1|\W_2|\W_3)}{(\Y'_1=\W_2\sqcup \W_3)}  \otimes_{\Sym(\W_2|\W_3)}
\frac{\Sym(\Y_2|\Y'_2|\W_1|\W_2|\W_3)}{(\Y'_2\sqcup \W_2=\Y_2\sqcup \W_1)}  \otimes_{\Sym(\W_1|\W_3)}
 \frac{\Sym(\W_1|\W_3|\Y_1)}{(\W_1\sqcup \W_3=\Y_1)} 
  \end{equation} 
It is well known that $\Sym(\W_2|\W_3)$ is a free $\Sym(\W_2\sqcup \W_3)$-module of rank ${|\W_2\sqcup \W_3|\choose |\W_2|}$ with a basis given by Schur polynomials $\pi^{\W_2}_\lambda$ with $\lambda\in P(|\W_2|,|\W_3|)$. Thus, the left tensor factor is of finite rank over $\Sym(\Y'_1)$. The case of the right tensor factor is completely analogous and, consequently, the total bimodule is finitely generated over $\Sym(\Y_1|\Y_2|\Y'_1|\Y'_2)$. 
\end{proof}

\begin{prop}\label{prop-findim} For a labelled knot $\cal{K}$, the reduced colored HOMFLY-PT homology $\LHHr{\cal{K}}{}$ is finite dimensional.
\end{prop}
\begin{proof}
 From Lemma~\ref{lem-finrank} we see that $H^+(\Mfu(\cal{L}))\cong H^+(\Dec(\cal{B})\otimes \cal{Z}_{\mathbf{y},\mathbf{x}})$ is finitely generated over $\Su$ and so is the iterated mapping cone $H^+(\Mfu(\cal{L}))\otimes \cal{Z}_{\X_i}$. From~\eqref{eqn-splitoff} and Lemma~\ref{lem-cone} we see:
 \[H^+(\Mfu(\cal{L}))\otimes \cal{Z}_{\X_i} \cong H^+(\UMf(\cal{B})/ \la \X_j=\X'_j\ra) \otimes H^*(S^1)^{\otimes k} \otimes \cal{Z}_{\X_i}\sim H^+(\overline{\Mfu(\cal{L}(i))})\otimes H^*(S^1)^{\otimes k}\]
The vertical homology of this iterated cone is isomorphic as a triply-graded vector space to a direct sum of $2^k$ copies of $\LHHr{\cal{K}}{}$ and we shall argue that it is finite-dimensional. The generators $\Sym(\X_i)$ act null-homotopically on the iterated cone (any chain endomorphisms $f$ induces a null-homotopic endomorphism of $\mathrm{Cone}(f)$) and, thus, by zero on its vertical homology. By the proof of Lemma~\ref{lem-markedshift}, the same holds for all alphabets in $\Su$ in the case of a knot $\cal{L}=\cal{K}$. The vertical homology of the iterated cone is a finitely generated $\Su$-module, but all variables act by zero, so it is finitely generated over $\C$. Consequently, its tensor factor $\LHHr{\cal{K}}{}$ is finite-dimensional.
\end{proof}

\comm{
\PW{Lemma providing basic homological algebra referred to above}
\begin{lem} Let $(C,d)$ be a chain complex in an additive category and $f$ and endomorphism of $C$. Then $f$ applied component-wise is a null-homotopic endomorphism of the cone $C(f)$.  
\end{lem}
\begin{proof} 
\comm{
An alternative, equivalent definition of cone:
$C(f)\cong C \otimes C$ with differential $d_f:=(x,y)\mapsto (d(x),d(y)+(-1)^{|x|}f(x))$. Let's check that this is a chain complex by squaring the differential
\[(x,y)\mapsto (d(x),d(y)+(-1)^{|x|}f(x)) \mapsto (d(d(x)),d(d(y))+(-1)^{|x|}d(f(x)) - (-1)^{|x|}f(d(x)) ) =(0,0)\]
Let $H$ be the map sending $(x,y)\mapsto (y,0)$. Then we compute 
\[(H d_f - d_f H)(x,y) = (d(y)+(-1)^{|x|}f(x),0)-(d(y),(-1)^{|y|}f(y))= ((-1)^{|x|}f(x),(-1)^{|y|}f(y))\] }
$C(f)\cong C \otimes C$ with differential $d_f:=(x,y)\mapsto (-d(x),d(y)+f(x))$. Let's check that this is a chain complex by squaring the differential
\[(x,y)\mapsto (-d(x),d(y)+f(x)) \mapsto (d(d(x)),d(d(y))+d(f(x)) -f(d(x)) ) =(0,0)\]
Let $H$ be the map sending $(x,y)\mapsto (y,0)$. Then we compute 
\[(H d_f + d_f H)(x,y) = (d(y)+f(x),0)+(-d(y),f(y))= (f(x),f(y))\]
This automatically implies that $(x,y)\mapsto (f(x),f(y))$ is a chain map. But let's check it anyway:
\[(f d_f) (x,y)=(f(d(x)),f(d(y)))+f(f(x))= (d(f(x)),d(f(y)))+f(f(x)) = (d_f f)(x,y) \]
\end{proof}
}

\begin{thm}
\label{thm-gradcollapse} For a labelled knot $\cal{K}$, the reduced colored HOMFLY-PT homology $\LHHr{\cal{K}}{}$ is a stabilization of reduced colored $\glnn{N}$ homologies $\overline{\mathrm{KhR}}^{\glnn{N}}(\cal{K})$ for large $N$. More precisely:

\begin{equation}
\label{eqn-gradcollapse}
\bigoplus_{{i+2N j = I}\atop{h-j=J}} \LHHr{\cal{K}}{i,j,h} \cong \overline{\mathrm{KhR}}_{I,J}^{\glnn{N}}(\cal{K})
\end{equation}

\end{thm}
\begin{proof}  We know that $\LHHr{\cal{K}}{}$ is finite-dimensional and, thus, supported in finitely many internal gradings. Since the $k^{th}$ differential in the spectral sequence from Theorem~\ref{thm-redHss} increases the internal grading by $2 N k$, all higher differentials must be trivial for large $N$. So the spectral sequence, which is known to compute $\glnn{N}$ homology, must have already converged on the first page, which is precisely the grading-collapsed version of HOMFLY-PT homology displayed in~\eqref{eqn-gradcollapse}. 
\end{proof}

\begin{cor} \textbf{(Refined exponential growth for HOMFLY-PT homology)} 
\label{cor_expgrw} There exist spectral sequences
\begin{align}
\LHHr{\cal{K}^k}{}  \quad &\rightsquigarrow \quad  \left( \LHHr{\cal{K}^1}{}\right)^{\otimes k}\\
\LHHr{\cal{K}^k}{}  \quad &\rightsquigarrow \quad  \LHHr{\cal{K}^{k-1}}{} \otimes \LHHr{\cal{K}^{1}}{}.
\end{align}
%which preserve the homological grading. 
In particular, the colored HOMFLY-PT homologies of a knot grow at least exponentially in color.
\end{cor}
\begin{proof}
Immediate from Theorem~\ref{thm-gradcollapse} and the spectral sequences from Corollary~\ref{cor_colss} for $N\gg 0$.
\end{proof}

\begin{cor} \textbf{(Color-reducing spectral sequences)} 
\label{cor_coldiff}
 For $k\geq h$ there is a spectral sequence
\begin{align}
\LHHr{\cal{K}^k}{}  \quad &\rightsquigarrow \quad  \LHHr{\cal{K}^h}{}.
\end{align}
%which preserves the homological grading (up to a shift).
\end{cor}
\begin{proof}
Immediate from Theorem~\ref{thm-gradcollapse} and the spectral sequences from Corollary~\ref{cor_colss2} for $N\gg 0$.
\end{proof}

\section{Appendix on web evaluation}
\label{sec-app}
This appendix contains an algorithm for the evaluation of $\glnn{N}$ webs with two boundary points. For this, we work in the decategorified framework of $\glnn{N}$ webs, see e.g. Cautis--Kamnitzer--Morrison's paper \cite{CKM} or Murakami--Ohtsuki--Yamada's earlier work \cite{MOY}. In this section, we will draw webs as oriented upwards instead of leftwards. First we recall the annular web evaluation algorithm of Queffelec--Rose. For this, we will need the following terminology.

\begin{defi} A $\glnn{N}$ web is in ladder form (or just \textit{is a ladder web}) if any of its edges is either vertical or horizontal. The former are called the \textit{uprights} and the latter the \textit{rungs} of the ladder web. 
\end{defi}
Every $\glnn{N}$ web can be isotoped into ladder form, although not uniquely, see \cite[Theorem 5.3.1]{CKM}.

\begin{lem}\cite[Lemma 5.2]{QR2}
\label{lem-annulareval} The annular closure (c.f. \eqref{eqn:braidclosure}) of a $\glnn{N}$ ladder web can be rewritten as a $\Z[q^{\pm 1}]$-linear combination of annular closures of identity webs (i.e. nested, labelled essential circles) using only the following four relations:
\begin{enumerate}
\item Sliding ladder rungs around the annulus
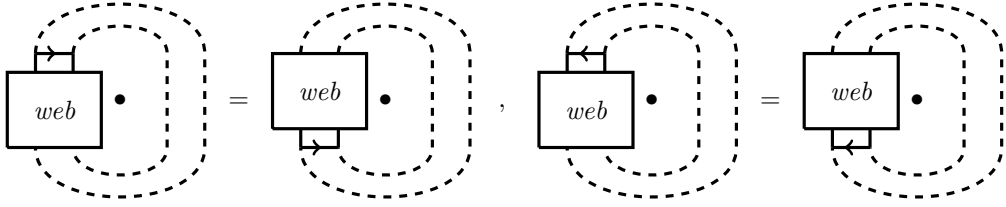
\[\xy
(0,0)*{
\begin{tikzpicture}[scale=.25]
	\draw [very thick] (-2.5,-3) to (-2.5,1) to (2.5,1) to (2.5,-3) to (-2.5,-3);
	\draw [very thick] (-1,1) to (-1,2);
	\draw [very thick] (1,1) to (1,2) ;
	\draw [very thick,dashed] (-1,2) to [out=90,in=90] (8,2) to (8,-3) to [out=270,in=270] (-1,-3);
	\draw [very thick,dashed]  (1,2) to [out=90,in=90] (6,2) to (6,-3) to [out=270,in=270] (1,-3);
	%%rungs
	\draw [very thick, directed=.55] (-1,2) to (1,2);
	\node at (0,-1) {$\text{web}$};
	\node at (3.5,-0.5) {$\bullet$};
\end{tikzpicture}
};
\endxy \;\; = \;\;
\xy
(0,0)*{
\begin{tikzpicture}[scale=.25]
	\draw [very thick] (-2.5,-2) to (-2.5,2) to (2.5,2) to (2.5,-2) to (-2.5,-2);
	\draw [very thick] (-1,-3) to (-1,-2);
	\draw [very thick] (1,-3) to (1,-2) ;
	\draw [very thick,dashed] (-1,2) to [out=90,in=90] (8,2) to (8,-3) to [out=270,in=270] (-1,-3);
	\draw [very thick,dashed]  (1,2) to [out=90,in=90] (6,2) to (6,-3) to [out=270,in=270] (1,-3);
	%%rungs
	\draw [very thick, directed=.55] (-1,-3) to (1,-3);
	\node at (0,0) {$\text{web}$};
		\node at (3.5,-0.5) {$\bullet$};
\end{tikzpicture}
};
\endxy
\quad,\quad 
\xy
(0,0)*{
\begin{tikzpicture}[scale=.25]
	\draw [very thick] (-2.5,-3) to (-2.5,1) to (2.5,1) to (2.5,-3) to (-2.5,-3);
	\draw [very thick] (-1,1) to (-1,2);
	\draw [very thick] (1,1) to (1,2) ;
	\draw [very thick,dashed] (-1,2) to [out=90,in=90] (8,2) to (8,-3) to [out=270,in=270] (-1,-3);
	\draw [very thick,dashed]  (1,2) to [out=90,in=90] (6,2) to (6,-3) to [out=270,in=270] (1,-3);
	%%rungs
	\draw [very thick, rdirected=.55] (-1,2) to (1,2);
	\node at (0,-1) {$\text{web}$};
		\node at (3.5,-0.5) {$\bullet$};
\end{tikzpicture}
};
\endxy\;\; = \;\;
\xy
(0,0)*{
\begin{tikzpicture}[scale=.25]
	\draw [very thick] (-2.5,-2) to (-2.5,2) to (2.5,2) to (2.5,-2) to (-2.5,-2);
	\draw [very thick] (-1,-3) to (-1,-2);
	\draw [very thick] (1,-3) to (1,-2) ;
	\draw [very thick,dashed] (-1,2) to [out=90,in=90] (8,2) to (8,-3) to [out=270,in=270] (-1,-3);
	\draw [very thick,dashed]  (1,2) to [out=90,in=90] (6,2) to (6,-3) to [out=270,in=270] (1,-3);
	%%rungs
	\draw [very thick, rdirected=.55] (-1,-3) to (1,-3);
	\node at (0,0) {$\text{web}$};
		\node at (3.5,-0.5) {$\bullet$};
\end{tikzpicture}
};
\endxy
\]

\item Ladder rung combination
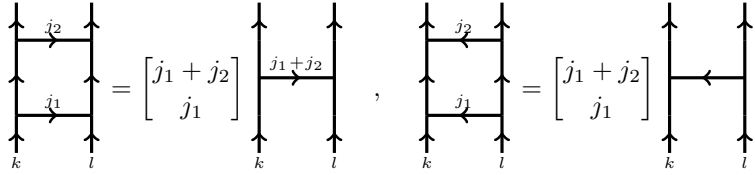
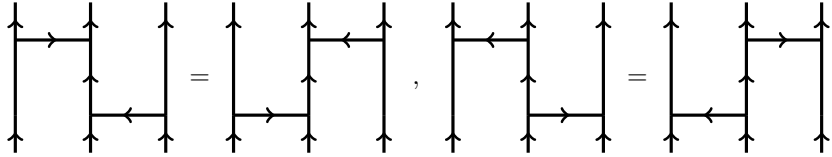
\[\xy
(0,0)*{
\begin{tikzpicture}[scale=.25]
	\draw [very thick, directed=.55] (-2,-4) to (-2,-2);
	\draw [very thick, directed=.55] (-2,-2) to (-2,2);
	\draw [very thick, directed=.55] (2,-4) to (2,-2);
	\draw [very thick, directed=.55] (2,-2) to (2,2);
	\draw [very thick, directed=.55] (-2,2) to (-2,4);
	\draw [very thick, directed=.55] (2,2) to (2,4);
	%%rungs
	\draw [very thick, directed=.55] (-2,2) to (2,2);
	\draw [very thick, directed=.55] (-2,-2) to (2,-2); 
	\node at (-2,-4.5) {\tiny $k$};
	\node at (2,-4.5) {\tiny $l$};
	\node at (0,-1.25) {\tiny $j_1$};
	\node at (0,2.75) {\tiny $j_2$};
\end{tikzpicture}
};
\endxy= {j_1+j_2 \brack j_1}\xy
(0,0)*{
\begin{tikzpicture}[scale=.25]
	\draw [very thick, directed=.55] (-2,-4) to (-2,-2);
	\draw [very thick] (-2,-2) to (-2,2);
	\draw [very thick, directed=.55] (2,-4) to (2,-2);
	\draw [very thick] (2,-2) to (2,2);
	\draw [very thick, directed=.55] (-2,2) to (-2,4);
	\draw [very thick, directed=.55] (2,2) to (2,4);
	%%rungs
	\draw [very thick, directed=.55] (-2,0) to (2,0);
	\node at (-2,-4.5) {\tiny $k$};
	\node at (2,-4.5) {\tiny $l$};
	\node at (0,0.75) {\tiny $j_1\! + \! j_2 $};
\end{tikzpicture}
};
\endxy
\quad,\quad
\xy
(0,0)*{
\begin{tikzpicture}[scale=.25]
	\draw [very thick, directed=.55] (-2,-4) to (-2,-2);
	\draw [very thick, directed=.55] (-2,-2) to (-2,2);
	\draw [very thick, directed=.55] (2,-4) to (2,-2);
	\draw [very thick, directed=.55] (2,-2) to (2,2);
	\draw [very thick, directed=.55] (-2,2) to (-2,4);
	\draw [very thick, directed=.55] (2,2) to (2,4);
	%%rungs
	\draw [very thick, rdirected=.55] (-2,2) to (2,2);
	\draw [very thick, rdirected=.55] (-2,-2) to (2,-2); 
	\node at (-2,-4.5) {\tiny $k$};
	\node at (2,-4.5) {\tiny $l$};
	\node at (0,-1.25) {\tiny $j_1$};
	\node at (0,2.75) {\tiny $j_2$};
\end{tikzpicture}
};
\endxy= {j_1+j_2 \brack j_1}\xy
(0,0)*{
\begin{tikzpicture}[scale=.25]
	\draw [very thick, directed=.55] (-2,-4) to (-2,-2);
	\draw [very thick] (-2,-2) to (-2,2);
	\draw [very thick, directed=.55] (2,-4) to (2,-2);
	\draw [very thick] (2,-2) to (2,2);
	\draw [very thick, directed=.55] (-2,2) to (-2,4);
	\draw [very thick, directed=.55] (2,2) to (2,4);
	%%rungs
	\draw [very thick, rdirected=.55] (-2,0) to (2,0);
	\node at (-2,-4.5) {\tiny $k$};
	\node at (2,-4.5) {\tiny $l$};
\end{tikzpicture}
};
\endxy\]
\item Frobenius relations
\[
\xy
(0,0)*{
\begin{tikzpicture}[scale=.25]
	\draw [very thick, directed=.55] (-2,-4) to (-2,-2);
	\draw [very thick] (-2,-2) to (-2,2);
	\draw [very thick, directed=.55] (-2,2) to (-2,4);
		\draw [very thick, directed=.55] (2,-4) to (2,-2);
	\draw [very thick, directed=.55] (2,-2) to (2,2);
	\draw [very thick, directed=.55] (2,2) to (2,4);
		\draw [very thick, directed=.55] (6,-4) to (6,-2);
	\draw [very thick] (6,-2) to (6,2);
	\draw [very thick, directed=.55] (6,2) to (6,4);
	%%rungs
	\draw [very thick, directed=.55] (-2,2) to (2,2);
	\draw [very thick, rdirected=.55] (2,-2) to (6,-2); 
\end{tikzpicture}
};
\endxy\;\;=\;\; \xy
(0,0)*{
\begin{tikzpicture}[scale=.25]
	\draw [very thick, directed=.55] (-2,-4) to (-2,-2);
	\draw [very thick] (-2,-2) to (-2,2);
	\draw [very thick, directed=.55] (-2,2) to (-2,4);
		\draw [very thick, directed=.55] (2,-4) to (2,-2);
	\draw [very thick, directed=.55] (2,-2) to (2,2);
	\draw [very thick, directed=.55] (2,2) to (2,4);
		\draw [very thick, directed=.55] (6,-4) to (6,-2);
	\draw [very thick] (6,-2) to (6,2);
	\draw [very thick, directed=.55] (6,2) to (6,4);
	%%rungs
	\draw [very thick, directed=.55] (-2,-2) to (2,-2);
	\draw [very thick, rdirected=.55] (2,2) to (6,2); 
\end{tikzpicture}
};
\endxy
\quad,\quad
\xy
(0,0)*{
\begin{tikzpicture}[scale=.25]
	\draw [very thick, directed=.55] (-2,-4) to (-2,-2);
	\draw [very thick] (-2,-2) to (-2,2);
	\draw [very thick, directed=.55] (-2,2) to (-2,4);
		\draw [very thick, directed=.55] (2,-4) to (2,-2);
	\draw [very thick, directed=.55] (2,-2) to (2,2);
	\draw [very thick, directed=.55] (2,2) to (2,4);
		\draw [very thick, directed=.55] (6,-4) to (6,-2);
	\draw [very thick] (6,-2) to (6,2);
	\draw [very thick, directed=.55] (6,2) to (6,4);
	%%rungs
	\draw [very thick, rdirected=.55] (-2,2) to (2,2);
	\draw [very thick, directed=.55] (2,-2) to (6,-2); 
\end{tikzpicture}
};
\endxy\;\;=\;\; \xy
(0,0)*{
\begin{tikzpicture}[scale=.25]
	\draw [very thick, directed=.55] (-2,-4) to (-2,-2);
	\draw [very thick] (-2,-2) to (-2,2);
	\draw [very thick, directed=.55] (-2,2) to (-2,4);
		\draw [very thick, directed=.55] (2,-4) to (2,-2);
	\draw [very thick, directed=.55] (2,-2) to (2,2);
	\draw [very thick, directed=.55] (2,2) to (2,4);
		\draw [very thick, directed=.55] (6,-4) to (6,-2);
	\draw [very thick] (6,-2) to (6,2);
	\draw [very thick, directed=.55] (6,2) to (6,4);
	%%rungs
	\draw [very thick, rdirected=.55] (-2,-2) to (2,-2);
	\draw [very thick, directed=.55] (2,2) to (6,2); 
\end{tikzpicture}
};
\endxy
\]

\item Square switch relations

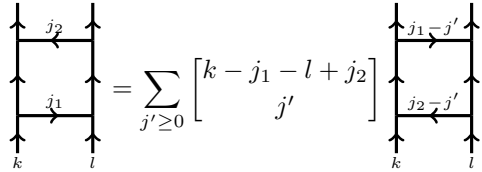
\[\xy
(0,0)*{
\begin{tikzpicture}[scale=.25]
	\draw [very thick, directed=.55] (-2,-4) to (-2,-2);
	\draw [very thick, directed=.55] (-2,-2) to (-2,2);
	\draw [very thick, directed=.55] (2,-4) to (2,-2);
	\draw [very thick, directed=.55] (2,-2) to (2,2);
	\draw [very thick, directed=.55] (-2,-2) to (2,-2);
	\draw [very thick, directed=.55] (-2,2) to (-2,4);
	\draw [very thick, directed=.55] (2,2) to (2,4);
	\draw [very thick, rdirected=.55] (-2,2) to (2,2);
	\node at (-2,-4.5) {\tiny $k$};
	\node at (2,-4.5) {\tiny $l$};
	\node at (0,-1.25) {\tiny $j_1$};
	\node at (0,2.75) {\tiny $j_2$};
\end{tikzpicture}
};
\endxy=\sum_{j^{\prime}\geq 0}{k-j_1-l+j_2 \brack j^{\prime}}\xy
(0,0)*{
\begin{tikzpicture}[scale=.25]
	\draw [very thick, directed=.55] (-2,-4) to (-2,-2);
	\draw [very thick, directed=.55] (-2,-2) to (-2,2);
	\draw [very thick, directed=.55] (2,-4) to (2,-2);
	\draw [very thick, directed=.55] (2,-2) to (2,2);
	\draw [very thick, rdirected=.55] (-2,-2) to (2,-2);
	\draw [very thick, directed=.55] (-2,2) to (-2,4);
	\draw [very thick, directed=.55] (2,2) to (2,4);
	\draw [very thick, directed=.55] (-2,2) to (2,2);
	\node at (-2,-4.5) {\tiny $k$};
	\node at (2,-4.5) {\tiny $l$};
	\node at (0,-1.25) {\tiny $j_2\! -\! j^{\prime}$};
	\node at (0,2.75) {\tiny $j_1\! -\! j^{\prime}$};
\end{tikzpicture}
};
\endxy\]
\end{enumerate}
\end{lem}
Here we write ${a \brack b}$ for the $q$-binomial coefficients. % and the webs in these relations have been rotated to make edges upward (or not-downward) oriented and they have been isotoped into the  \emph{ladder form} consisting of a fixed number of \emph{uprights} and \emph{ladder rungs} between them. 
We recall the explicit algorithm, which is used to prove this result, as we will need a certain subroutine later.
\begin{proof}\cite{QR2} At the start, and then throughout, the algorithm combines all possible adjacent ladder rungs pointing in the same direction using relations (2) and (3). As a first goal, the algorithm aims to free the leftmost upright from the rest of the diagram by repeatedly applying relation (4) to pairs of rungs adjacent to it. If possible, this has the effect of increasing a label on the leftmost upright or it allows to reduce the number of rungs touching it, which shows that this algorithm eventually terminates and succeeds in splitting of the leftmost upright. 
It suffices to argue that relation (4) can always be applied to a pair of rungs connected to the leftmost upright (except in the trivial case when the leftmost upright is already disconnected). Clearly, one can always find a leftwards rung above a rightward rung by relation (1), however, they might be separated by \emph{trapped rungs} further to the right in the web. We reproduce an example given in \cite{QR2}:
\[
\xy
(0,0)*{
\begin{tikzpicture}[scale=.25]
	\draw [very thick] (-2,-4) to (-2,4);
	\draw [very thick,dotted] (-2,-5) to (-2,5);
		\draw [very thick] (2,-4) to (2,4);
		\draw [very thick,dotted] (2,-5) to (2,5);
		\draw [very thick] (6,-4) to (6,4);
		\draw [very thick,dotted] (6,-5) to (6,5);
		\draw [very thick] (10,-4) to (10,4);
		\draw [very thick,dotted] (10,-5) to (10,5);
		\draw [very thick] (14,-4) to (14,4);
		\draw [very thick,dotted] (14,-5) to (14,5);
	%%rungs
	\draw [very thick, rdirected=.55] (-2,3.5) to (2,3.5);
	\draw [very thick, directed=.55] (-2,-3.5) to (2,-3.5); 
	\draw [very thick, red, directed=.55] (2,-3) to (6,-3);
	\draw [very thick, red, rdirected=.55] (2,0) to (6,0);
	\draw [very thick, red, directed=.55] (6,-2.5) to (10,-2.5);  
	\draw [very thick, red, directed=.55] (6,-1.5) to (10,-1.5);  
	\draw [very thick, red, directed=.55] (10,-2) to (14,-2);  
	\draw [very thick, red, rdirected=.55] (2,3) to (6,3);
	\draw [very thick, red, rdirected=.55] (6,2.5) to (10,2.5);  
	\draw [very thick, red, rdirected=.55] (6,1.5) to (10,1.5);  
	\draw [very thick, red, rdirected=.55] (10,1) to (14,1);  
	\draw [very thick, blue, directed=.55] (10,2) to (14,2);  
	\draw [very thick, blue, rdirected=.55] (10,-3) to (14,-3);  
\end{tikzpicture}
};
\endxy
\]
Here the black rungs are the chosen ones, which are separated by trapped rungs colored red. The blue rungs are not trapped and can be moved outside the local picture by applying relation (3). Queffelec and Rose show that all trapped rungs can be untrapped by applying relations (3) and (4), without moving the chosen pair of rungs on the left. We shall call this:  \\
\noindent \textbf{The untrapping subroutine:} Its strategy is to move the lowest leftward pointing trapped rung downwards until it becomes untrapped. This is possible because below this leftward rung, there can only be rightward pointing trapped rungs, which do not obstruct the downward passage of the leftward rung. To see this, we distinguish three cases:
\begin{itemize} 
\item The rightward rung appears below the leftward rung, but they are adjacent to four different uprights. Then the leftward rung slides down past the other rung by a planar isotopy.
\item The rightward rung appears below the leftward rung and they share exactly one upright. Then the leftward rung can be moved down past the rightward rung by relation (3).
\item The rightward rung appears immediately below the leftward rung. Then we apply relation (4) to switch them, which creates summands in which the leftwards rung is below the corresponding rightward rung (this includes the possible case of label $0$ on a rung, in which case we think of it as erased).
\end{itemize}
 After all leftward trapped rungs have been moved downwards and outside the local picture, all rightwards (formerly) trapped rungs can be moved upwards by relation (3). Then, the chosen rungs adjacent to the leftmost upright are no longer separated and relation (4) can be applied, as desired.
 
Once the leftmost upright is disconnected, the algorithm is recursively applied to the rest of the diagram.  
\end{proof}

\begin{defi} We call an $\glnn{N}$ ladder web of \emph{type EF} if each leftward-oriented ladder rung appears under all rightward-oriented ladder rungs adjacent to the same two uprights. Analogously, we call a web an $\glnn{N}$ ladder web of \emph{type FE} if rightward rungs appear below leftward rungs in each \emph{column}.
\end{defi}

We need a generalization of the untrapping subroutine
\begin{lem} \label{lem_sortroutine}\textbf{The resorting subroutine.} Any $\glnn{N}$ ladder web can rewritten as a $\Z[q^{\pm 1}]$-linear combination of ladder webs of type EF (FE), using only the relations (2), (3) and (4).
\end{lem}
\begin{proof} We only describe an algorithm for resorting to EF webs. The other case is completely analogous; instead of the relations (4), it uses the relations (4'), which can be deduced from (4):
\begin{enumerate}
\item[(4')] Other square switch relations 
\[\xy
(0,0)*{
\begin{tikzpicture}[scale=.25]
	\draw [very thick, directed=.55] (-2,-4) to (-2,-2);
	\draw [very thick, directed=.55] (-2,-2) to (-2,2);
	\draw [very thick, directed=.55] (2,-4) to (2,-2);
	\draw [very thick, directed=.55] (2,-2) to (2,2);
	\draw [very thick, rdirected=.55] (-2,-2) to (2,-2);
	\draw [very thick, directed=.55] (-2,2) to (-2,4);
	\draw [very thick, directed=.55] (2,2) to (2,4);
	\draw [very thick, directed=.55] (-2,2) to (2,2);
	\node at (-2,-4.5) {\tiny $k$};
	\node at (2,-4.5) {\tiny $l$};
	\node at (0,-1.25) {\tiny $j_1$};
	\node at (0,2.75) {\tiny $j_2$};
\end{tikzpicture}
};
\endxy=\sum_{j^{\prime}\geq 0}{l-j_1-k+j_2 \brack j^{\prime}}\xy
(0,0)*{
\begin{tikzpicture}[scale=.25]
	\draw [very thick, directed=.55] (-2,-4) to (-2,-2);
	\draw [very thick, directed=.55] (-2,-2) to (-2,2);
	\draw [very thick, directed=.55] (2,-4) to (2,-2);
	\draw [very thick, directed=.55] (2,-2) to (2,2);
	\draw [very thick, directed=.55] (-2,-2) to (2,-2);
	\draw [very thick, directed=.55] (-2,2) to (-2,4);
	\draw [very thick, directed=.55] (2,2) to (2,4);
	\draw [very thick, rdirected=.55] (-2,2) to (2,2);
	\node at (-2,-4.5) {\tiny $k$};
	\node at (2,-4.5) {\tiny $l$};
	\node at (0,-1.25) {\tiny $j_2\! -\! j^{\prime}$};
	\node at (0,2.75) {\tiny $j_1\! -\! j^{\prime}$};
\end{tikzpicture}
};
\endxy\]
\end{enumerate}
 Also, we have the choice of starting from the leftmost or from the rightmost upright in the ladder web diagram; we describe the former.  Without further mention, we assume that the algorithm applies the relations (2) whenever possible.  

Suppose that adjacent to the leftmost upright, there is a leftward rung immediately above a rightward rung. Then the untrapping subroutine shows that all rungs that are trapped between our chosen rungs can be moved away, modulo relations applied further to the right in the web. The result is a configuration, in which our chosen rungs form a square that can be switched via relation (4). In doing so, the leftward rung moves down past the rightward rung. As long as there are leftward rungs above rightward rungs in the first column, this procedure can be iterated. Finally, it terminates in a linear combination of webs in which the rungs in the first column satisfy the EF condition. 

Now suppose that the rungs in the first $x-1$ columns on the left are already sorted to satisfy the EF condition. Suppose there is a leftward rung immediately above a rightward rung in the $x^{\mathrm{th}}$ column. By the untrapping routine we may assume that there are no trapped further to the right in the web. Since the web is already of type EF in the $(x-1)^{\mathrm{th}}$ column, there are also no trapped rungs on the left. Thus, relation (4) can be applied to move the leftward rung down past the rightward rung. This process does not change anything  to the left of the $x^{\mathrm{th}}$ column (modulo relation (3)) and it terminates after finitely many iterations in a linear combination of webs, in which the rungs in the first $x$ columns satisfy the EF condition. 
\end{proof}

\begin{rem} Via quantum skew Howe duality \cite{CKM}, the existence of a resorting subroutine follows from the triangular decomposition of quantum $\mathfrak{gl}_m$.
\end{rem}

In the following, we need another bigon relation which is satisfied for $\glnn{N}$ webs in $\R^2$:
\begin{equation}
\label{eqn-addbigon}
\xy
(0,0)*{
\begin{tikzpicture}[scale=.25]
	\draw [very thick, directed=.55] (-2,-4) to (-2,-2);
	\draw [very thick, directed=.55] (-2,-2) to (-2,2);
	\draw [very thick, directed=.55] (-2,2) to (-2,4);
	%%rungs
	\draw [very thick, directed=.35] (-2,2) to [out=30,in=90] (2,0) to [out=270,in=330] (-2,-2);
	\node at (-2,-4.5) {\tiny $k$};
	\node at (1,0) {\tiny $j$};
\end{tikzpicture}
};
\endxy \;\; = \;\; {N-k \brack j}
\xy
(0,0)*{
\begin{tikzpicture}[scale=.25]
	\draw [very thick] (-2,-4) to (-2,-2);
	\draw [very thick, directed=.55] (-2,-2) to (-2,2);
	\draw [very thick] (-2,2) to (-2,4);
	\node at (-2,-4.5) {\tiny $k$};
\end{tikzpicture}
};
\endxy
\;\; = \;\;
\xy
(0,0)*{
\begin{tikzpicture}[scale=.25]
	\draw [very thick, directed=.55] (2,-4) to (2,-2);
	\draw [very thick, directed=.55] (2,-2) to (2,2);
	\draw [very thick, directed=.55] (2,2) to (2,4);
	%%rungs
	\draw [very thick, directed=.35] (2,2) to [out=150,in=90] (-2,0) to [out=270,in=210] (2,-2);
	\node at (2,-4.5) {\tiny $k$};
	\node at (-1,0) {\tiny $j$};
\end{tikzpicture}
};
\endxy
\end{equation}

\begin{prop} \label{prop-annulareval}
Let $W$ be a balanced $\glnn{N}$ ladder web with a marked boundary point of label $k$ and consider the partial closure $\overline{W}$ of all unmarked boundary strands:
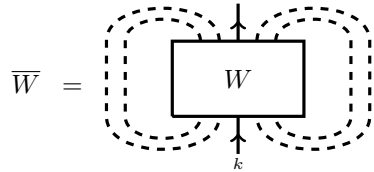
\[\overline{W} \;\; = \;\; \xy
(0,0)*{
\begin{tikzpicture}[scale=.25]
	\draw [very thick] (-3.5,-2) to (-3.5,2) to (3.5,2) to (3.5,-2) to (-3.5,-2);
	\draw [very thick,directed=.55] (0,2) to (0,4);
	\draw [very thick,directed=.55] (0,-4) to (0,-2);
	\draw [very thick,dashed] (-1,2) to [out=90,in=90] (-7,2) to (-7,-2) to [out=270,in=270] (-1,-2);
	\draw [very thick,dashed] (-2,2) to [out=90,in=90] (-6,2) to (-6,-2) to [out=270,in=270] (-2,-2);
	\draw [very thick,dashed]  (2,2) to [out=90,in=90] (6,2) to (6,-2) to [out=270,in=270] (2,-2);
	\draw [very thick,dashed]  (1,2) to [out=90,in=90] (7,2) to (7,-2) to [out=270,in=270] (1,-2);
	\node at (0,0) {$W$};
	\node at (0,-4.5) {\tiny $k$};
\end{tikzpicture}
};
\endxy 
\]
Then $\overline{W}$ can be rewritten as a $\Z[q^{\pm 1}]$-linear combination of the $k$-labelled web edge using only the four relations from Lemma~\ref{lem-annulareval} and the additional relation~\eqref{eqn-addbigon}. The latter is used in the following form (and its horizontally reflected version):
\begin{enumerate}
\item[(5)] Partial closure relations
\[\xy
(0,0)*{
\begin{tikzpicture}[scale=.25]
	\draw [dotted ,thick] (-2,-5) to (-2,5);
	\draw [very thick, directed=.55] (2,-3) to (2,3);
	\draw [very thick, directed=.55] (-2,-3) to (2,-3);
	\draw [very thick, directed=.55] (-2,-1.5) to (2,-1.5);
	\draw [very thick, rdirected=.55] (-2,1.5) to (2,1.5);
	\draw [very thick, rdirected=.55] (-2,3) to (2,3);
	\draw [very thick,dashed]  (2,3) to (2,4) to [out=90,in=90] (6,4) to (6,-4) to [out=270,in=270] (2,-4) to (2,-3);
	\node at (2,-5.5) {\tiny $l$};
	\node at (0,-2.25) {\tiny $\dots$};
	\node at (0,2.25) {\tiny $\dots$};
	\node at (0,-3.75) {\tiny $j_1$};
	\node at (0,-0.75) {\tiny $j_k$};
\end{tikzpicture}
};
\endxy
\;\; = \;\; {N-\sum j_i \brack l}
\xy
(0,0)*{
\begin{tikzpicture}[scale=.25]
	\draw [dotted ,thick] (-2,-5) to (-2,5);
	\draw [very thick, directed=.55] (2,-3) to (2,3);
	\draw [very thick, directed=.55] (-2,-3) to (2,-3);
	\draw [very thick, directed=.55] (-2,-1.5) to (2,-1.5);
	\draw [very thick, rdirected=.55] (-2,1.5) to (2,1.5);
	\draw [very thick, rdirected=.55] (-2,3) to (2,3);
	\node at (0,-2.25) {\tiny $\dots$};
	\node at (0,2.25) {\tiny $\dots$};
	\node at (0,-3.75) {\tiny $j_1$};
	\node at (0,-0.75) {\tiny $j_k$};
\end{tikzpicture}
};
\endxy
\]
\end{enumerate}
\end{prop}

\begin{proof} We describe an explicit algorithm for the case that $W$ has closed strands on both sides of the marked upright; the other cases are easier. First we apply the resorting subroutine from Lemma~\ref{lem_sortroutine} to replace $W$ by a linear combination of webs of type FE. Then the relation (5) can be applied to the rightmost upright to remove its closure. We interpret the resulting webs as ladder webs whose rightmost upright has label $0$ at the bottom and the top. 

In the next step, we resort all webs to EF type. This has the effect of removing any interaction with the rightmost upright and reduces the problem to ladder webs with one fewer upright. Also, having resorted, we can now use the reflected version of relation (5) to remove the closure of the leftmost upright. After finitely many steps of resorting between types FE and EF, and removing closures via the relations (5), the algorithm terminates in a multiple of the marked edge.
\end{proof}

\bibliographystyle{plain}
\bibliography{refs}

\end{document}